\documentclass[12pt,reqno]{amsart}

\usepackage[leqno]{amsmath}
\usepackage{amssymb,amsfonts,xfrac,MnSymbol}
\usepackage{amsthm}
\usepackage{cite}
\usepackage{hyperref}
\usepackage{todonotes}
\usepackage{enumitem}
\usepackage{graphicx}
\usepackage{tikz-cd, tikz}
\usepackage{color}
\usepackage{import}
\usepackage[all]{xy}
\usepackage[margin=1.2in]{geometry}
\usepackage[perc]{overpic}
\usepackage{subfigure}
\usepackage[all]{xy}
\usepackage{caption}

\numberwithin{equation}{section}

\makeatletter
\def\thm@space@setup{%
  \thm@preskip=12pt plus 4pt minus 6pt
  \thm@postskip=15pt
}
\makeatother

\newtheorem{theorem}{Theorem}[section]
\newtheorem*{theorem*}{Theorem}
\newtheorem{proposition}[theorem]{Proposition}
\newtheorem*{proposition*}{Proposition}
\newtheorem{lemma}[theorem]{Lemma}
\newtheorem*{lemma*}{Lemma}
\newtheorem{claim}[theorem]{Claim}
\newtheorem*{claim*}{Claim}

\newtheorem*{subclaim*}{Sub-claim}
\newtheorem{corollary}[theorem]{Corollary}
\newtheorem*{corollary*}{Corollary}

\theoremstyle{definition}
\newtheorem{definition}[theorem]{Definition}
\newtheorem{observation}[theorem]{Observation}
\newtheorem{remark}[theorem]{Remark}

\newtheorem{conjecture}[theorem]{Conjecture}

\newtheorem*{definition*}{Definition}
\newtheorem*{observation*}{Observation}
\newtheorem*{remark*}{Remark}
\newtheorem*{example*}{Example}
\newtheorem*{question*}{Question}
\newtheorem*{exercise*}{Exercise}
\newtheorem*{fact*}{Fact}
\newtheorem*{notation*}{Notation}

\newtheorem*{conjecture*}{Corollary}

\newcommand{\bbD}{\mathbb{D}}

\newcommand{\bbN}{\mathbb{N}}

\newcommand{\calB}{\mathcal{B}}
\newcommand{\calC}{\mathcal{C}}
\newcommand{\calD}{\mathcal{D}}

\newcommand{\calK}{\mathcal{K}}
\newcommand{\calL}{\mathcal{L}}

\newcommand{\calN}{\mathcal{N}}
\newcommand{\NN}{\mathcal{N}}

\newcommand{\calT}{\mathcal{T}}

\newcommand{\interior}[1]{\text{int}(#1)}

\newcommand{\homeo}{\cong}

\newcommand{\ii}{^{-1}}

\newcommand{\tild}[1]{\widetilde{#1}}
\newcommand{\ssm}{\smallsetminus}
\newcommand{\abs}[1]{\left| #1 \right|}

\newcommand{\bbl}[1]{\#_b(#1)}
\newcommand{\bdr}[1]{\#_I(#1)}
\newcommand{\bbb}[1]{\#(#1)}

\newcommand{\pos}{>0}
\newcommand{\np}{\le 0}
\newcommand{\negative}{< 0}

\title{Highly twisted plat diagrams}
\author{Nir Lazarovich}
\author{Yoav Moriah}

\author{Tali Pinsky}
\thanks{NL was supported by the Israel Science Foundation (grant no. 1562/19), and by the German-Israeli Foundation for Scientific Research and Development.}

\date{}



\address{Department of Mathematics\\
Technion\\
Haifa, 32000 Israel}
\email{lazarovich@technion.ac.il}

\address{Department of Mathematics\\
Technion\\
Haifa, 32000 Israel}
\email{ymoriah@technion.ac.il}

\address{Department of Mathematics\\
Technion\\
Haifa, 32000 Israel}
\email{talipi@technion.ac.il}

\begin{document}

\maketitle

\begin{abstract} We prove that the knots and links in the infinite set of $3$-highly 
twisted $2m$-plats,with $m \geq 2$, are all hyperbolic.  This should be compared with 
a result of Futer-Purcell for  $6$-highly twisted diagrams. While their proof uses 
geometric methods our  proof is achieved by  showing that the  complements of such 
knots or links are unannular and atoroidal. This is  done by using a new approach 
involving an Euler  characteristic argument. 

\end{abstract}

\section{introduction}\label{sec: introduction}

The prevailing feeling among low dimensional topologists is that ``most" links $\mathcal{L}$ in 
$S^3$ are hyperbolic. That means that the open manifold $S^3 \ssm \mathcal{L}$ can be endowed with
a complete hyperbolic metric of sectional curvature $-1$. Being hyperbolic is a property of the manifold 
with far reaching consequences. However, proving that a specific
link $\mathcal{L}$ is hyperbolic turns out to be not trivial. This is especially true if 
the link $\mathcal{L}$ is ``heavy duty", i.e., has a very large crossing number. See for example
\cite{minsky-moriah:surplus}.

The question of when can one decide if the complement of a  link in $S^3$ is a hyperbolic manifold 
from its projection diagram  has been of interest for a long time. Just to give three examples: 
The first result  in this direction is by Hatcher and Thurston who proved that  complements of 
$2$-bridge knots which have at least two twist regions (they are not torus knots or links) are 
hyperbolic, see \cite{Hatcher-Thurston}. The second is Menasco's result \cite{Menasco} that
a non-split prime alternating link which is not a torus link is hyperbolic.  Later  Futer 
and Purcell proved in  \cite{FuterPurcell}, among other results, that every link with a  
$6$-highly twisted irreducible diagram and  which has at least  two twist regions  is hyperbolic.  
Their result is  obtained by applying Marc 
Lackenby's $6$-surgery theorem, see \cite{lackenby2000word}, to the corresponding fully 
augmented links. Our main theorem is:

\begin{theorem}\label{thm: highly twisted plats are hyperbolic}
Let $L$ be a $3$-highly twisted $2m$-plat, $m \geq 2,$ with at least three twist regions, 
then $L$ is  hyperbolic.
\end{theorem}

Every link $\mathcal{L}$ in $S^3$ has a plat projection \cite{BuZi}. 
Assume that $\mathcal{L}$ has a $2m$-plat 
projection $D(\mathcal{L})=D_P(\mathcal{L})$ in some plane $P$, for some $m \geq 2$.
Note that every $2m$-plat projection defines a knot or link projection with 
$m\geq b(\mathcal{L})$ bridges, and every $m$-bridge knot or link has a 
$2m$-plat projection (see \cite[p.~24]{BuZi}).  Here,  $b(\mathcal{L})$ 
denotes the bridge number of $\mathcal{L}$  as defined in the next 
section. It follows from Theorem \ref{thm: highly twisted plats are hyperbolic} 
that not all links have a  $3$-highly twisted plat diagram. 
The subset of links that do is a ``large'' subset in a sense that can 
be made precise, see the discussion in \cite{lustig2012large}. In  
Theorem \ref{thm: highly twisted plats are hyperbolic}
we weaken the conditions imposed in \cite{FuterPurcell} on $\mathcal{L}$  
from $6$-highly twisted to $3$-highly twisted, at the price of requiring 
that the diagram of $\calL$ be a plat. It seems that the techniques 
developed here using the Euler characteristic, with some additional work,  might be 
adequate for more general diagrams. Thus, we would like to make the following conjecture:

\begin{conjecture}\label{con: just highlt twisted}
Let $L$ be a link in $S^3$ with a link diagram which is prime, twist-reduced, $3$-highly twisted and 
has least two twist regions,  then the link  $L$ is  hyperbolic.
\end{conjecture}

\vskip15pt

\section{preliminaries}\label{sec: preliminaries}

\subsection{Bubbles and twist regions.}\label{subsec: bubbles and twist regions}

\vskip10pt

Given a projection of a link $\mathcal{L}$ onto a plane $P$, surround each crossing in 
the projection diagram by a small $3$-ball $B$. Denote the collection of these $3$-balls
by $\mathcal{B}$. Then $\mathcal{L}$ is isotopic to a link $L$ that is embedded 
in $P \cup{\partial{\mathcal{B}}}$. For a  single crossing we refer to $\partial B$ 
as a \emph{bubble}. Note that $P$ divides each  bubble into two hemispheres denoted 
by $\partial B^+$ and $\partial B^-$. Denote the union of all the $\partial B^\pm$ 
by $\mathcal{B}^\pm$ respectively. Denote the two disjoint $2$-spheres 
$(P \ssm \mathcal{B}) \cup \mathcal{B}^\pm$ by $P^\pm$ respectively. Each of $P^\pm$ bounds 
a $3$-ball $H^\pm$ in $S^3\ssm L$.

A \emph{twist region} $T$ in $L$ is  a ``cube" $D\times [-\varepsilon,+\varepsilon]$ where $D$ 
is a maximal disk in $P$ so that  and $(T,T\cap L)$ is a  trivial integer 2-tangle. 
For example, in Figure~\ref{fig:plat}, a box labeled $a_{i,j}$ indicates a
\emph{twist region} with $a_{i,j}$ crossings, where $a_{i,j}$ can be
positive, negative, or zero. In the example in the figure, $a_{1,1}=a_{2,2}=-3$, 
and all other $a_{i,j}=-4$.  

\begin{figure}
  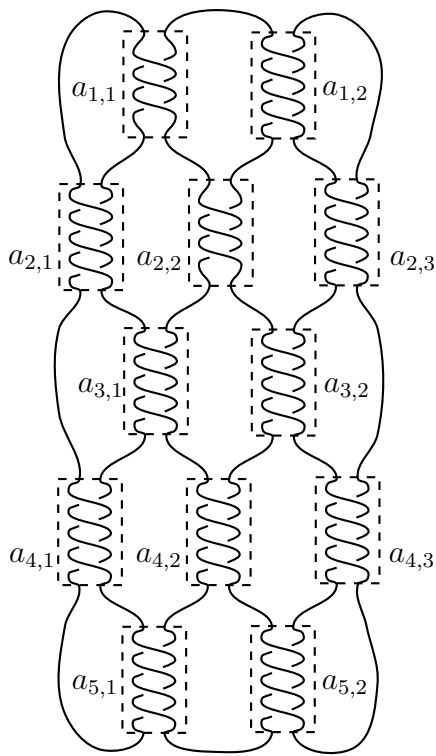
\caption{A 6-plat projection of a $3$-bridge knot. }
\label{fig:plat}
\end{figure}

\subsection{Plats.}
Let $b$ be an element in the braid group $\mathcal{B}_{2m}$ on $2m-1$ generators 
$\{\sigma_1,\dots,\sigma_{2m-1}\}$. We require that $b$ is written 
as a concatenation of sub-words as follows:
$$b=b_1\cdot b_2 \cdot .... \cdot b_{n-1},$$

where $n$ is odd, and where $b_i$ has the following properties:

\begin{enumerate}      
\item When $i$ is odd, $b_i$ is a product of all $\sigma_j$ with $j$
  even. Namely:

$$b_i = \sigma_2^{a_{i, 1}} \cdot \sigma_4^{a_{i, 2}} \cdot \dots \cdot
\sigma_{2m-2}^{a_{i, m-1}}$$

\medskip

\item When $i$ is even, $b_i$ is a product of all $\sigma_j$ with $j$
  odd. Namely:

$$b_i = \sigma_1^{a_{i, 1}} \cdot \sigma_3^{a_{i, 2}} \cdot \dots \cdot
\sigma_{2m - 1}^{a_{i, m}}$$
\end{enumerate}

Consider the geometric braid on $2m$ strings corresponding to the element $b$. 
At the top of $b_1$, connect each pair of strands $\{1, 2\}, \{3, 4\}, \dots \{2m - 1, 2m\}$,
(ordered from left to right) by a small unknotted arc. Similarly connect the same pairs of strands 
at the bottom of $b_n$.

The obtained knot or link is a {\it $2m$-plat}.  The number  $m$  is the \emph{width} of the plat 
and $n$, is the \emph{length} of the plat. The braid $b$
will be called  the  {\it underlying braid} of the plat.

For any knot or link $\mathcal{L}\subset S^3$ there is an $m\in\bbN$ so that $\mathcal{L}$ 
has a $2m$-plat projection, on some projection plane $P$ as indicated in Figure~\ref{fig:plat}. 
This follows from the fact that any knot or link can be presented as a closed braid (proved by  
Alexander  in 1923  see \cite[p.~23]{BuZi})
and then strands of the braid  closure can be pulled across the braid diagram; for example see
\cite[p.~24]{BuZi}. The number $m$ is by no means unique. If $m = br(\mathcal{L})$, where
$br(\mathcal{L})$ is the bridge number of $\mathcal{L}$ (see Schubert \cite{Schubert}), 
then  the $2m$-plat is a \emph{minimal plat} for $\mathcal{L}$.

Given a knot or link in a $2m$-plat projection then the twist regions  corresponding to the 
$\sigma_{i, j}^{a_{i, j}}$ are called \emph{twist boxes}. The twist boxes are arranged in a  
configuration which is ``almost'' a matrix:  There are $n \in \mathbb{N}$ rows indexed by 
$i \in \{1, \dots, n\}$. For odd $i$'s there are $m-1$ columns and rows with an even $i$
have $m$ columns.  Denote the crossing number in each twist box  by $t_{i,j} = a_{i,j}$.

\begin{definition}
A $2m$-plat will be called \emph{$c$-highly twisted} if
$\abs{t_{i,j}}\geq c$ for some constant $c$, for all $i,j$.
Similarly, a knot or link that admits a $c$-highly twisted
plat projection will be called a \emph{$c$-highly twisted knot or link}. 
\end{definition}

\subsection{Diagram regions}\label{subsec: diagram regions} 
Every knot of link diagram $D(L)$ in a projection plane $P \subset S^3$ defines a planar $4$-regular 
graph. The complementary surfaces of $P \ssm D(L)$ can be colored black and white so that two 
surfaces adjacent to an edge are colored in different colors. 

A  $2m$-plat diagram $D(L)$ determines a template diagram $\mathcal T$ in $P$ where each 
twist region is replaced by a rectangle. Thinking of these rectangles as vertices of valency
$4$ in the graph determined by $\mathcal{T}$, one can color the complementary regions of 
the graph in a black and white checkerboard  manner. 

\begin{definition}[Lackenby \cite{lackenby2004volume}]\label{def: twist reduced}
A link diagram $D(L)$ is \emph{prime} if any simple closed curve in $P$ intersecting $D(L)$ 
in two points bounds a subdiagram with no crossings.

A link diagram  $D(L)$ is  \emph{twist-reduced} if any simple closed curve in $P$ which 
intersects the edges of $\mathcal T$ transversally in four points composed of two pairs 
each of which is adjacent to a crossing of $D(L)$, bounds a subdiagram which is the 
diagram of an integer $2$-tangle.
\end{definition}

\begin{remark}\label{rem: Plats are twist reduced} Let $L$ be a $3$-highly twisted $2m$-plat with $m \geq 3$.
Then one can observe directly that the corresponding diagram $D(L)$ is  prime and twist-reduced.
\end{remark}

\begin{remark}\label{rem: Intersect or bubble}
An arc in $P$ connecting two regions of different colors must cross an edge of the graph 
or go through a rectangle. Hence an arc on $P^\pm$ in the $2m$-plat diagram connecting
two regions of different colors must intersect $L$ or meet a bubble (by which we mean that 
the arc intersects $\partial B$  in an arc for some bubble $B$).
\end{remark}

\begin{definition}\label{def: Distance}
Given two regions in $ A, B \in P \ssm \mathcal{T}$ define the distance $d(A, B)$ between 
$A$ and $B$ to be the minimal number of color changes over all arcs between $A$ and $B$. 
Similarly, if $a,b$ are points in regions $A,B$ respectively, then define the distance 
$d(a,b)$ to be the distance $d(A,B)$.
\end{definition}

\begin{definition}\label{def: regions}
The regions in $P\ssm \mathcal{T}$  are composed of quadrilaterals, triangles, four bigons and
a single unbounded region. 
The bigon regions are called the \emph{corners} of the template $\mathcal T$.
\end{definition}

\begin{definition}
We say that an arc of $L$ connects two regions if its endpoints are contained in the closure of the 
two regions respectively (note that the same arc can connect different pairs of regions).
\end{definition} 

\begin{definition}\label{def:BridgeNumber}
Given a knot or link projection with the over-crossings and under-crossings indicated,
a \emph {bridge}\footnote{Note that the standard definition of a bridge requires the 
arc to be maximal and contain at least one over-crossing.} does not contain any under-crossings, 
i.e., it is a subarc of $L\cap P^+$. Note that in a plat, a bridge can pass over at most two crossings.
\end{definition}

\begin{observation}\label{obs: single arc}
If $L$ is 2-highly twisted, and  $\alpha$ a bridge which connects two regions 
of the diagram $D(L)$ then: 
\begin{enumerate}
    \item \label{obs: single arc. two bubbles}  If $\alpha$ passes over two crossings, 
    then $\alpha$ is the unique bridge connecting the same regions that 
    passes over two crossings. 
    \item \label{obs: single arc. two bubbles and distance 2}If $\alpha$ passes over 
    two crossings and the two regions that are connected by $\alpha$ are at distance 
    at least 2, then $\alpha$ is the unique bridge connecting them.
\end{enumerate}
Note that the only case of an arc $\alpha$ which passes over two crossings and connects regions 
of distance less than $2$ occurs on the boundaries of the corner bigon regions of the diagram. 
In this case, there is a second arc which does not pass over any crossings which connects the 
same regions.
\end{observation}

\vskip20pt

\section{Surfaces in link complements.}
\vskip10pt
\subsection{Normal position.}\label{subsec: normal position} 
We are interested in studying compact surfaces $S$ properly embedded in 
$S^3 \ssm \NN(L)$. If $\partial S \neq \emptyset$ we extend $S$ by shrinking the neighborhood 
$\NN(L)$ radially. This determines  a map $i:S\to S^3$, whose image we denote by $S$ as well, 
which is an embedding on  $S\ssm\partial S$ and $i(\partial S)\subseteq L$.

\begin{lemma}\label{lem: normal form} 
Let $S\subset S^3 \ssm \NN(L)$ be a proper surface with no meridional boundary components, and 
let $(T,t)$ be a twist region. Then,  up to isotopy, and $S \cap T$ is a disjoint  union of disks  
$D \subset (T,t)$ of one of the 
following three types:
\begin{enumerate}
    \item[ \underline{Type 0}:] $D$ separates the two strings of $t$.
    \vskip7pt
    \item[ \underline{Type 1}:] $\partial D$ decomposes as the union of two arcs $\alpha \cup \beta$ 
    such that $\alpha\subset t$ and $\beta\subset \partial T$.
    \vskip7pt
    \item[ \underline{Type 2}:] $\partial D$ decomposes as the union of four arcs 
    $\alpha_1 \cup \beta_1 \cup \alpha_2 \cup \beta_2$ where $\alpha_i \subset t_i$ and 
    $\beta_i \subset \partial T$.
\end{enumerate}
Moreover, the isotopy decreases the number of bubbles that $S$ meets, and if no component of 
$\partial S$ is a meridian, then we may further assume that $i|_{\partial L}:\partial S \to L$ 
is a covering map. 
\end{lemma}

\begin{proof}
If no component of $\partial S$ is a meridian, we may assume that up to isotopy 
$i|_{\partial L}:\partial S \to L$ is a covering map. 

The twist region $(T,t)$ is a trivial 2-tangle. The complement $T\ssm \NN(t)$ can be 
identified with $P\times [0,1]$ where $P$ is a twice holed disk.  Let $E$ be the disk
$\alpha\times [0,1]$ where $\alpha$ is the simple arc connecting the two holes of $P$.
Up to a small isotopy, we may assume that $S$ intersects $E$ transversely. 
Since the bubbles in $T$ are in some neighborhood of $E$, we may assume that $S$ meets 
a bubble if it does so in $E$. The intersection $S\cap E$ comprises of simple closed 
curves and arcs. All curves and arcs except those connecting $\alpha\times \{0\}$ to 
$\alpha\times \{1\}$ can be eliminated by an isotopy pushing $S$ off $T$. 
This isotopy decreases the number of bubbles $S$ meets. The number of bubbles the 
resulting surface meets equals the number of such arcs times the number of twist 
in the twist box.

Up to isotopy, we may also assume that $S$ intersects $P\times \{ \tfrac{1}{2}\}$ transversely. 
Hence, $S\cap (P\times \{ \tfrac{1}{2}\})$ is a collection of simple closed curves and arcs. 
By pushing $S$ outwards towards the boundary of the disk $P$, one can assume that each 
component of $S\cap (P\times \{ \tfrac{1}{2}\})$ is of the following form:
\begin{enumerate} \setcounter{enumi}{-1}
    \item An arc connecting the boundary of the disk $P$ to itself separating the holes, 
    and intersecting $\alpha$ once.
    \item An arc connecting a hole to the boundary of the disk and not intersecting $\alpha$.
    \item An arc connecting the two holes and not intersecting $\alpha$.
\end{enumerate}
Thus, $S\cap (P\times [\tfrac{1}{2}-\varepsilon,\tfrac{1}{2}+\varepsilon])$ is a collection 
of disks of types (0),(1) or (2) as stated. By an ambient isotopy, we can stretch the slab 
$P\times [\tfrac{1}{2}-\varepsilon,\tfrac{1}{2}+\varepsilon]$ to $P\times [0,1]=T$.
The number of bubbles the resulting surface meets equals the number of arcs of type (0) times 
the number of twist in the twist box. The arcs of type (0) are in one-to-one correspondence 
with the arcs of $S\cap E$. 

The fact that $i:\partial S \to L$ is a covering map was not affected by the isotopies above.
\end{proof}

\begin{definition}\label{def: normal position}
A surface $S\subset S^3 \ssm \NN(L)$ is \emph{in normal position} if its extension intersects 
each twist region  as specified in Lemma \ref{lem: normal form} and $i:\partial S \to L$ is 
a covering map. In particular, $S$ has no meridional boundary components.
\end{definition}

\begin{corollary}
Let $S \subset S^3 \ssm \calN (L)$ be a proper  incompressible surface, and let $(T,t)$ be a 
twist region. Then, up to isotopy, each component of the intersection $S \cap T \cap P^\pm$ 
looks as in Figure \ref{fig: three types of intersection}.
\end{corollary}

\begin{figure}
   \centering
    \begin{overpic}[height=4cm]{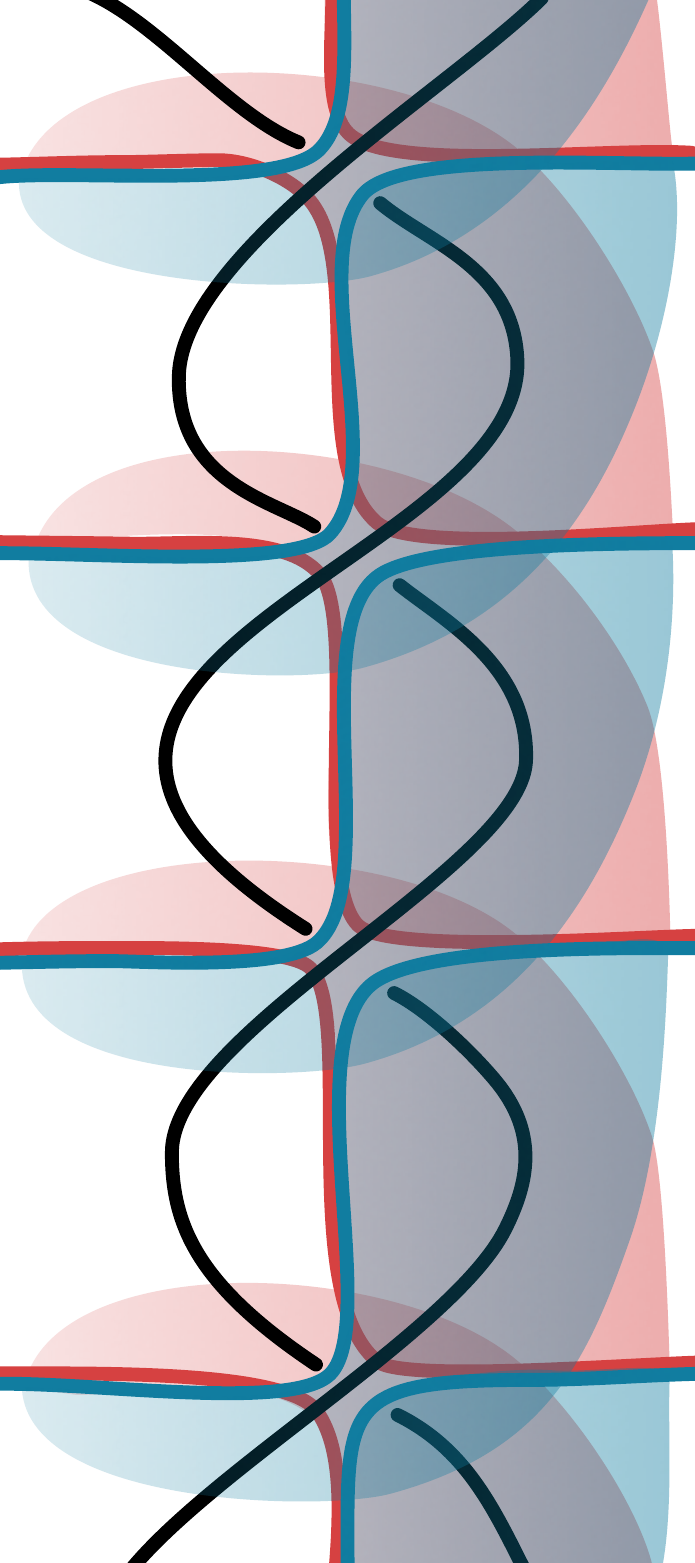}
    \put(5,-10){Type 0}
    \end{overpic}
        \hskip 1cm
    \begin{overpic}[height=4cm]{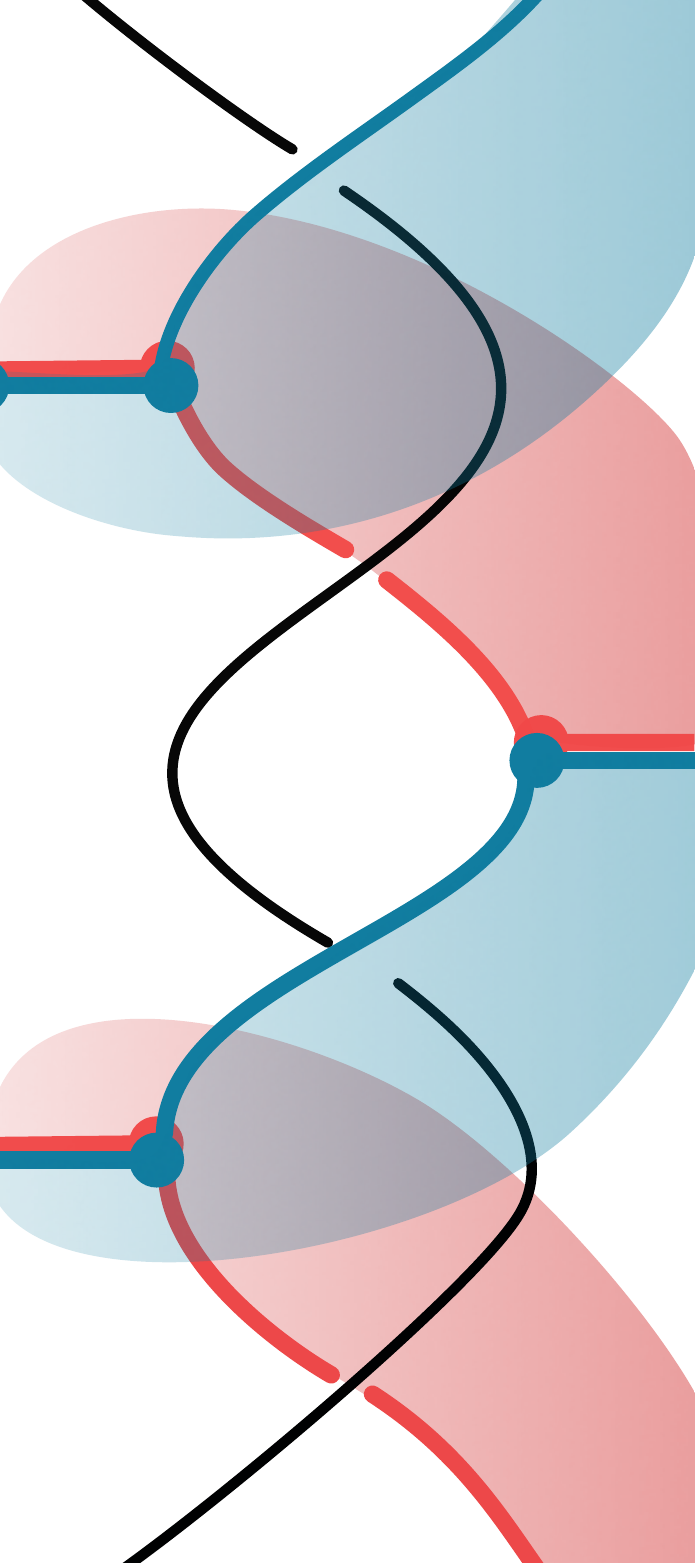}
    \put(5,-10){Type 1}
    \end{overpic}
        \hskip 1cm
    \begin{overpic}[height=4cm]{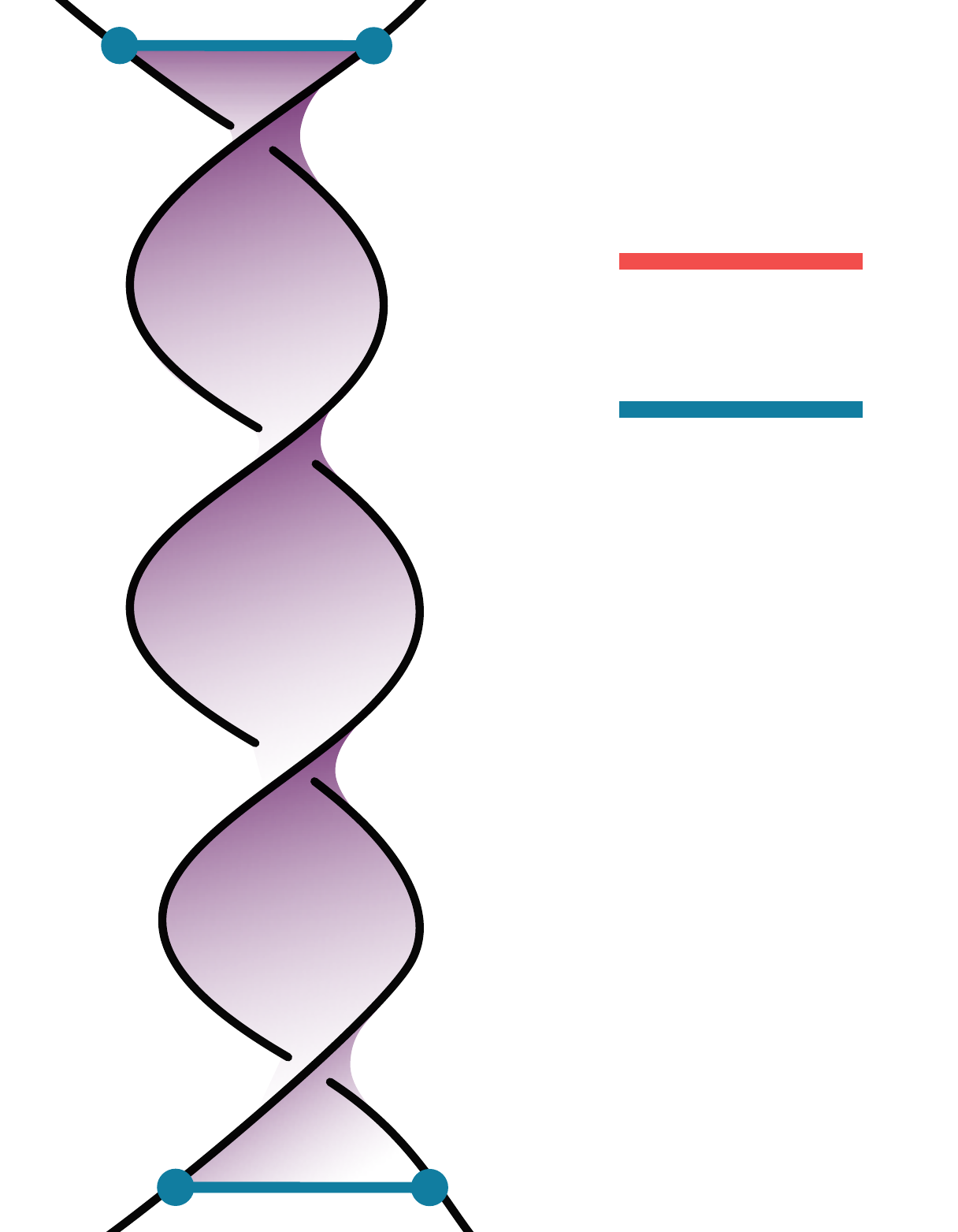}
    \put(5,-10){Type 2}
    \put(75,76){$S\cap P^-$}
    \put(75,63){$S\cap P^+$}
    \end{overpic}
    \medskip
    \caption{The possible three types of intersection of $S$ with a twist box.}
    \label{fig: three types of intersection}
\end{figure}

\subsection{Curves of intersection.}\label{subsec: Curves of intersection}
Let $S \subseteq S^3 \ssm \NN(L)$ be a surface in normal position. 
We would like to study the surface $S$ through its curves of intersection with the planes $P^\pm$. 
However, disks of Type (2) cause some technical complications. In order to simplify the situation 
we consider the surface $\widehat{S}$ obtained by removing those disks from $S$. 
Explicitly: Let $\calT$ be the union of all twist boxes of the plat $L$. 
Consider the collection of disks $\calD$ of Type (2) which occur as 
intersections $S\cap \calT$. We may assume that $\partial \calD \subset P \cup L$, and that the 
subsurface $\widehat{S}=S \ssm \calD$ is transversal to $P^\pm$.

Define $\calC^+ = \partial i\ii(\widehat{S} \cap H^+)$ and $\calC^- =  \partial  i\ii(\widehat{S} \cap H^-)$.
Now define $\calC =\calC^+ \cup \calC^-$.  As each of  $P^\pm$ is a $2$-sphere, $\widehat{S}\cap H^\pm$ is a 
collection of subsurfaces of $\widehat{S}$, the boundary of which are simple closed curves $c \subset S$. 
For $c\in \calC^+$, denote by $S_c$ the component of $\widehat{S} \cap H^+$ so that $c\subset \partial S_c$, 
and respectively for $c\in \calC^-$. Although formally defined as the preimage under $i$, we think of curves 
in $\calC^\pm$ as curves on $P^\pm$, as such they are disjoint outside $L$.

\begin{definition}\label{def: c passes}\label{def: numbers}
For $c\in \calC$,
\begin{enumerate}
    \item We say that a curve $c \in \calC$ \emph{passes through a bubble} $B$ if 
$c\cap (\partial B \ssm L)  \neq \emptyset$. 
    \item An \emph{intersection point} on $c\in \calC$ is an endpoint of the intervals in $c\ssm L$. 
    \item Denote by $\bbl{c}$ the number of bubbles (with multiplicities) through which $c$ passes.
    \item Denote by $\bdr{c}$ the number of intersection points along $c$.
    \item Define $\bbb{c}=\bbl{c}+\bdr{c}$.
    \item Define $\calC_{i,j} = \{ c\in\calC \mid \bbl{c}=i,\bdr{c}=j\}$.
\end{enumerate}
\end{definition}

\subsection{Taut surfaces}\label{subsec: taut surfaces}
\begin{definition}\label{def: complexity} Given an incompressible surface $S \subset S^3 \ssm \NN(L)$
we define a { \it lexicographic complexity} of $S$ as follows: 
\begin{equation}
\text{Com}(S) = (\abs{\mathcal{C}},\, \sum_{c\in\calC} \bbl{c},\, 
\sum_{c\in\calC} \bdr{c})
\end{equation}
\end{definition}

Recall that a  properly embedded surface $S$ in a $3$-manifold $M$ is called \emph{essential} if 
it is either  a 2-sphere which does not bound a 3-ball, or it is incompressible, boundary incompressible 
and not boundary parallel.

\begin{lemma}\label{lem: properties of curves}
Let $S\subset S^3 \ssm \NN(L)$ be an essential surface in normal position. Assume that either
\begin{enumerate}[label=(\roman*)]
    \item $S$ is an essential 2-sphere, and minimizes complexity among all essential 2-spheres, or
    \item $S$ is not a 2-sphere, the link $L$ is not split (i.e., $S^3 \ssm L$ is irreducible), and 
    $S$ minimizes complexity in its isotopy class.
\end{enumerate}
Then, for all $c \in \calC$ we have:
\vskip5pt
\begin{enumerate}
    \item $S_c \homeo \bbD^2$.
    \vskip5pt
    \item $\bdr{c}$ is even.
    \vskip5pt
    \item If $\bdr{c}\le 2$ then $\bbl{c} >0$.
    \vskip5pt
   \item  If $\bdr{c}=0$ then $\bbl{c}$ is even.
   \vskip5pt
   \item The curve $c$ does not pass twice through the same bubble on the same side of $L$.
   \vskip5pt
   \item The curve $c$ does not contain an arc, as depicted in Figure \ref{reducing bubbles a}, 
   that passes through exactly one bubble and has two intersection points with an edge of 
   $L$ that emanates from the  bubble on both sides of the edge.
   \vskip5pt
   \item The curve $c$ does not contain an arc, as depicted in Figure \ref{reducing bubbles b}, 
   that is contained in a twist box, and passes through exactly one bubble and one intersection point.
      \vskip5pt
  \item There is no arc of $c\cap L$ such that a small extension of the arc along $c$ has endpoints 
  in the same region.
\end{enumerate}
\end{lemma}

\begin{proof} Let $S\subset S^3 \ssm \NN(L)$ be an essential surface satisfying (i) or (ii). Note that 
in both cases, compressing along a disk $D\subset S^3 \ssm \NN(L)$ with $D\cap S = \partial D$ results 
either in two essential spheres, or a surface in the same isotopy class of $S$. Thus, by the assumption 
on $S$, surfaces obtained by such a compression cannot have lower complexity.
\vskip5pt

\noindent (1) Since $S$ is essential, each subsurface $S_c$ must be planar, as otherwise it 
contains a non-trivial compression disk. If $S_c$ has more than one boundary component then 
compressing along a disk in $H^+$ or $H^-$ whose boundary separates boundary components of $S_c$   
will result in a surface with fewer  intersections  with $P$ in contradiction to the choice of $S$.
\vskip5pt

\noindent (2) By definition, $\bdr{c}$ is the number of endpoints of arcs in $c\ssm L$. 
Since each arc has two endpoints, $\bdr{c}$ is even.
\vskip5pt

\noindent (3) By (2) $\bdr{c}$ is either two or zero. If $\bdr{c} = 0$ and $\bbl{c} = 0$ then $c$ bounds
a disk on $P \ssm L$. Compressing $S$ along this disk reduces the number of intersections with $P$. 

If $\bdr{c} = 2$ and $\bbl{c} = 0$ then $c$ bounds a disk $D$ in $P$ such that 
$\partial D = \alpha \cup \beta$,  where $\alpha$ is an arc in $L$ and $\beta$ 
is an arc in  $P$. Since both $\alpha$ and $\beta$ do not pass through bubbles 
$\partial D$ bounds a disk in $P\ssm L$. By choosing an innermost such $D$ we 
may assume that $D \subset P \ssm (L\cup S)$ and thus $D$ is a boundary compression 
disk. Compressing $S$ along $D$ we get an  isotopic surface with less intersections 
with $P$ which is a contradiction.

\vskip5pt

\noindent (4) Consider the colors of complementary regions of $P \ssm \mathcal{T}$  which 
the curve $c$ intersects.  If $\bdr{c}=0$ every change of colors, of these regions along $c$, 
accounts for one bubble that $c$ meets. Since $c$ is a closed curve the total number of 
color changes is even, and correspondingly $\bbl{c}$  is even.

\vskip5pt

\noindent (5) This claim follows directly from Lemma 1(ii) of \cite{Menasco}.

\vskip5pt

\noindent (6) In this case there is an  ambient isotopy of the surface $S$ which pushes
the disk bounded by the curve $c$ through the bubble thus reducing the number of  bubbles 
met by $S$ by one. The isotopy is indicated by the arrow in Figure \ref{reducing bubbles a}. 
The surface $S$ is assumed to be in normal form,  by Definition \ref{def: normal position}. 
This contradicts the assumption on the choice of $S$ as minimizing  the complexity as in 
Definition \ref{def: complexity}.

\vskip5pt

\noindent (7) The proof in this case is the same as in case (6), using the isotopy described 
by Figure \ref{reducing bubbles b}. Note that as a result of  the isotopy we see two more 
intersection points with $L$ but one less bubble.

\vskip5pt

\noindent (8) If such an arc $\alpha \subseteq c\cap L$ then by pushing $S$ through $P$ 
in a neighborhood of $\alpha$ we reduce the number of intersection points by 2, in 
contradiction to the minimal complexity of $S$.
\end{proof}

\begin{figure}[ht]
   \centering
    \begin{overpic}[width=3.5cm]{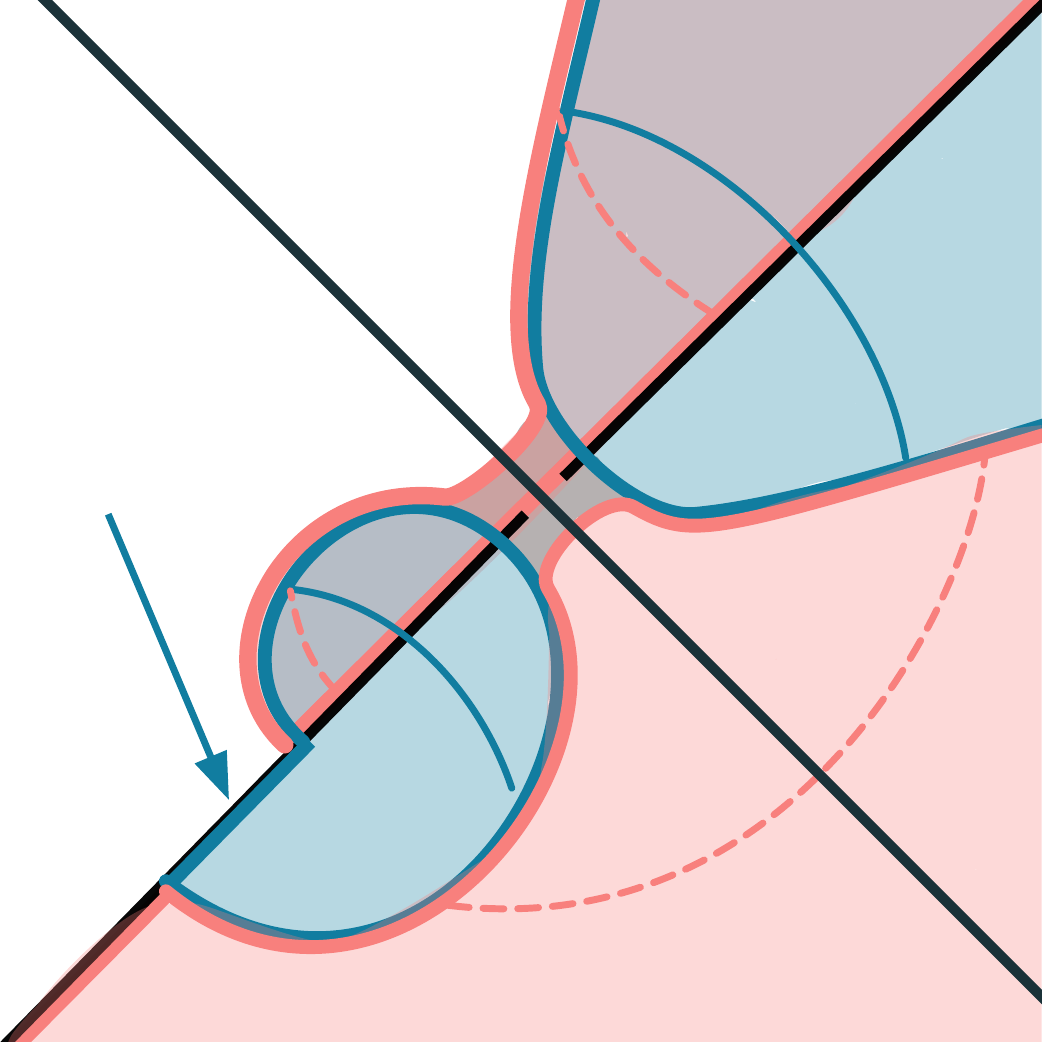}
    \put(5,55){$\alpha$}
    \end{overpic}
    \includegraphics[width=1.5cm]{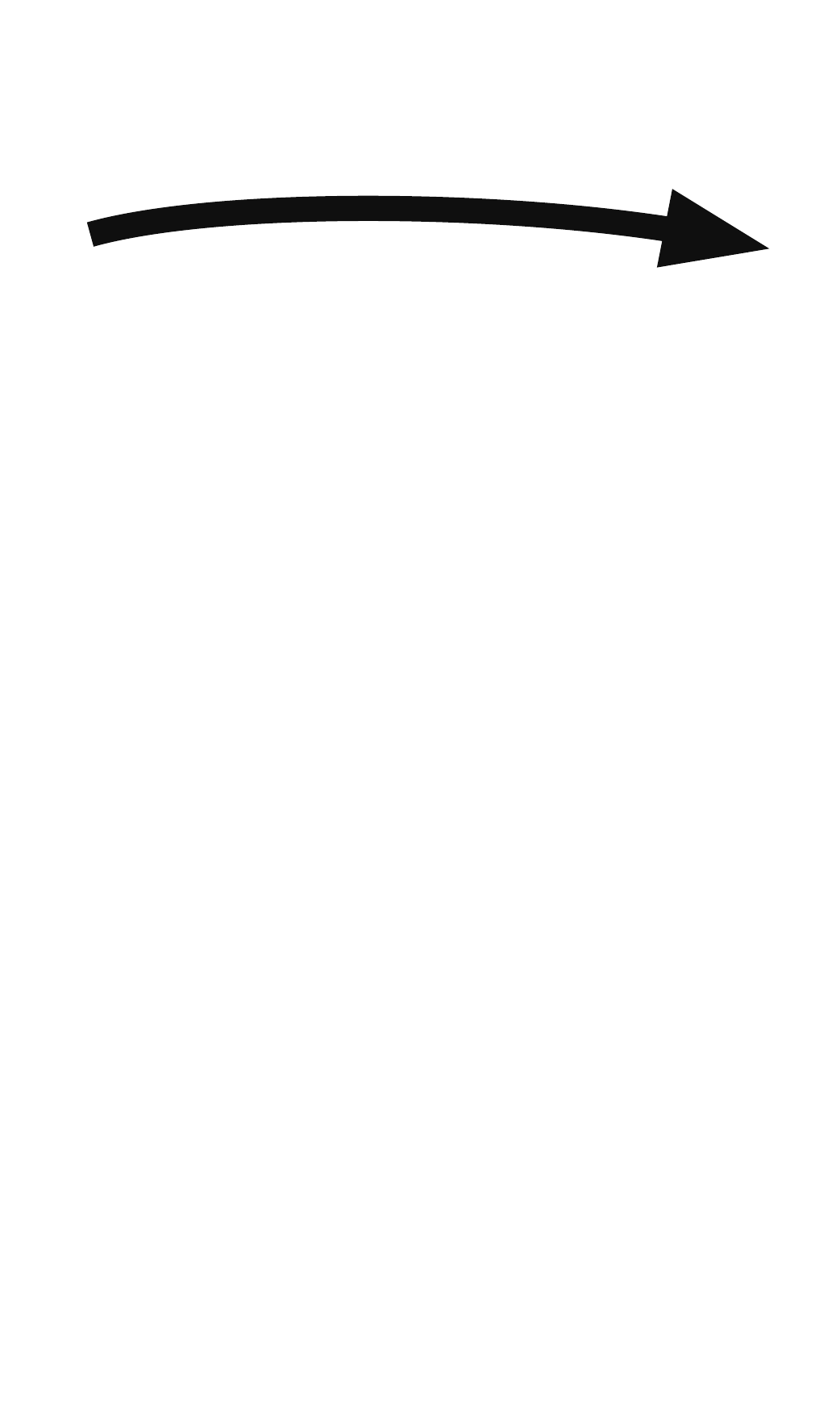}
    \includegraphics[width=3.5cm]{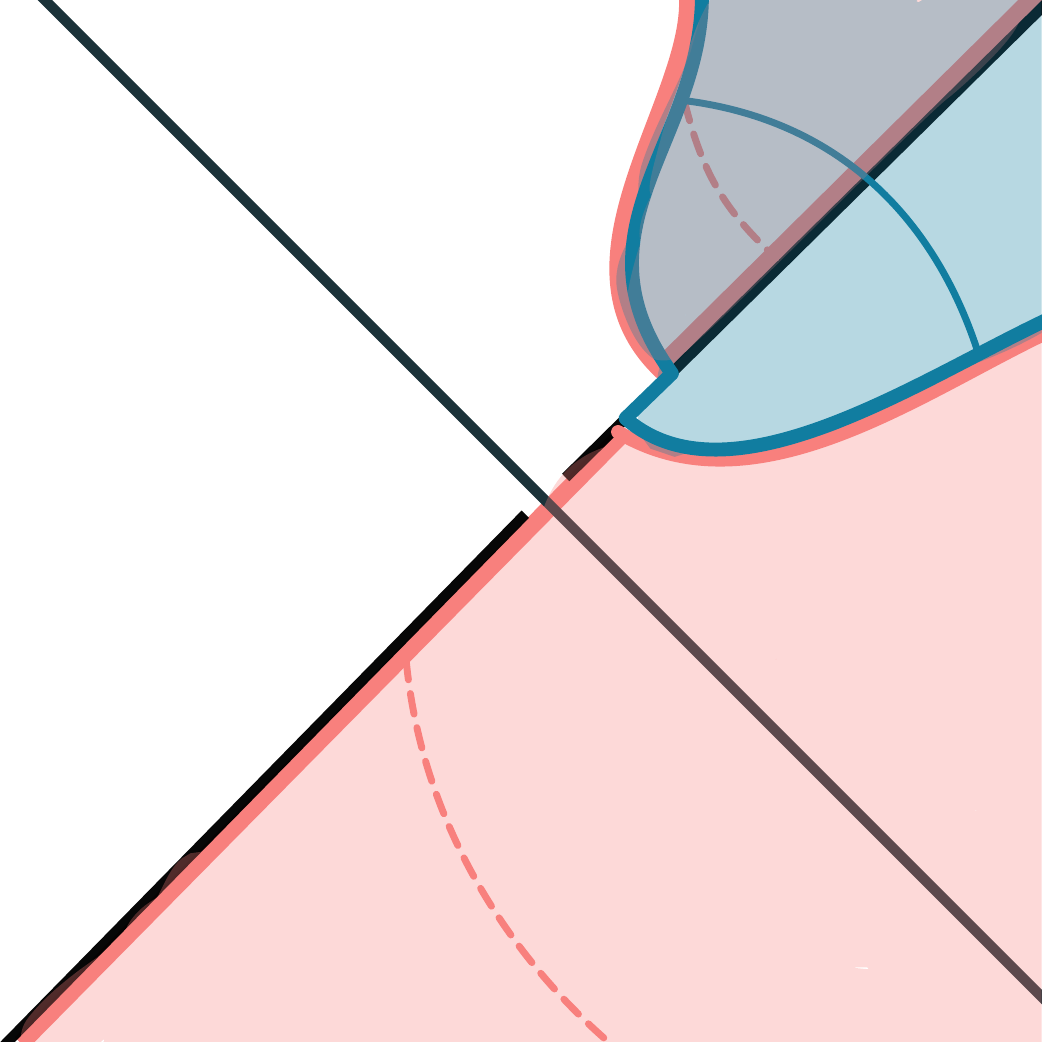}
    \caption{The isotopy in Case (6) of Lemma \ref{lem: properties of curves}}
    \label{reducing bubbles a}
\end{figure}

\begin{figure}[ht]
   \centering
    \begin{overpic}[width=3.5cm]{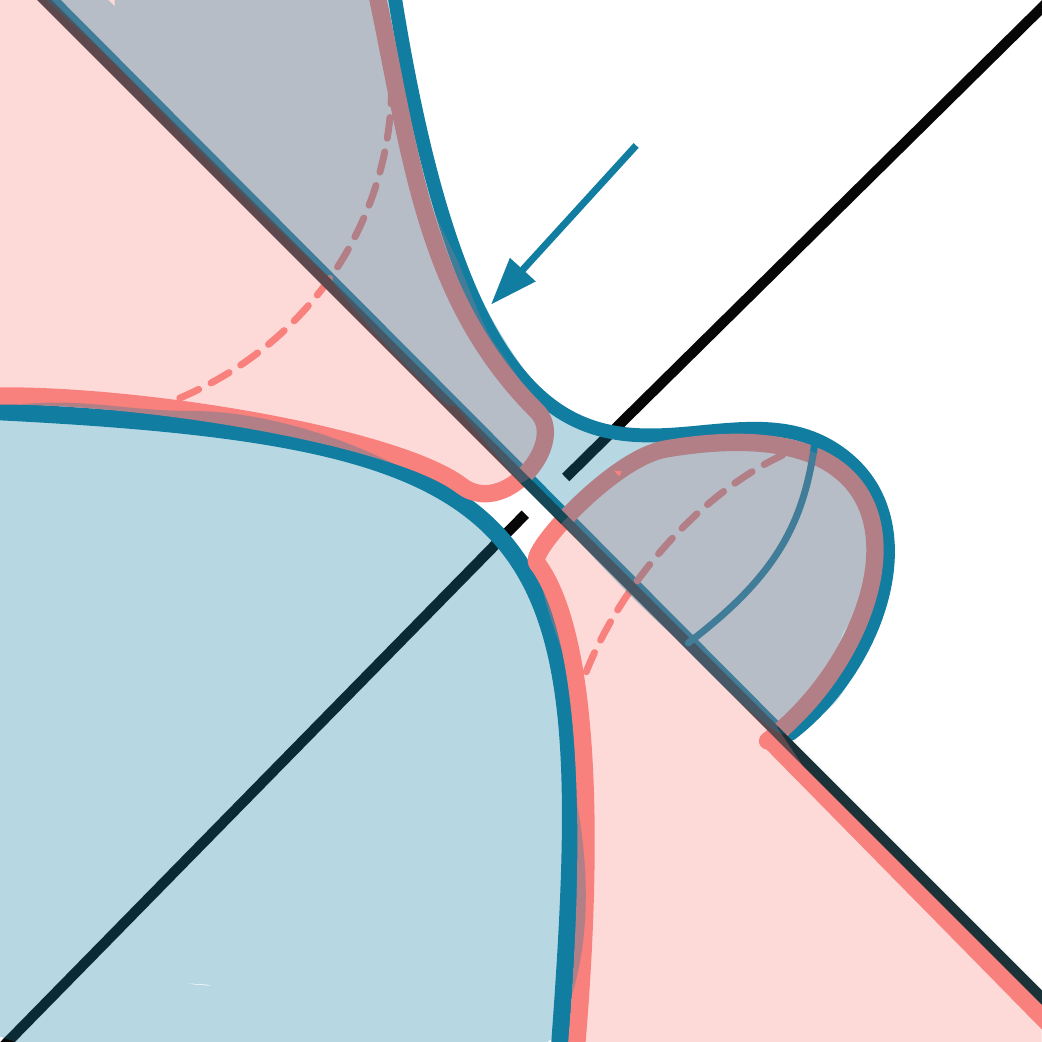}
    \put(60,90){$\alpha$}
    \end{overpic}
    \includegraphics[width=1.5cm]{figures/rightarrow.pdf}
    \includegraphics[width=3.5cm]{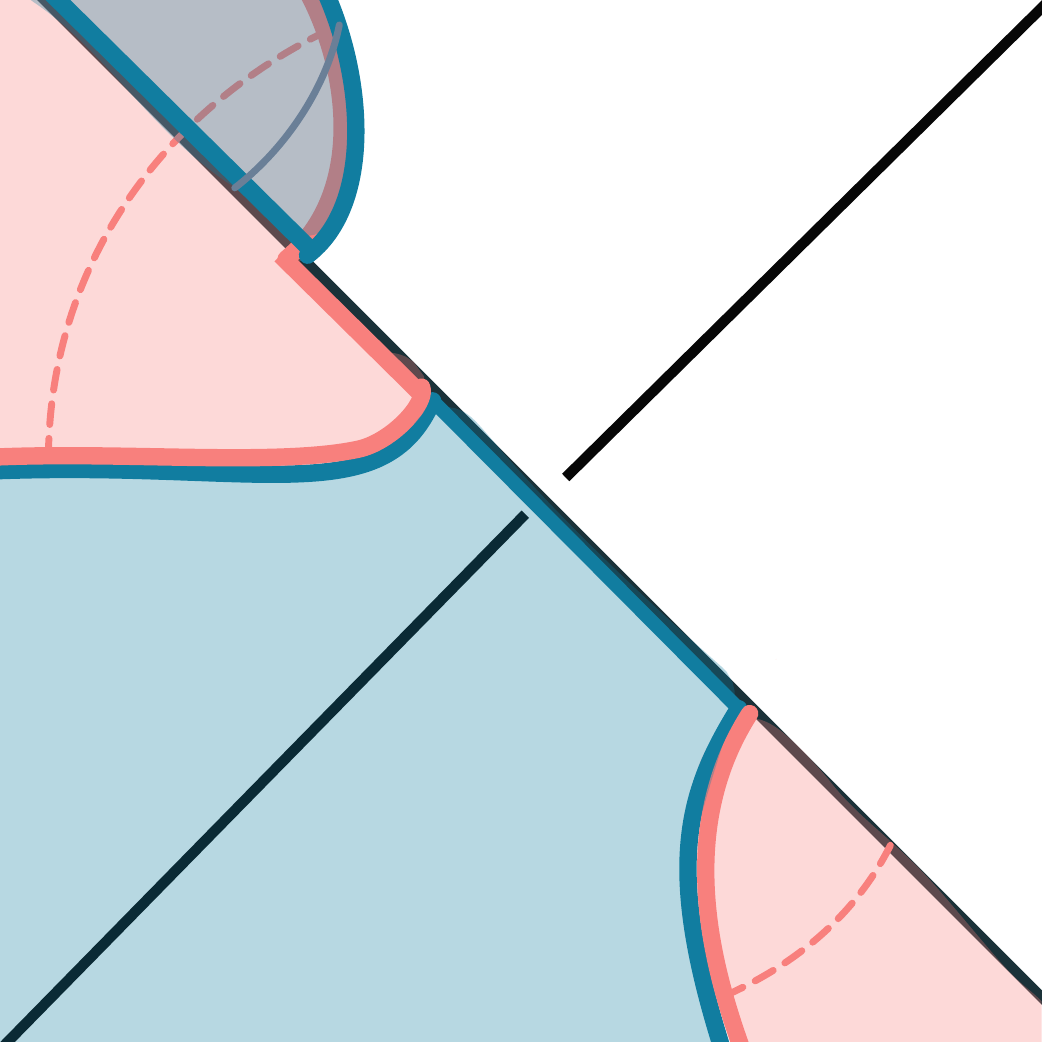}
    \caption{The isotopy in Case (7) of Lemma \ref{lem: properties of curves}} 
    \label{reducing bubbles b}
\end{figure}

\begin{definition}
A surface $S$ is \emph{taut} if it satisfies the conditions specified in 
Lemma~\ref{lem: properties of curves}.
\end{definition}

\begin{remark}
Note that if $S$ is taut then 
    $\calC_{0,0}=\calC_{0,2}=\calC_{i,2k+1}=\calC_{2k+1,0}=\emptyset$ for all $i,k\in\bbN\cup \{0\}$.
\end{remark}

\begin{remark}\label{rem: Assume taut} 
From now on we assume that the surface $S$ is taut.
\end{remark}

\medskip

\section{Euler Characteristic and curves of intersection} \label{sec: Curves and the 
Euler Characteristic}

\vskip10pt
\subsection{Distributing Euler characteristic among curves}
For each curve $c\in \calC$ we will define the \emph{contribution} of $c$, and show 
that the Euler characteristic of $S$ can be computed by summing up the contributions 
of curves $c\in\calC$. 

\begin{definition}
The \emph{contribution} $\chi_+(c)$  of a curve $c\in\calC$ is defined by
$$\chi_+(c) = \frac{\chi(S_c)}{|\partial S_c|} - \frac{1}{4}\bbb{c}.$$
\end{definition}
Note that if $S$ is taut, then $S_c\homeo D^2$, and therefore $\chi_+(c)=1 - \frac{1}{4}\bbb{c}$.

\begin{lemma}\label{lem: Euler characteristic from Euler contributions}
If $S \subset S^3 \ssm \NN(L)$ is taut then $\chi(S)=\sum_{c\in\calC} \chi_+(c)$.
\end{lemma}

\begin{proof}
The union of the collection of all the curves $c\in\calC$ on $S$ is an embedded graph $X$ of $S$. 
Let $X^0$ and $X^1$ denote the vertex and the edge sets of $X$ respectively. The vertices of the 
graph are the points in $S$ corresponding to the points on $c$ which are on 
$P\cap \calB$  (i.e., where $c$ meets a bubble) or $P\cap L$ (i.e., an intersection point of $c$). 
The graph $X$ partitions $S$ into disk regions of three types:
\begin{enumerate}
    \item the subsurfaces $S_c \subseteq \widehat{S} \cap H^\pm$ for $c\in \calC^\pm$, 
    \smallskip
    \item the regions $R \subseteq \widehat{S}\cap B$ where $B\in\calB$ is a 3-ball bounded by a bubble, or
    \smallskip
    \item regions $D \subset S$ corresponding to Type (2) disks.
\end{enumerate}
In case (3), the regions $D$ are disks whose boundary consists of two arcs on $L$ and two edges of $X$. 
By collapsing  each such disk $D$ to one of the edges in $X$ we get a homotopic surface. By abuse of 
notation, we call it $S$, and call the corresponding graph $X$. Note that in the new surface, 
$\partial S \subset X$. It follows that 
\begin{align}\label{eq: chi S using the graph}
\begin{split}
    \chi(S) &= \chi(X) + \sum_{S' \subseteq S \cap H^\pm} \chi(S') + \sum_{R \subseteq S\cap B} \chi(R)\\
    &= |X^0| - |X^1| + \sum_{S' \subseteq S \cap H^\pm} \chi(S') + \sum_{R \subseteq S\cap B} \chi(R).
\end{split}
\end{align}
We compute how each $c\in\calC$ contributes to each of the summands in \eqref{eq: chi S using the graph}: 

\underline{The vertices of $X$.} Every curve 
$c\in\calC$ passes through $2\bbl{c}$ vertices of $X^0$ in the interior of $S$ (because it goes in 
and out of a bubble). Further, it goes through $\bdr{c}$ vertices of $X^0$ in $\partial S$. 
Moreover, each of these vertices belongs to two curves $c\in \calC$. Hence,

\begin{equation}
    |X^0| = \sum_{c\in\calC} (\bbl{c} + \tfrac{1}{2} \bdr{c})
\end{equation} 

\underline{The edges of $X$.} Every curve $c\in \calC$ passes through $2\bbl{c} + \bdr{c}$ 
edges in $X^1$. Note that every edge in $X^1\cap \interior{S}\cap \calB$ or in $X^1\cap \partial S$ 
belongs to exactly one curve in $\calC$, while each edge in $X^1 \cap \interior{S} \ssm \calB$ 
belongs to two curves in $\calC$. Thus, edges in  $X^1\cap \interior{S}\cap \calB$ 
are in one to one  correspondence with the bubbles they meet. Hence,
$$|X^1 \cap \interior{S} \cap \calB | = \sum_{c\in \calC} \bbl{c}.$$
Similarly  each edge in $X^1\cap \partial S$ accounts for 2 in $\bdr{c}$. Hence,
$$|X^1 \cap \partial S  | = \sum_{c\in \calC} \tfrac{1}{2} \bdr{c}.$$
Each edge in $X^1 \cap \interior{S} \ssm \calB$
accounts for two vertices in $X^0$. So the number of these  edges is equal to  
 $$|X^1 \cap \interior{S} \ssm \calB| = \tfrac{1}{2} |X^0| =  
 \tfrac{1}{2}(\sum_{c\in\calC} (\bbl{c} + \tfrac{1}{2} \bdr{c})).$$
Adding these contributions together gives
\begin{equation}
    |X^1| = \sum_{c\in\calC}( \tfrac{3}{2}\bbl{c} + \tfrac{3}{4} \bdr{c}).
\end{equation} 

\underline{Regions $S' \subset S\cap H^\pm$.} To every curve $c\in\calC$ there is a surface 
$S_c \subseteq S\cap H^\pm$, and each such surface is 
associated to $|\partial S_c|$ curves $c\in\calC$. Thus,
\begin{equation} 
    \sum_{S' \subseteq S \cap H^\pm} \chi(S') = \sum_{c\in\calC} \frac{\chi(S_c)}{|\partial S_c|}.
\end{equation}

\underline{Regions $R\subset S\cap B$.} Each curve $c\in \calC$ passes through the 
boundary of $\bbl{c}$ such regions. As each such 
region has four curves passing through its boundary, we have
\begin{equation}
    \sum_{R \subseteq S\cap B} \chi(R) = \sum_{c\in C} \tfrac{1}{4}\bbl{c}.
\end{equation}

\medskip

Summing over all of the above we get,
\begin{align*}
\chi(S) &= |X^0| - |X^1| + \sum_{S' \subseteq S \cap H^\pm} \chi(S') + \sum_{R \subseteq S\cap B} \chi(R)\\
&= \sum_{c\in \calC} (\tfrac{\chi(S_c)}{|\partial S_c|} - \tfrac{1}{4}(\bbl{c} + \bdr{c}))\\
&= \sum_{c\in \calC} \chi_+(c).
\end{align*}
\end{proof}

\subsection{Redistributing positive Euler characteristic among curves}
In order to prove that $L$ is hyperbolic, we would eventually be interested in surfaces 
$S$ with non-negative Euler characteristic, namely spheres, annuli and tori. It would 
thus be useful to distribute the Euler characteristic of $S$ in such a way that each 
summand is non-positive. This would rule out 2-spheres, and show that each summand must 
be zero for $S$ to be a torus or an annulus.

Lemma \ref{lem: Euler characteristic from Euler contributions} shows that the Euler 
characteristic of $S$ is the sum of the contributions of curves in $\calC$.  However, 
some curves might have positive contributions. Our strategy would be to ``pass'' the 
contributions of curves with $\chi_+> 0$ to ``neighbouring'' curves with $\chi_+<0$. 
This will be done by defining $\chi'(c)$ in Definition \ref{def: distributing}, and 
proving in Lemmas \ref{lem: summing chi' gives Euler char} and 
\ref{lem: non-positive contribution of chi'} respectively that $\chi(S)= \sum\chi'(c)$ 
and $\chi'(c)\le 0$.

Our search for ``neighbouring'' curves will use the following definition:
\begin{definition}\label{def: arcs and opposite}
\hfill
\begin{enumerate}
    \item Let $c\in \calC$, two subarcs $\alpha_1,\alpha_2$ of $c$ whose endpoints are not 
    in $\calB \cup L$ are \emph{equivalent} if they are isotopic in $c$ so that the endpoints 
    of each arc in the isotopy are not in $\calB \cup L$.
    Note that for equivalent arcs $\alpha_1,\alpha_2$, we have $\bbl{\alpha_1}=\bbl{\alpha_2}$ 
    and $\bdr{\alpha_1}=\bdr{\alpha_2}$. 
    We will use the term \emph{subarc of $c$} (or simply \emph{arc}) to indicate its 
    equivalence class. Two arcs $\alpha,\alpha'$ are \emph{disjoint} if they are subarcs of 
    different curves or if they have disjoint representatives. Alternatively, two arcs 
    $\alpha,\alpha'$ are non-disjoint, if they are subarcs of the same curve, and they 
    overlap in a bubble or an intersection point.
    \item Two curves (or subarcs of curves) $c,c'\in\calC$ are said to be \emph{opposite 
    along an arc $\alpha$}, or simply \emph{opposite}, if 
    $c\ne c'$ and $\alpha \subset c\cap c' \ssm L$.
    Note that if $c,c'$ are opposite then one of them is in $P^+$ and the other in $P^-$.
\end{enumerate}
\end{definition}

\begin{definition}\label{def: sets}
Denote by $\calC_{\pos}$ (resp. $\calC_{=0},\calC_{\np},\calC_{\negative}$) the collection of 
all $c\in \calC$ such that $\chi_+(c)>0$ (resp. $=0,\le 0, <0$). 
\end{definition}

Note that if $S$ is taut, then
$\calC_{\pos} = \calC_{2,0} \cup \calC_{1,2}$ and 
$\calC_{=0} = \calC_{4,0}\cup \calC_{2,2} \cup \calC_{0,4}$. 
We begin by studying the curves in $\calC_{\pos}$:
\medskip

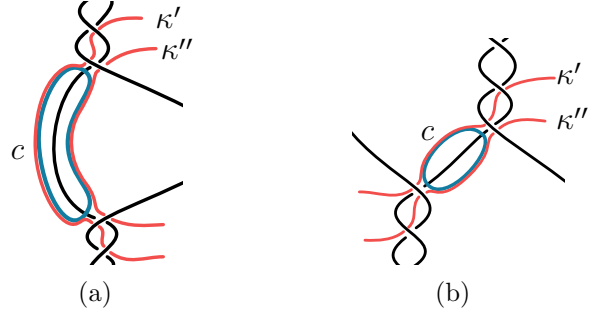
\begin{figure}
    \centering
    \subfigure[]{
   \begin{overpic}[height=3.5cm]{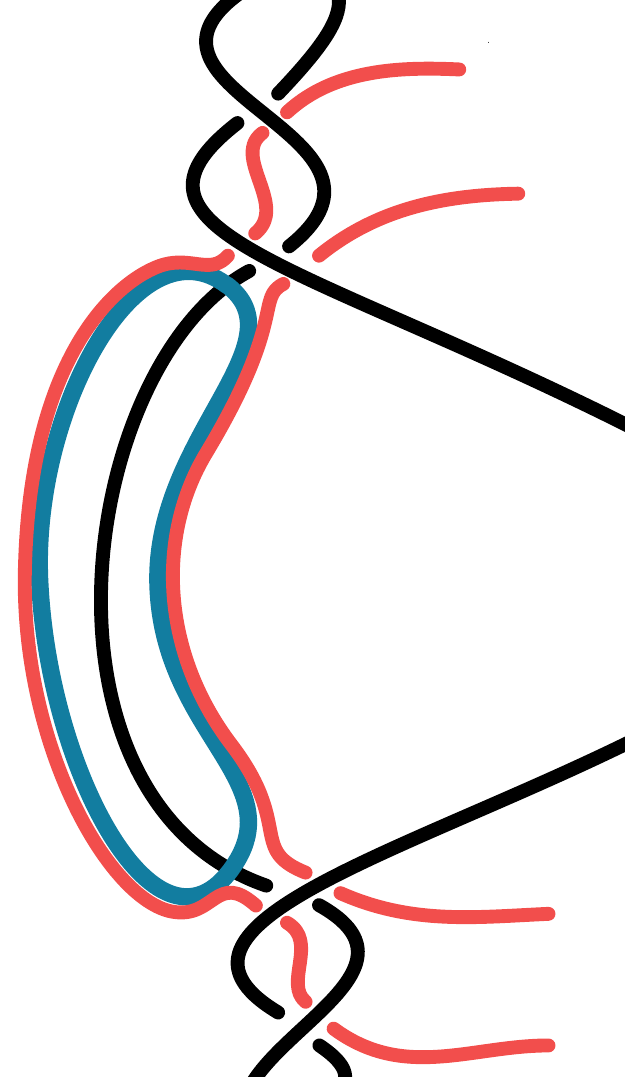}
\put(48,90){$\kappa'$}
\put(50,78){$\kappa''$}
\put(-7,40){$c$}
\end{overpic}}
    \hskip 2cm
    \subfigure[]{
    \begin{overpic}[height=3.cm]{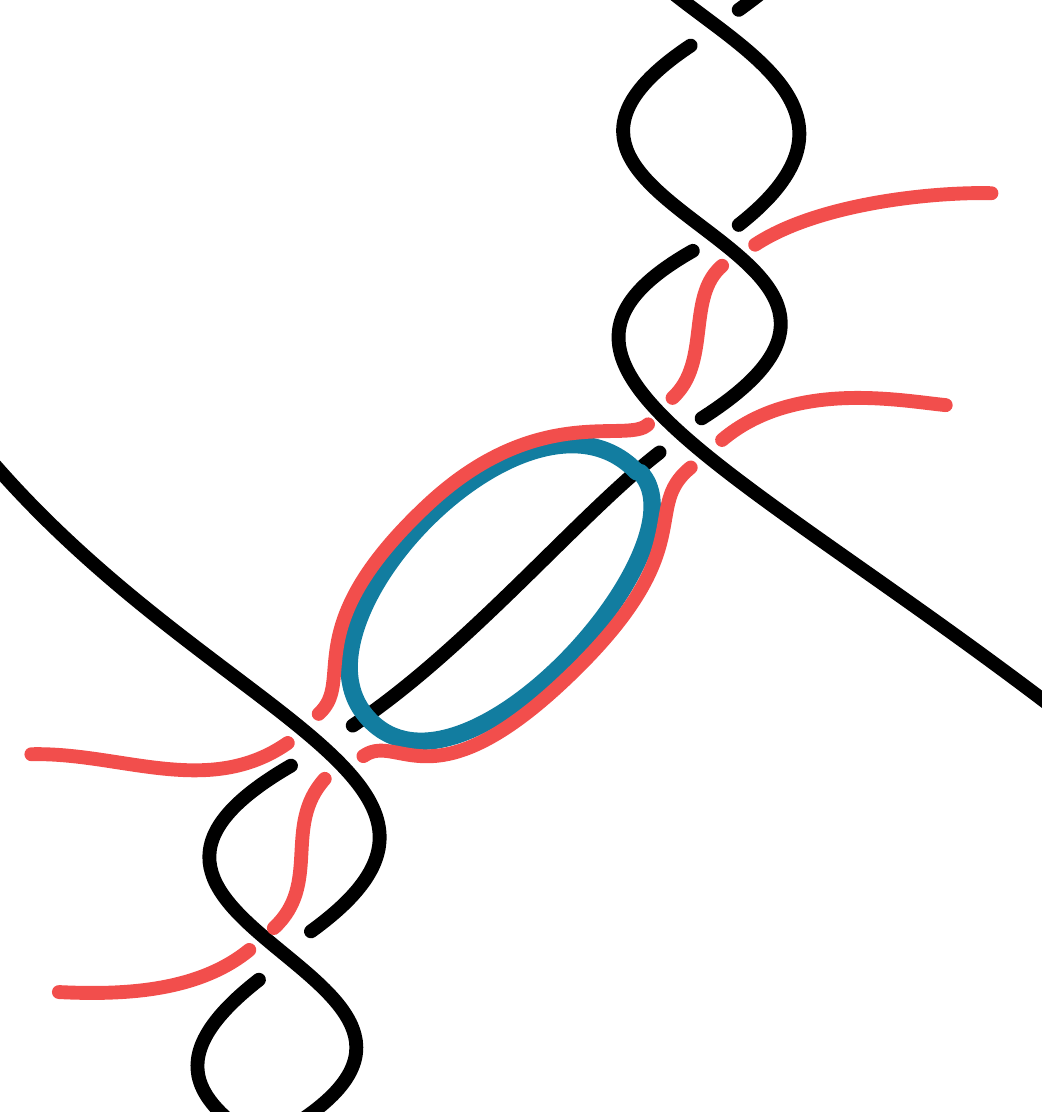}
\put(90,80){$\kappa'$}
\put(90,60){$\kappa''$}
\put(30,55){$c$}
\end{overpic}
}
    \caption{$c\in C_{2,0}$ and the opposite arcs $\kappa'$ and $\kappa''$.}\label{fig:c2}
    
\end{figure}

\underline{Curves in $\calC_{2,0}$.}
Let $c\in \calC_{2,0}$. The two possible configurations for $c$ are shown in Figure ~\ref{fig:c2}. 

\smallskip

In case (a), opposite to $c$ there are two arcs $\kappa'$ and $\kappa''$ shown in 
Figure \ref{fig:c2} passing through two and four bubbles respectively.
Let $\calK_{4,0}$ denote the collection of all such arcs $\kappa''$ opposite to some 
$c\in \calC_2$. (In this and following discussion, the subscript $4,0$ indicates 
that the arcs in $\calK_{4,0}$ pass through four bubbles and zero intersection points.)

In case (b), opposite to $c \in \calC_{2,0}$ there are two arcs, $\kappa'$ and $\kappa''$ 
shown in Figure \ref{fig:c2}.
Each of the arcs $\kappa'$ and $\kappa''$ passes through three bubbles.
Let $\calK_{3,0}$  denote the collection of all such $\kappa',\kappa''$ 
which are opposite to some $c\in\calC_{2,0}$ as described above. 

Each arc $\kappa'\in \calK_{3,0}$ is part of some closed curve $c'\in \calC$. We would 
like to distinguish those $\kappa'$ for which $c'\in \calC_{4,0}$ as those have $\chi_+(c')=0$. 
We denote the subcollection of all $\kappa'\in\calK_{3,0}$ for which $c'\notin \calC_{4,0}$ by
$\widehat\calK_{3,0}\subset \calK_{3,0}$.

Let $\kappa'\in\calK_{3,0} \ssm \widehat\calK_{3,0}$. 
Then $\kappa'$ is a subarc of a curve $c'$ with $\bbl{c'}=4$, which is opposite to some 
$c\in\calC_{2,0}$.  
 Note that this case can only occur in the ``corners'' of the plat, 
as shown in Figure \ref{fig: cases of kappa tilde} for the top left ``corner''. 
Let $\tild{\kappa}$ be the arc shown in 
Figure \ref{fig: cases of kappa tilde}.
Let $\tild{\calK}$ denote the collection of all $\tild{\kappa}$ constructed in this way.

\medskip
\underline{Curves in $\calC_{1,2}$.}
Schematically, a curve $c\in\calC_{1,2}$ must be as shown in Figure ~\ref{fig:c12}.

\begin{remark}\label{rem: c12 passes through top or bottom}
As $c$ emanates from $c\cap L$ into regions of different colors, $c\cap L$ cannot pass 
over a crossing which is not the top or bottom in its twist box.
\end{remark}

\smallskip
\begin{figure}[ht!]
\centering
\begin{overpic}[width=3.5cm]{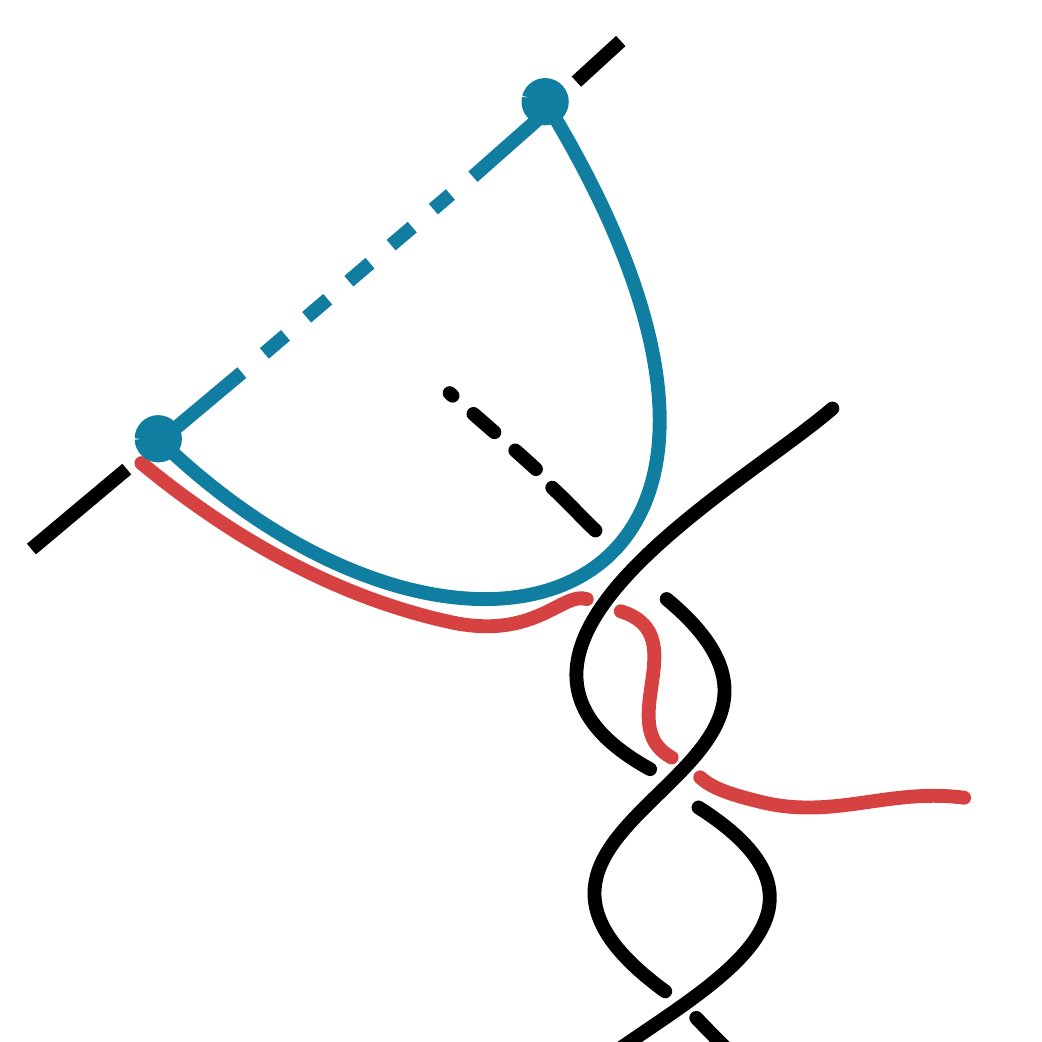}
\put(65,75){$c$}
\put(95,23){$\kappa$}
\end{overpic}
    \caption{$c\in C_{1,2}$  and the opposite arc $\kappa$.}\label{fig:c12}
\end{figure}

Similarly, opposite to a curve $c\in\calC_{1,2}$ there is an arc $\kappa$ passing through two bubbles 
and one intersection point, shown in  Figure \ref{fig:c12}.
Let $\calK_{2,1}$ be the collection of all such $\kappa$ opposite to some $c\in\calC_{1,2}$.

Define $\calK =\widehat\calK_{3,0}\cup \calK_{4,0} \cup \tild{\calK}\cup \calK_{2,1}$.
\begin{lemma}\label{lem: Kappaare disjoint}
Arcs in $\calK$ are pairwise disjoint.
\end{lemma}

\begin{proof}
By definition (see Definition \ref{def: arcs and opposite}), two arcs are not-disjoint if they 
overlap in a bubble or an intersection point.

One can check that the intersection of an arc $\kappa$ with a bubble in a twist box 
determines its intersection with the entire twist box, and that $\kappa$ must pass 
through either the top or bottom bubble in this twist box.
Because the plat is 3-highly twisted, an arc $\kappa$ cannot pass through both the 
top and bottom bubbles of the twist box. 
The part of the arc leaving the twist box determines the curve 
$c\in \calC_{2,0}\cup \calC_{1,2} \cup \calC_{4,0}$ opposite to it.
Thus, two arcs $\kappa_1,\kappa_2$ which overlap in a bubble determine the same 
opposite curve $c$, hence they are identical.

Similarly, if two arcs $\kappa_1,\kappa_2$ overlap in an intersection point, they are 
opposite to the same curve $c\in \calC_{1,2}$, and they are identical.
\end{proof}

Recall that our goal is to redistribute the Euler characteristic among curves so that 
each will contribute non-positively. The quantity $\chi'$ defined below is the sought for redistribution.

\begin{definition}\label{def: distributing}
Let $c'\in \calC_{\np}$. Let $n_3$ (resp. $\widehat{n}_3$; $n_4$; $\tild{n}$; $n_{2,1}$) be the number 
of subarcs $\kappa \in \calK_{3,0}$ (resp. $\widehat\calK_{3,0}$; $\calK_{4,0}$; $\tild{\calK}$;
$ \calK_{2,1}$) in  $c'$. 
We associate to $c'$ the following quantity
\[
    \chi'(c') = \chi_+(c')+ \tfrac{1}{4}\widehat{n}_3 + \tfrac{1}{2}n_4 + \tfrac{1}{4}\tild{n} + 
    \tfrac{1}{4} n_{2,1}.
 \]
\end{definition}

The next lemma shows that $\chi'$ is a redistribution of the Euler characteristic of $S$ 
among curves in $\calC_{\le 0}$. 
    
\begin{lemma}\label{lem: summing chi' gives Euler char}
$\chi(S) = \sum_{c'\in\calC_{\np}} \chi'(c').$
\end{lemma}

\begin{proof}
Since by Lemma \ref{lem: Euler characteristic from Euler contributions}, 
$\chi(S) = \sum_{c\in \calC} \chi_+ (c)$, it remains to prove that 
$$\sum_{c\in \calC} \chi_+ (c) =  \sum_{c'\in\calC_{\np}} \chi'(c').$$
Subtracting $\sum_{c'\in\calC_{\np}}\chi_+(c')$ from both sides and recalling that 
$\calC_{\np} = \calC \ssm \calC_{\pos}$, we have to show
$$\sum_{c\in\calC_{\pos}} \chi_+(c) = \sum_{c'\in \calC_{\np}} (\chi'(c')-\chi_+(c')).$$
The left hand side is simply $\tfrac{1}{2}|\calC_{2,0}| + \tfrac{1}{4}|\calC_{1,2}|$ since 
$\calC_{\pos} = \calC_{2,0} \cup \calC_{1,2}$ and $$\chi_+(c)=\begin{cases} \tfrac{1}{2} & 
\mbox{if }c\in\calC_{2,0} \\ \tfrac{1}{4} & \mbox{if }c\in\calC_{1,2} \end{cases}.$$

By the definition of $\chi'$, the right hand side gives $\tfrac{1}{4}|\widehat\calK_{3,0}| +
\tfrac{1}{2}|\calK_{4,0}| + \tfrac{1}{4}|\tild{\calK}| + \tfrac{1}{4}|\calK_{2,1}|$.
Note that for every arc $\kappa\in\calK_{3,0} \ssm \widehat\calK_{3,0}$ there exists a \emph{unique} arc
$\tild{\kappa}\in\tild{\calK}$. Hence we have $|\calK_{3,0} \ssm \widehat\calK_{3,0}| =|\tild{\calK}|$.
Therefore, the sum becomes $\frac{1}{4}|\calK_{3,0}| + \frac{1}{2}|\calK_{4,0}|+\tfrac{1}{4}|\calK_{2,1}|$.

Since every curve in $\calC_{2,0}$ is opposite to either two arcs in $\calK_{3,0}$, 
or one arc in $\calK_{4,0}$ we get $|\calC_{2,0}| = \frac{1}{2}|\calK_{3,0}| + |\calK_{4,0}|$ 
which, after dividing by 2, gives 
\begin{equation}\label{eq: C2}
    \tfrac{1}{2}|\calC_{2,0}| = \tfrac{1}{4}|\calK_{3,0}| + \tfrac{1}{2}|\calK_{4,0}|.
\end{equation}
Similarly, every curve in $\calC_{1,2}$ is opposite to one arc in $\calK_{2,1}$ and so 
$|\calC_{1,2}|=|\calK_{2,1}|$ which, after dividing by 4, gives 
\begin{equation}\label{eq: C12}
\tfrac{1}{4}|\calC_{1,2}|=\tfrac{1}{4}|\calK_{2,1}|
\end{equation} 
Adding together \eqref{eq: C2} and \eqref{eq: C12} completes the proof.
\end{proof}

The next lemma shows that indeed $\chi'$ is non-positive.

\begin{lemma}\label{lem: non-positive contribution of chi'}
 $\chi'(c')\le 0$ for all $c'\in\calC_{\np}$.
\end{lemma}

\begin{proof}
Let $c'\in\calC_{\np}$ and let $\widehat{n}_3,n_4,\tild{n},n_{2,1}$ be as in 
Definition ~\ref{def: distributing}
of $\chi'(c')$. Then we have,
\begin{align}\label{eq: chi' inequality}
\begin{split}
\chi'(c')&=\chi_+(c')+\tfrac{1}{4}\widehat{n}_3 + \tfrac{1}{2}n_4 + 
\tfrac{1}{4}\tild{n}+\tfrac{1}{4}n_{2,1} \\
& \leq (1-\tfrac{1}{4}\bbb{c'}) +\tfrac{1}{4}\widehat{n}_3 + \tfrac{1}{2}n_4 +
\tfrac{1}{4}\tild{n}+\tfrac{1}{4}n_{2,1} \\
& \le 1 
+ \sum_{c'\supset\kappa\in\widehat\calK_{3,0}} (\tfrac{1}{4} - \tfrac{1}{4}\bbb{\kappa})
+\sum_{c'\supset\kappa\in\calK_{4,0}} (\tfrac{1}{2} - \tfrac{1}{4}\bbb{\kappa})\\
&\quad\quad + \sum_{\substack{c'\supset\kappa\in\tild{\calK}}} (\tfrac{1}{4} -
\tfrac{1}{4}\bbb{\kappa})
+ \sum_{c'\supset\kappa\in\calK_{2,1}} (\tfrac{1}{4} - \tfrac{1}{4}\bbb{\kappa})\\
&= 1-\tfrac{1}{2}(\widehat n_3+n_4+\tild{n}+n_{2,1}).
\end{split}
\end{align}

Where the last equality follows since each one of the summands is $-\tfrac{1}{2}$. We divide into 
cases depending on the sum $n=\widehat n_3+n_4+\tild{n}+n_{2,1}$.

\textbf{Case 0. $n=0$.} We have $\chi'(c')=\chi_+(c')$. But since 
$c'\notin \calC_{\pos}$ we have $\chi_+(c')\le 0$ and we are done. (Note that this case includes the case 
$\bbl{c'}=4$ and $n_3=1$. See Remark \ref{rem: special case of redistribution})

\textbf{Case 1. $n=1$.} That is, $c'$ contains a single subarc 
$\kappa\in \widehat \calK_{3,0} \cup \calK_{4,0} \cup \tild\calK \cup \calK_{2,1}$. 
We further divide into sub-cases:

\emph{Case 1.a.} $\kappa\in\widehat\calK_{3,0}$. Clearly $\bbl{c'}\ge \bbl{\kappa}= 3$. By definition 
of $\widehat\calK_{3,0}$, $\bbl{c'}\ne 4$. Then, either $\bbl{c'}\geq6$ or $c'$ has at least two intersection 
points with $K$. Therefore, 
$\bbb{c'}\ge 5$ which gives $$\chi'(c') = \chi_+(c')+\tfrac{1}{4}\le (1-\tfrac{5}{4}) + \tfrac{1}{4} = 0.$$

\emph{Case 1.b.} $\kappa\in\calK_{4,0}$ so $\bbl{\kappa}= 4$. The distance between the 
endpoints of $\kappa$ is 2 in the dual graph, so the closed curve $c'$ that contains 
$\kappa$ must contain two additional bubbles or two additional intersection points. Thus, 
$\bbb{c'}\ge 6$, which gives $$\chi'(c') = \chi_+(c') + \tfrac{1}{2} \le (1-\tfrac{6}{4}) + 
\tfrac{1}{2} = 0 .$$

\emph{Case 1.c.} $\kappa\in\tild{\calK}$. Here again $\bbl{\kappa}=3$.
This happens in one of the cases described in Figure \ref{fig: cases of kappa tilde}. 
The endpoints of $\tild{\kappa}$ are at distance $\ge 2$. Hence, the closed curve $c'$ 
containing $\kappa$ must have two additional bubbles or two additional intersection 
points. It follows that $\bbb{c'}\ge 5$, which gives 
$$\chi'(c') = \chi_+(c') + \tfrac{1}{4} \le (1-\tfrac{5}{4}) + \tfrac{1}{4}= 0 .$$

\emph{Case 1.d.} $\kappa\in\calK_{2,1}$ so  $\bbl{\kappa}=2$ and $\bdr{\kappa}=1$. 
The curve $c'$ must 
contain an additional intersection point.
Let $\alpha$ be the subarc of $c'$ between the two intersection points. Let $\kappa^*$ 
be a small continuation of $\kappa \cup \alpha$ along $c'$.

\begin{claim*}
The endpoints of $\kappa^*$ are at distance greater or equal to 1.
\end{claim*}
\begin{proof}
The arc $\alpha$ cannot contain more than one overpass: Otherwise, $\alpha$ passes over 
two crossings which occur in different twist boxes. This implies that the curve 
$c\in\calC_{1,2}$ opposite to $\kappa$ is such that $c\cap L$ is contained in one of 
the twist boxes, and connects the regions to its left and right, which must have the 
same color. This contradicts the fact that $c\ssm (c\cap L)$ passes through exactly 
one bubble.

If $\alpha$ does not contain an overpass the regions containing the endpoints of 
$\kappa^*$ have different colors, thus the endpoints of $\kappa^*$ are at distance 
greater or equal to 1.

If $\alpha$ contains one overpass then if the endpoints of $\kappa^*$ can be connected by an arc 
containing no intersection points or bubbles then the union of $\kappa^*$ and the arc bound a 
subdiagram of $D(L)$ which contradicts the assumption that $L$ is twist-reduced.
\end{proof}

It follows from the claim that $\bbb{c'}\ge \bbb{\kappa^*}+1=4+1=5$, which gives 
$$\chi'(c') = \chi_+(c') + \tfrac{1}{4} \le (1-\tfrac{5}{4}) + \tfrac{1}{4}= 0.$$

\begin{figure}
    \centering
   \begin{overpic}[height=4cm]{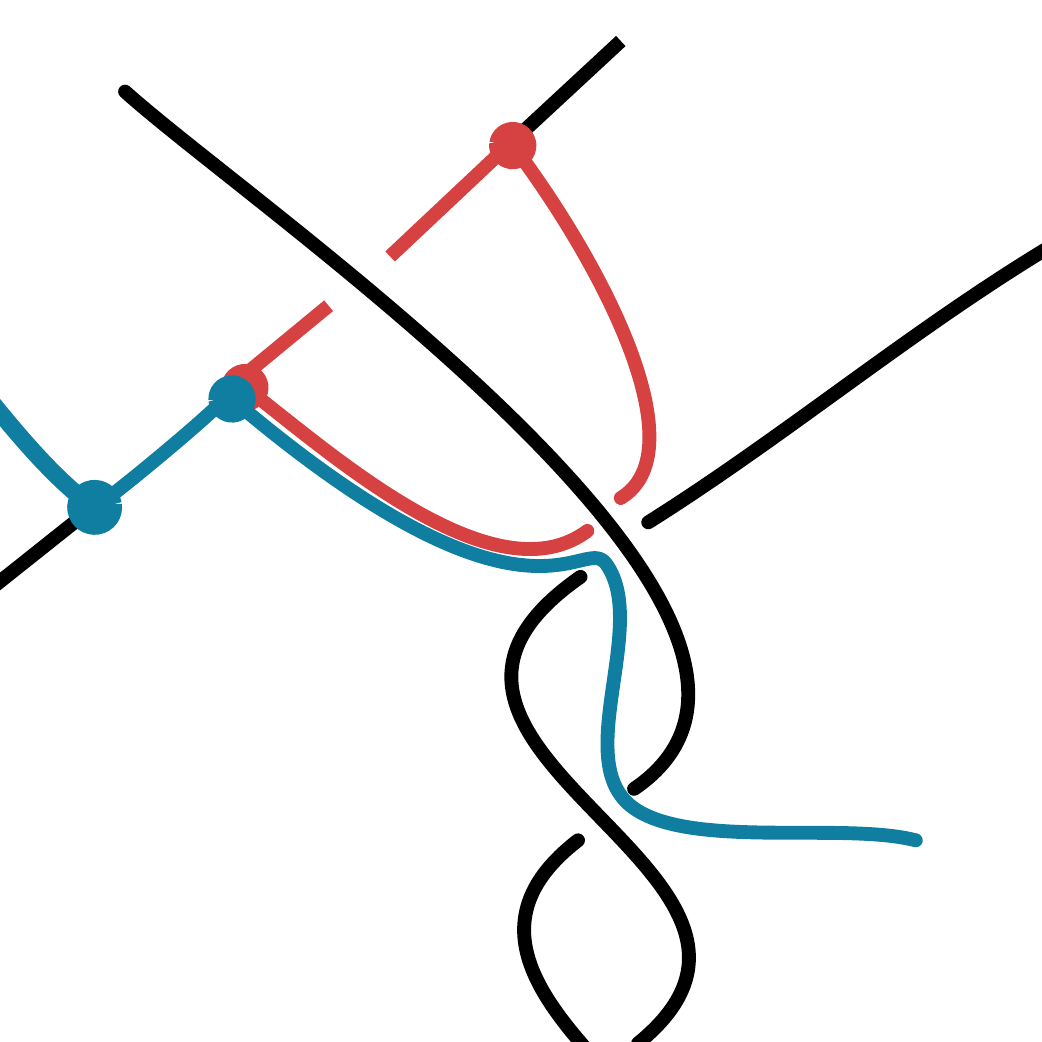}
\put(61,72){$c$}
\put(-9,55){$c'$}
\end{overpic}
    \hskip 2cm
    \begin{overpic}[height=4cm]{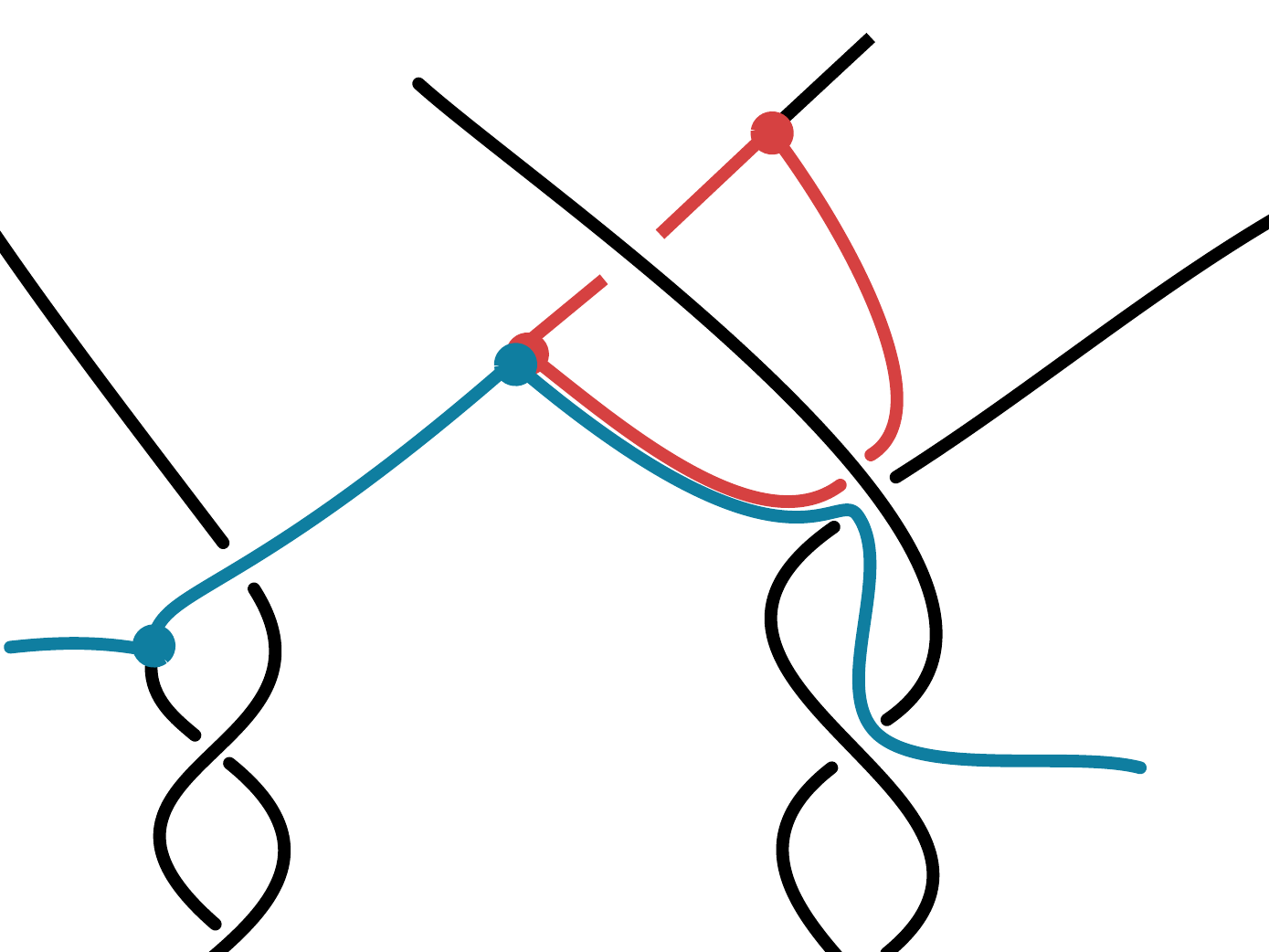}
\put(-6,25){$c'$}
\put(70,55){$c$}
\end{overpic}
    \caption{Two possibilities of a curve containing an arc in $\calK_{2,1}$.}
    \label{fig: K21}
\end{figure}

\vskip7pt

\textbf{Case 2. $n\ge 2$.} In this case we are done by 
inequality \eqref{eq: chi' inequality}.
\end{proof}

\begin{remark}\label{rem: special case of redistribution}
Note that, when 
$n_3=1 \mbox{ and }c'\in \calC_{4,0}$, i.e., when $c'$ contains a subarc in 
$\calK_{3,0} \ssm \widehat\calK_{3,0}$, then $\chi'(c')=\chi_+(c')=0$.
In this case, the positive contribution of the unique $c\in\calC_{2,0}$ opposite $c'$ is counted 
not by $\chi'(c')$ but by $\chi'(\tilde c)$ where $\tilde c$ is the curve opposite to $c'$ 
containing the corresponding arc $\tilde\kappa$. 
\end{remark}

The proof above also gives the following lemma:
\begin{lemma}\label{lem: classification of chi'=0}
 The curve $c'\in\calC_{\np}$ satisfies $\chi'(c')=0$ exactly in the following cases:
 \vskip5pt
\begin{enumerate}\setcounter{enumi}{-1}
    \item \label{lem: chi'=0 case no calK} $c'$ does not contain a subarc in $\calK$ and $\bbb{c'}=4$. 
    (it might have $n_3=1$)
    \vskip5pt
    \item \label{lem: chi'=0 case one arc}$c'$ contains exactly one subarc $\kappa$ of $\calK$ 
    and satisfies either:
      \vskip5pt
    \begin{enumerate} 
        \item \label{lem: chi'=0 case 3,0} $\bbb{c'}=5$ and $\kappa \in \widehat\calK_{3,0}$.
        \vskip5pt
        \item \label{lem: chi'=0 case 4,0} $\bbb{c'}=6$ and $\kappa \in \calK_{4,0}$.
        \vskip5pt
        \item \label{lem: chi'=0 case tilde}$\bbb{c'}=5$ and $\kappa \in \tild{\calK}$.
        \vskip5pt
        \item \label{lem: chi'=0 case 2,1}$\bbb{c'}=5$ and $\kappa \in \calK_{2,1}$.
        \vskip5pt
    \end{enumerate}
    In all cases, $\bbb{c'}=\bbb{\kappa}+2$.
      \vskip5pt
    \item \label{lem: chi'=0 case two arcs} $c'$ is the union of exactly two arcs of $\calK$.\qed
\end{enumerate}
\end{lemma}

As an immediate corollary to Lemma \ref{lem: non-positive contribution of chi'}  we get the following:
\begin{corollary}\label{cor: nonsplit}
The link $L$ is non-split.
\end{corollary}

\begin{proof}
Assume for contradiction that $S^3\ssm L$ contains an essential sphere. Let $S$ be an essential sphere with 
least complexity among all essential spheres. By Lemma \ref{lem: properties of curves}, $S$ is taut. 
By Lemma~\ref{lem: summing chi' gives Euler char}, $\chi(S) = \sum_{c'\in\calC_{\np}} \chi'(c').$ Thus it 
follows from Lemma~\ref{lem: non-positive contribution of chi'} that $\chi(S) \leq 0$, which is a
contradiction.
\end{proof}

\vskip15pt
\section{Analysing the curves}\label{sec: Analysing the curves}

As stated in the introduction our goal is to prove the following:

\begin{theorem}\label{thm: boundary parallel}
There are no essential tori or annuli in $ S^3 \ssm \NN(L)$. 
\end{theorem}

In what follows, \textbf{assume that $S$ is a taut essential torus or annulus.}
By Lemma~\ref{lem: summing chi' gives Euler char},  
$\sum_{c'\in\calC_{\np}} \chi'(c')=\chi(S) = 0$. By Lemma \ref{lem: non-positive contribution of chi'} 
each summand $\chi'(c')\le 0$, and thus must be  equal to $0$. 
It follows that all the curves 
$c'\in\calC_{\np}$ are as classified in Lemma \ref{lem: classification of chi'=0}. 

The proof of Theorem \ref{thm: boundary parallel} will proceed by analysing each case of 
Lemma \ref{lem: classification of chi'=0} separately, and showing that $S$ must be boundary parallel.
Before we give the proof we first need to define some notation which will be used below.

\begin{definition}
The projection of a twist region in the diagram is a rectangle called a twist box, we 
refer to its edges as top/bottom/left/right as obvious from Figure \ref{fig: turn and cross}.
\end{definition}

\begin{figure}[ht]
    \centering
   \begin{overpic}[height=5cm]{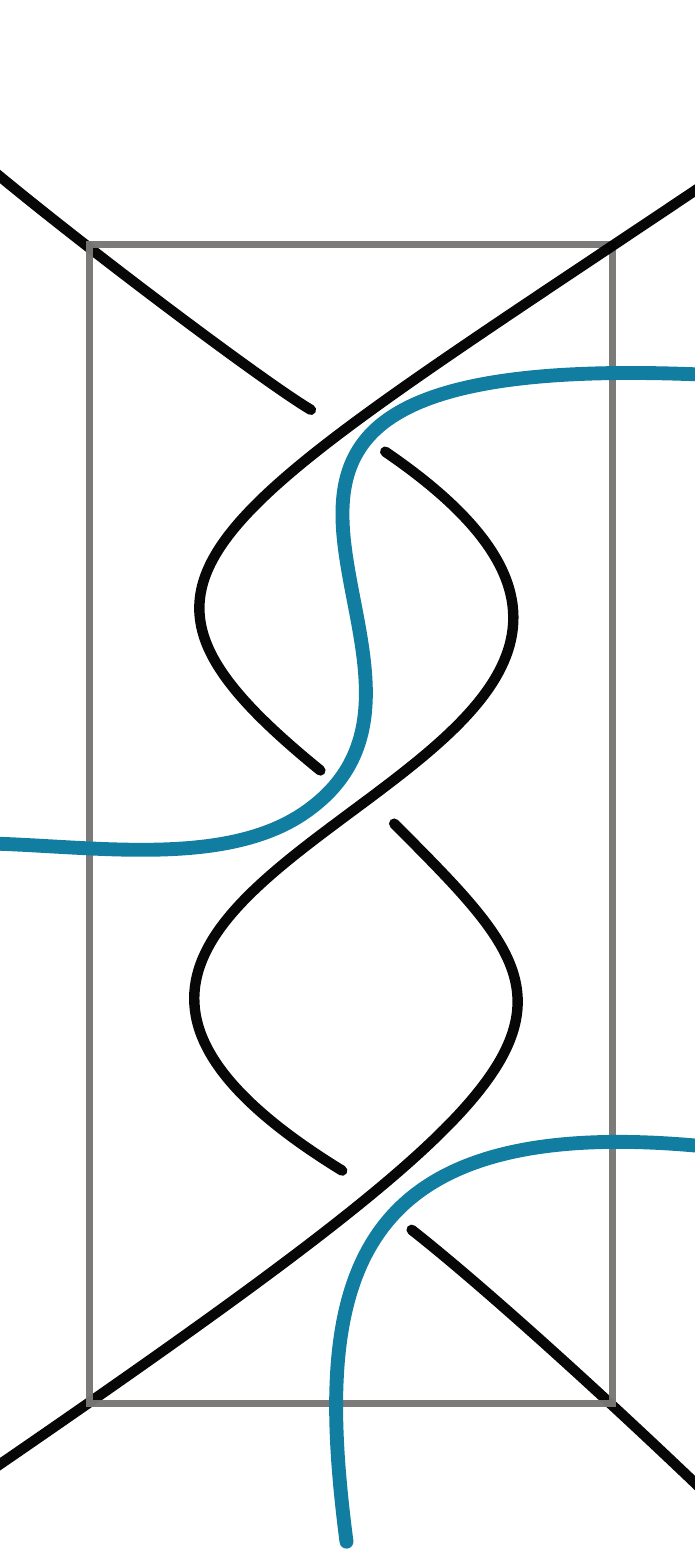}
\put(16,90){top}
\put(12,-3){bottom}
\put(-10,55){left}
\put(44,55){right}
\put(48,76){cross}
\put(48,22){turn}
\end{overpic}
    \caption{A crossing and turning arc in a twist box}
    \label{fig: turn and cross}
\end{figure}
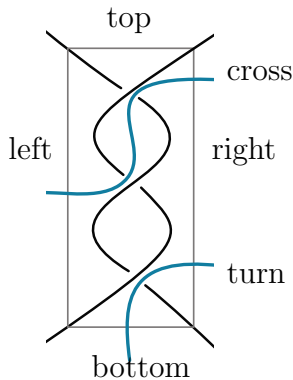

\begin{lemma}
\label{lem: no tilde}
There are no arcs in $\calK_{3,0} \ssm \widehat\calK_{3,0}$. In particular $\tild\calK=\emptyset$.

\end{lemma}

\begin{proof}
Assume in contradiction that there exists an arc $\kappa' \in \calK_{3,0} \ssm \widehat\calK_{3,0}$. 
By assumption, the curve $c'$ containing $\kappa'$ is in $\calC_{4,0}$. Let $c\in\calC_{2,0}$ be the 
curve opposite to $\kappa'$, let $\tild{\kappa}$ be the corresponding arc in $\tild\calK$ opposite 
to $c'$, and let $\tild c$ be the curve containing $\tild\kappa$. Without loss of generality, we may 
assume that $c\subset P^+$. This implies that $\kappa',c' \subset P^-$ and $\tild \kappa, \tild c\subset P^+$.
The possible configurations of $c,c',\kappa',\tild{\kappa}$ are depicted in Figure \ref{fig: cases of 
kappa tilde}. Note that under the assumptions above, the sign of the twists in the figure must 
be as depicted.

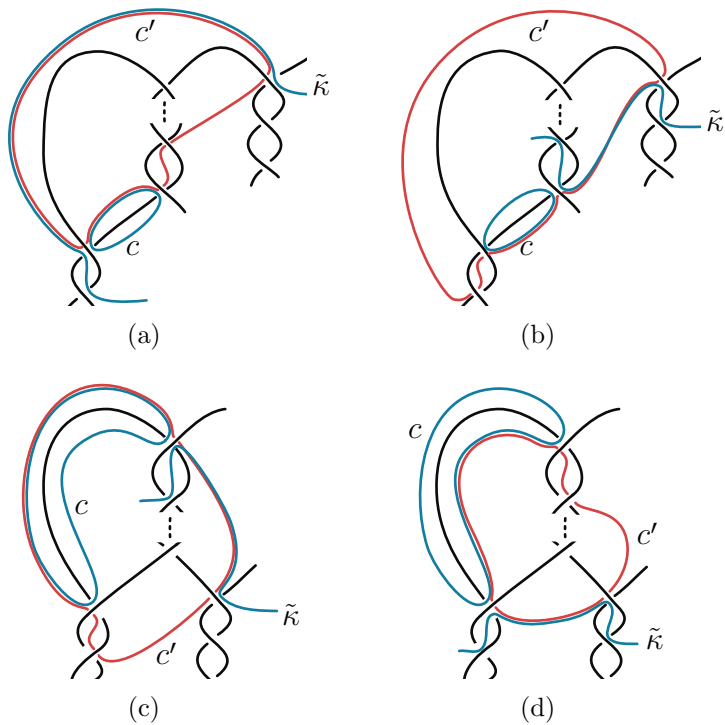
\begin{figure}[ht]
    \centering
    \subfigure[]{
   \begin{overpic}[height=4cm]{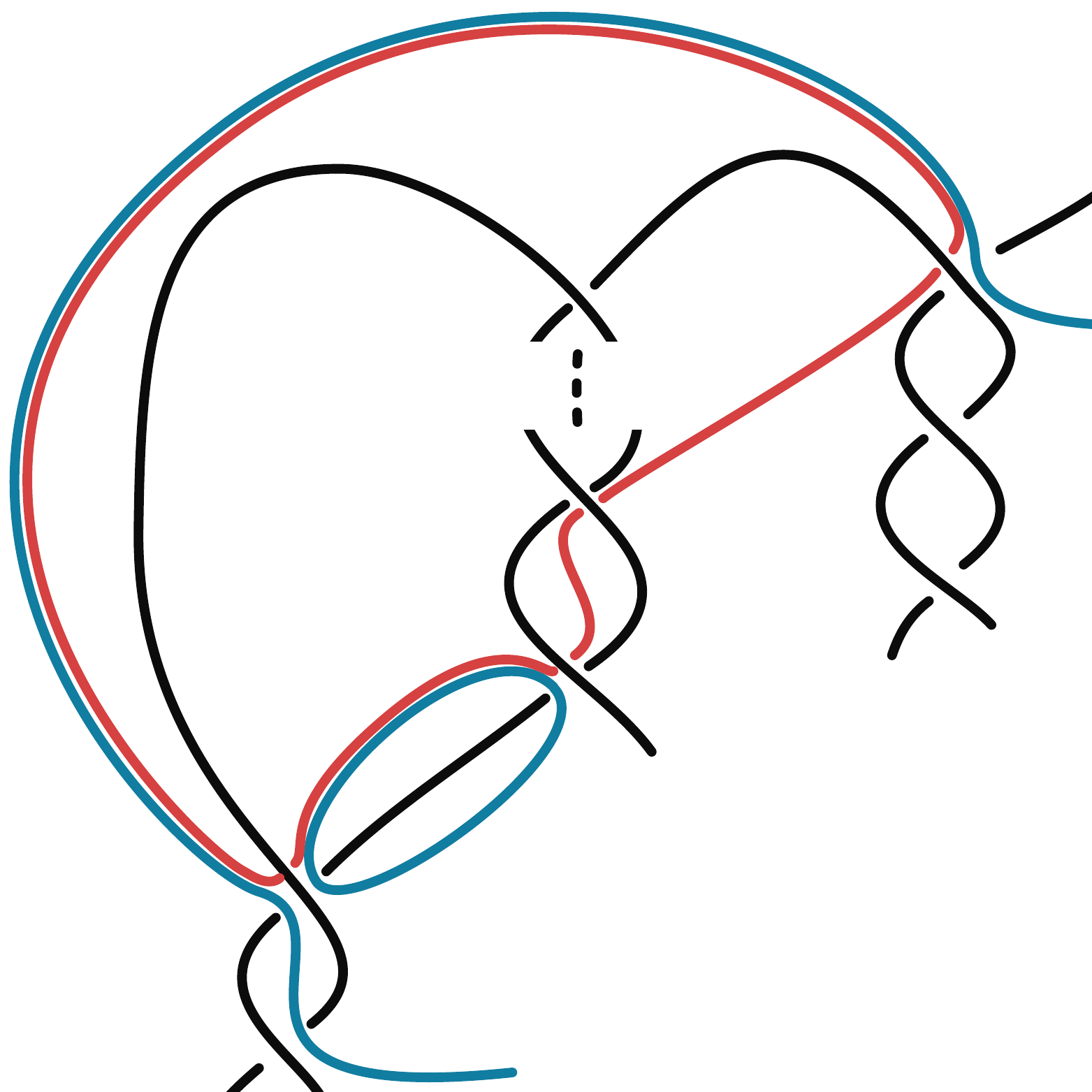}
\put(40,17){$c$}
\put(102,70){$\tilde\kappa$}
\put(43,88){$c'$}
\end{overpic}}
    \hskip 1cm
   \subfigure[]{
   \begin{overpic}[height=4cm]{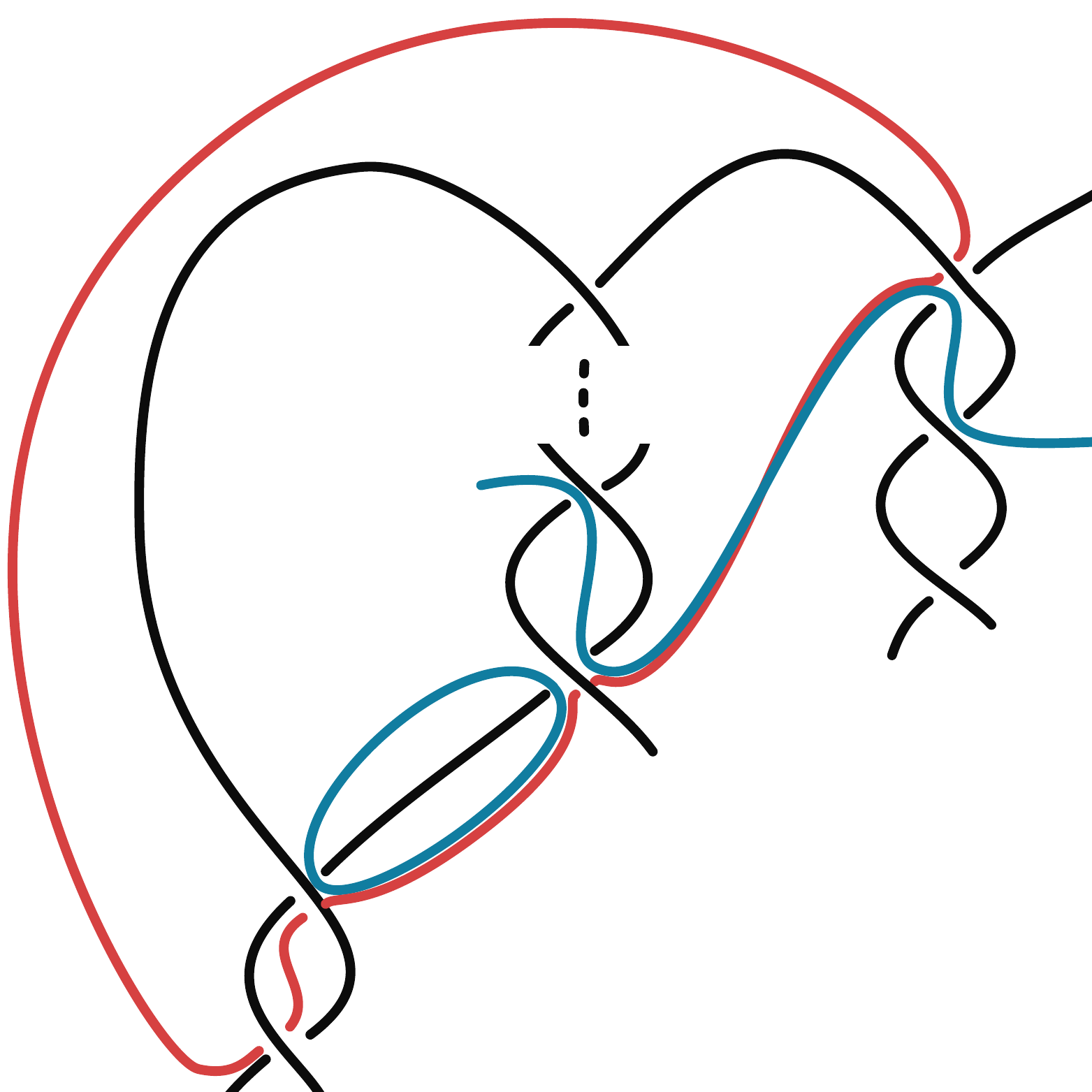}
\put(40,17){$c$}
\put(102,58){$\tilde\kappa$}
\put(43,88){$c'$}
\end{overpic}}\\

\subfigure[]{
   \begin{overpic}[height=4cm]{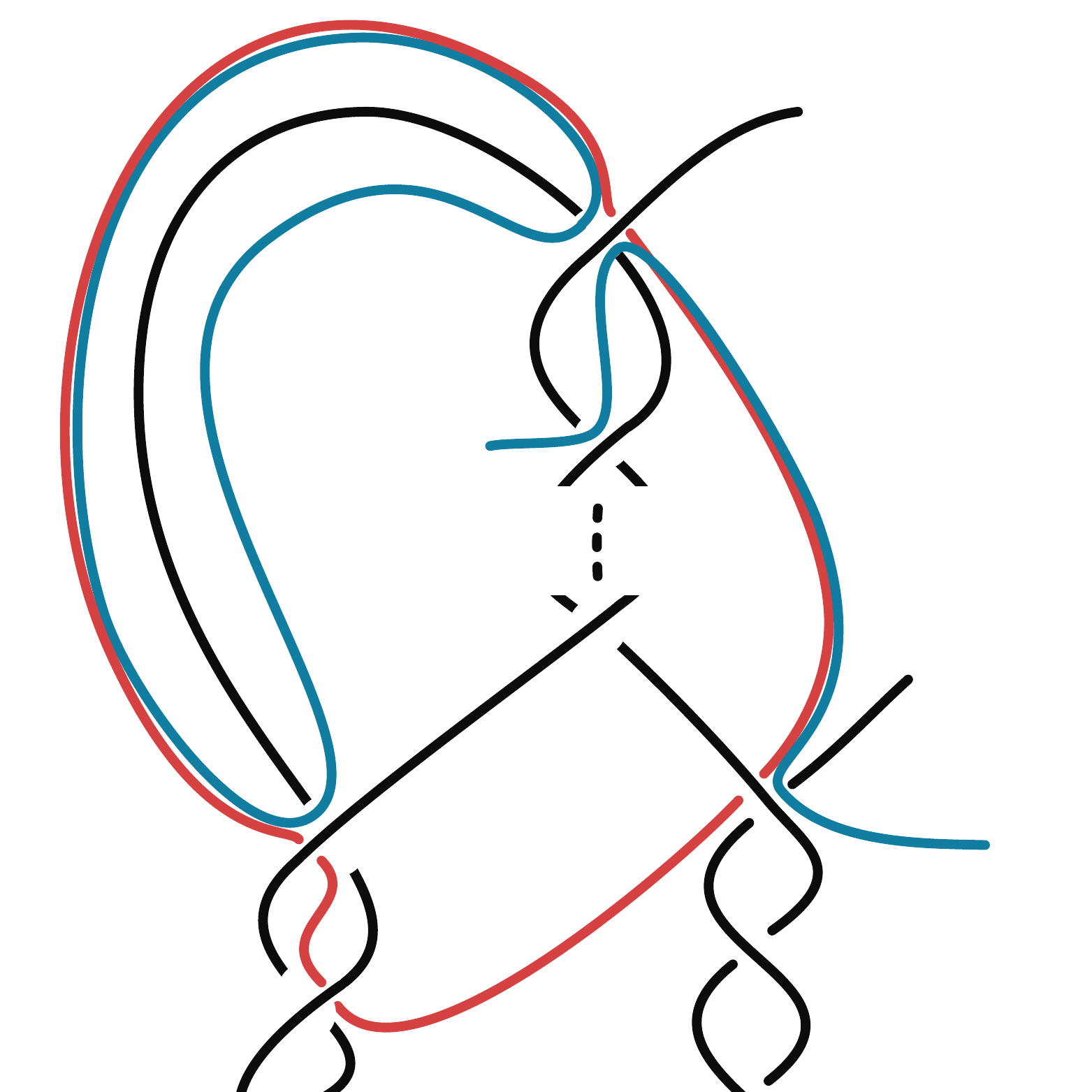}
\put(23,55){$c$}
\put(92,18){$\tilde\kappa$}
\put(50,5){$c'$}
\end{overpic}}
\hskip 1cm
  \subfigure[]{
   \begin{overpic}[height=4cm]{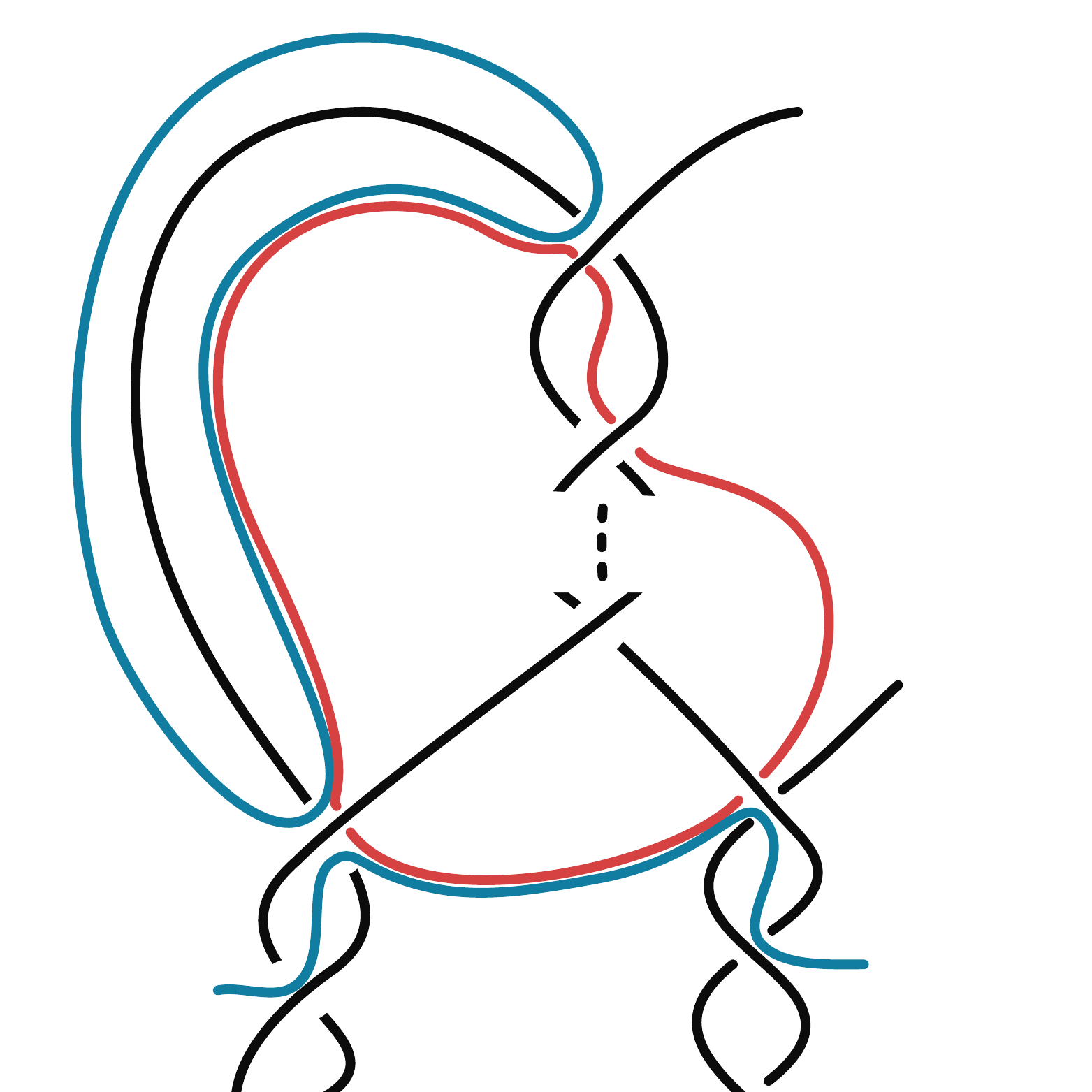}
\put(3,80){$c$}
\put(82,10){$\tilde\kappa$}
\put(79,45){$c'$}
\end{overpic}}
    \caption{The possible cases of $\tild{\kappa} \in \tild{\calK}$.}  
    \label{fig: cases of kappa tilde}
\end{figure}

The curve $\tild c$ must satisfy one of the cases of Lemma \ref{lem: classification of chi'=0}. 
Note that  only Cases (\ref{lem: chi'=0 case tilde}) and (\ref{lem: chi'=0 case two arcs}) 
are applicable.

\vskip7pt

\emph{Case (\ref{lem: chi'=0 case tilde}).} 
In this case $\bbl{\tild c} = 3 $ and $\bdr{\tild c}=2$.
The complementary subarc $\beta=\tilde c \ssm \tild \kappa$ must have two intersection 
points and no bubbles. We rule out each of the cases.
In Case (a) of Figure \ref{fig: cases of kappa tilde}, the arc $\beta$ meets $L$ and 
travels along it in a single arc in $P^+$ which cannot go 
through an underpass. As can be seen from Figure \ref{fig: cases of kappa tilde}, no such arc exists.
In Case (b), the arc $\tild \kappa$ ends in a twist box, and so the complementary arc 
$\beta$ must pass through a bubble. Cases (c) and (d) are 
ruled out similarly to Cases (a) and (b) respectively.

\vskip7pt

\emph{Case (\ref{lem: chi'=0 case two arcs}).} In this case the complementary arc 
$\beta$ is in $\calK$.  The curve $\kappa$ cannot be in $\calK_{4,0}$ or $\calK_{2,1}$ 
as otherwise the curve $c$ would be in $\calC_{7,0}$ or $\calC_{5,1}$. However these sets are empty by 
Lemma \ref{lem: properties of curves}.

\vskip7pt

In Case (a), the endpoints of $\beta$ are the same as those of $\tild\kappa$. If $\beta\in\tild\calK$ 
then it is one of the arcs depicted in Figure \ref{fig: cases of kappa tilde}. Hence, it must also be 
as in Case (a) and must run parallel to $\tild \kappa$. But this would imply that $\tild c$ passes 
through the same bubble twice on the 
same side of $L$, in contradiction to tautness of $S$. If $\beta \in \calK_{3,0}$, then there 
is a corresponding subarc $\alpha$ in $L\cap P^+$ depicted in Figure \ref{fig: subarc parallel 
to kappa'}  that connects the same regions as $\beta$. As Figure \ref{fig: cases of kappa tilde} 
is a precise depiction of the possibilities, one can readily check that there is no such arc.

\begin{figure}[ht]
    \centering
   \begin{overpic}[height=3.5cm]{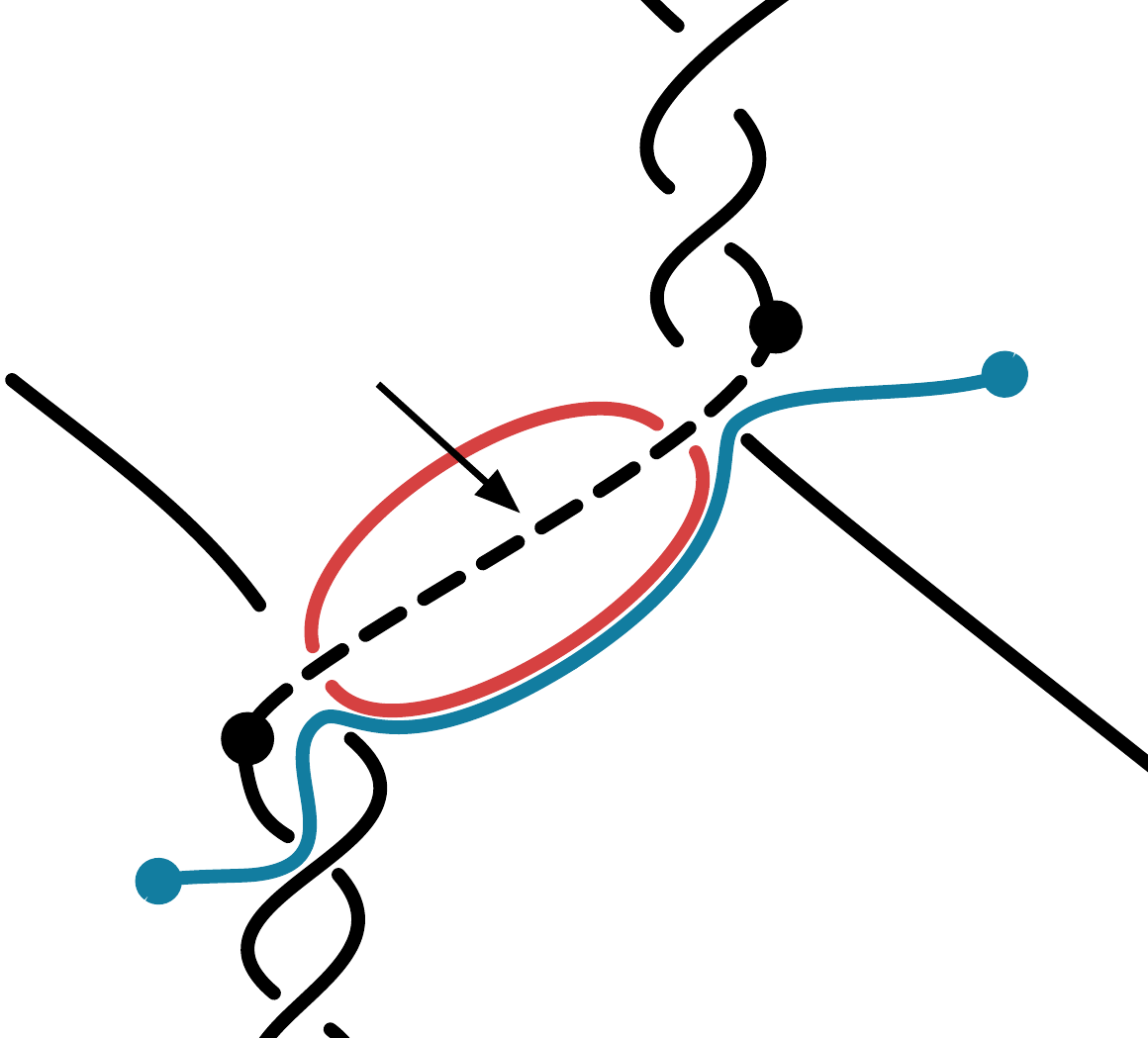}
\put(75,48){$\kappa$}
\put(92,58){$b$}
\put(5,15){$a$}
\put(25,56){$\alpha$}
\end{overpic}
    \caption{The arc $\alpha$ parallel to $\kappa\in\calK_{3,0}$}
    \label{fig: subarc parallel to kappa'}
\end{figure}
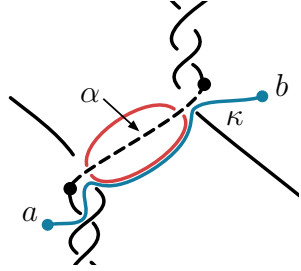

\vskip7pt

In Case (b), 
$\tild\kappa$ ends in the ``middle'' of a twist box after passing through its top bubble. The arc 
$\beta = \tild c \ssm \tild \kappa$ continuing $\tild{\kappa}$ must pass through the bubble 
immediately below.  The arc of $\beta$ intersecting this bubble cannot be contained in any 
arc in $\calK$. Hence, this case is impossible.
This finishes the proof of the claim.
\end{proof}

\begin{lemma}\label{lem: kappa of K_{4,0}}
If $c'$ contains a subarc $\kappa$ of $\calK_{4,0}$, then 
it is of the form shown in Figure~\ref{fig:C6a}.
\end{lemma}

\begin{figure}[ht!]
\centering
\begin{overpic}[height=3.5cm]{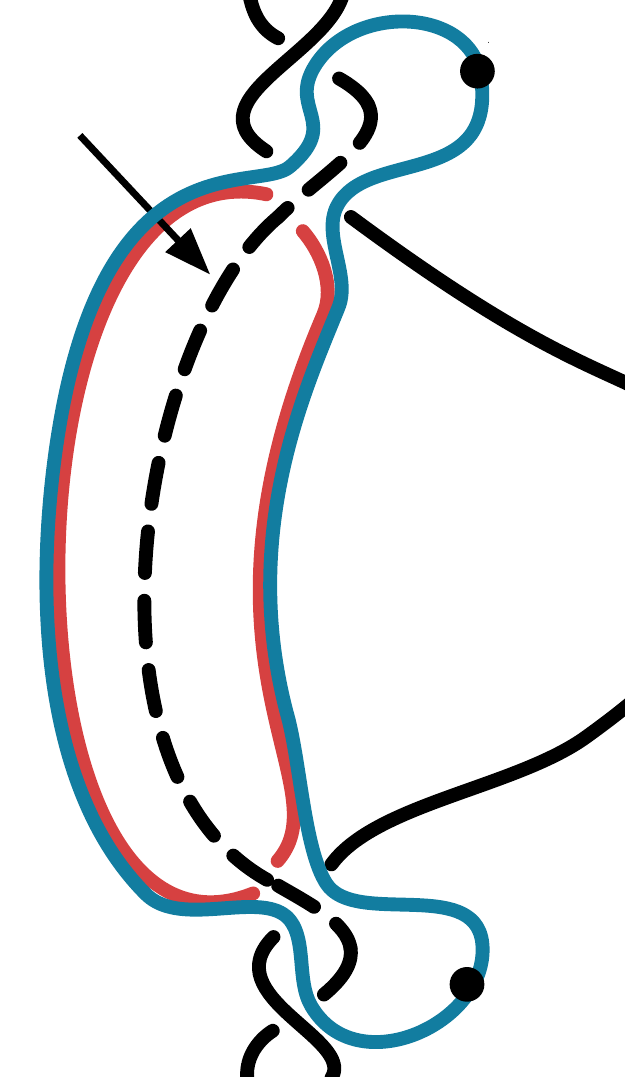}
\put(47,96){$a$}
\put(46,2){$b$}
\put(0,90){$\alpha$}
\end{overpic}
\caption{The unique form of a curve in $S\cap P^+$ containing an arc 
$\kappa$ in $\calK_{4,0}$.}\label{fig:C6a}
\label{fig:  desired itersection 1}
\end{figure}

\begin{proof}
Let $\kappa$ be the subarc of $c'$ in $\calK_{4,0}$, with endpoints $a$ and $b$. As $c$ must have 
$\chi'(c')=0$, it follows from Lemma~\ref{lem: classification of chi'=0} that the complementary 
arc of $\kappa$ in $c'$ 
is either in $\calK_{4,0}$, or passes through exactly two   bubbles, or through two 
intersection points with the link $L$. 

The distance between $a$ and $b$ is 2: It is clearly smaller or equal to $2$. It must be 
even as the regions have the same color and if it is  $0$  the diagram would not be  
twist-reduced. So, by  Lemma~\ref{obs: single arc} the only subarc $\alpha\subset L\cap P^+$ 
connecting the region containing $a$ to the region containing $b$ is the bridge adjacent to 
$\kappa$. See Figure~\ref{fig:C6a}.  
The complement of $\kappa$  in $c'$ cannot be an arc of $\calK_{4,0}$ since any arc of $\calK_{4,0}$ 
connecting these regions must follow $\alpha$ from the same side, contradicting the assumption that
$S$ is taut. The complement of $\kappa$ in $c'$  cannot contain exactly two intersection points, 
as the arc spanning them must be $\alpha$, which again results in a contradiction to the tautness of $S$. 
Thus, the complement of $\kappa$ in $c'$ must contain two bubbles. The only possible such arc is an 
arc on the other side of $\alpha$ adjacent to $\kappa$, as in  Figure~\ref{fig:C6a}. 
\end{proof}

\begin{lemma}\label{lem: kappa in K_{3,0}}
If $c'$ contains a subarc $\kappa\in\calK_{3,0}$ then it is of the form shown in Figure~\ref{fig:C6}.
\end{lemma}

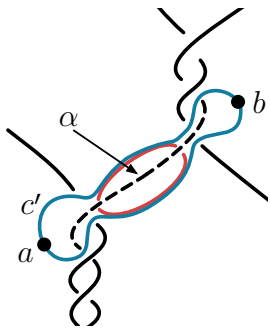
\begin{figure}[ht!]
\centering
\begin{overpic}[width=3.5cm]{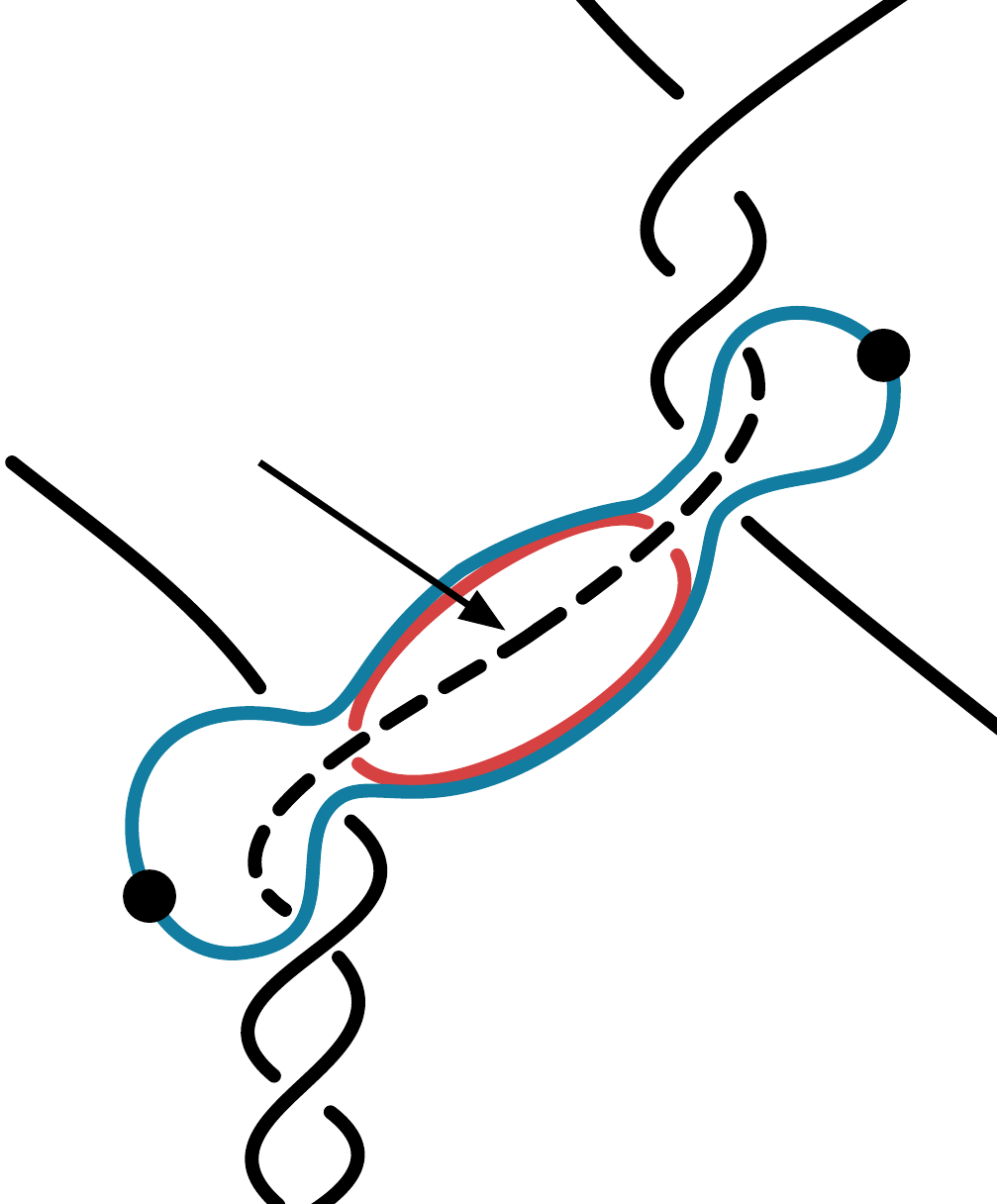}
\put(4,20){$a$}
\put(17,63){$\alpha$}
\put(5,35){$c'$}
\put(78,67){$b$}
\end{overpic}
\caption{The unique form of a curve $c'$ containing an arc of $\calK_{3,0}$.}\label{fig:C6}
\label{fig:  desired itersection 2}
\end{figure} 

\begin{proof}
Let $c'$ be a curve containing a subarc $\kappa\in \calK_{3,0}$. Let $a,b$ be the endpoints of $\kappa$ as in Figure \ref{fig: subarc parallel to kappa'}. As $\chi'(c')=0$, it follows from Lemma~\ref{lem: non-positive contribution of chi'} that the complementary subarc $\beta = c'\ssm \kappa$ is either in $\calK_{3,0}$ (in which case $c'\in \calC_{6,0}$), or 
contains exactly two intersections and no bubbles (in which case $c'\in\calC_{3,2}$).

If $\beta \in \calK_{3,0}$,
let $\alpha$ be the the subarc of $L\cap P^+$ connecting the regions containing $a$ and $b$, adjacent to $\kappa$, as in Figure \ref{fig: subarc parallel to kappa'}.
Similarly, let $\alpha'$ be the subarc of $L\cap P^+$ connecting the regions containing $a$ and $b$, adjacent to $\beta$.
As each of $\alpha,\alpha'$ passes over two crossings and the uniqueness implied by Lemma~\ref{obs: single arc}, $\alpha=\alpha'$.
Since $S$ is taut, $\kappa$ and $\beta$ must be two different sides of $\alpha$, resulting in the configuration depicted in Figure \ref{fig:C6}.

Next suppose $\beta$ has two intersection points and no bubbles. Note that $\kappa$ must pass through
two bubbles in one twist box and through a single bubble in another. Let $T$ be the second twist box. 
The arc $\beta$ starts from $b$ and first meets $L$ at some intersection point $p$.
The two endpoints of $\beta$ must belong to regions of different colors since its complement in $c'$ passes through three bubbles. In particular $\beta$ cannot connect the two regions to the right and left of the twist box $T$. It follows that the point $p$ cannot be any of the points depicted by small empty squares in Figure \ref{fig:C32}(a), as otherwise the arc $c'\cap L\subseteq \beta$ would be one of the arcs in $P^+\cap L \cap T$ connecting its right and left sides. The point $p$ also cannot be the point depicted by a small empty circle as otherwise $S$ is not taut.  

\begin{figure}[ht!]
\centering
\subfigure[]{
\begin{overpic}[width=4cm]{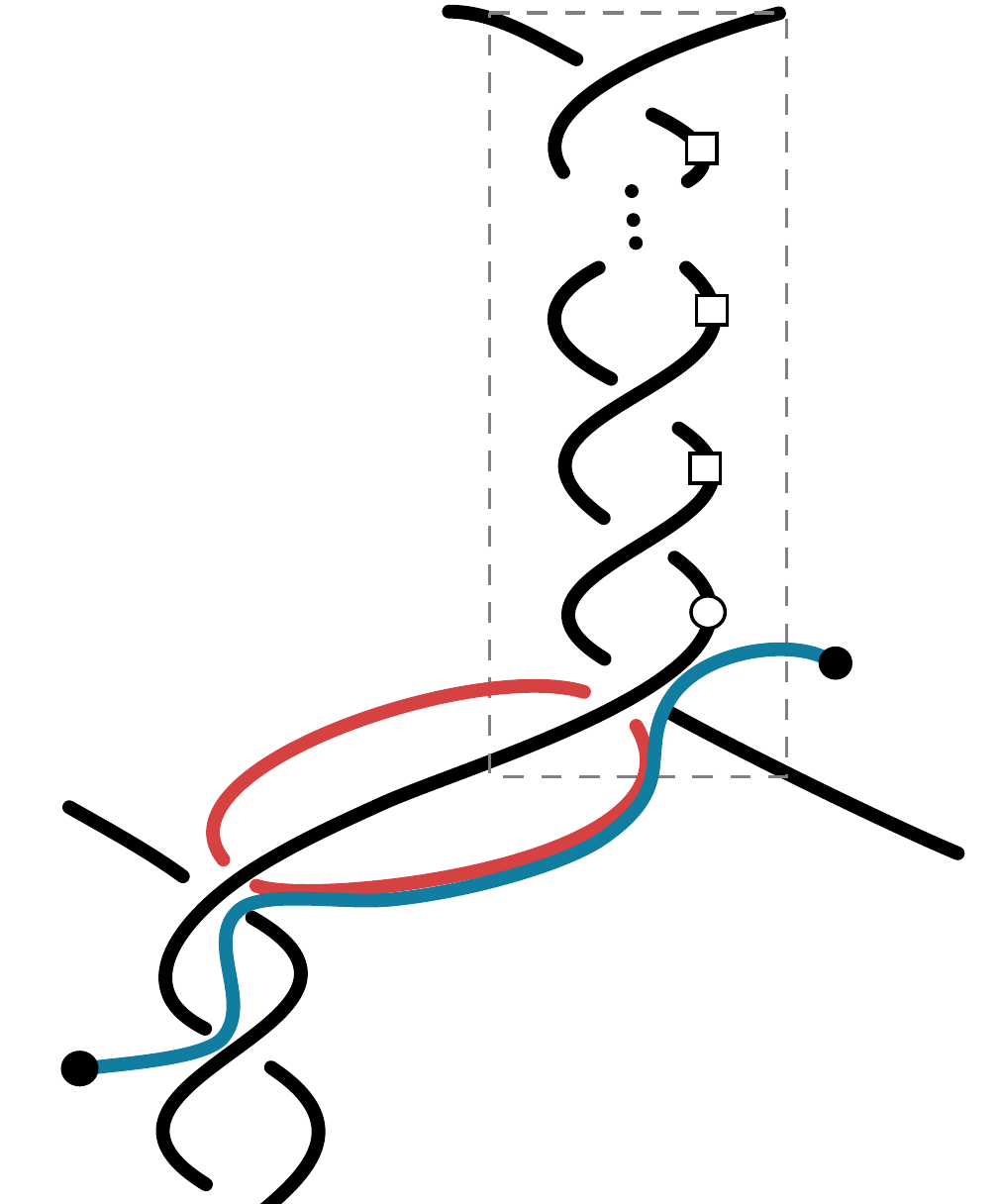}
\put(70,47){$b$}
\put(32,80){$T$}
\put(52,80){$\vdots$}
\put(40,18){$c'$}
\end{overpic}
}
\hskip1cm
\subfigure[]{
\begin{overpic}[width=4cm]{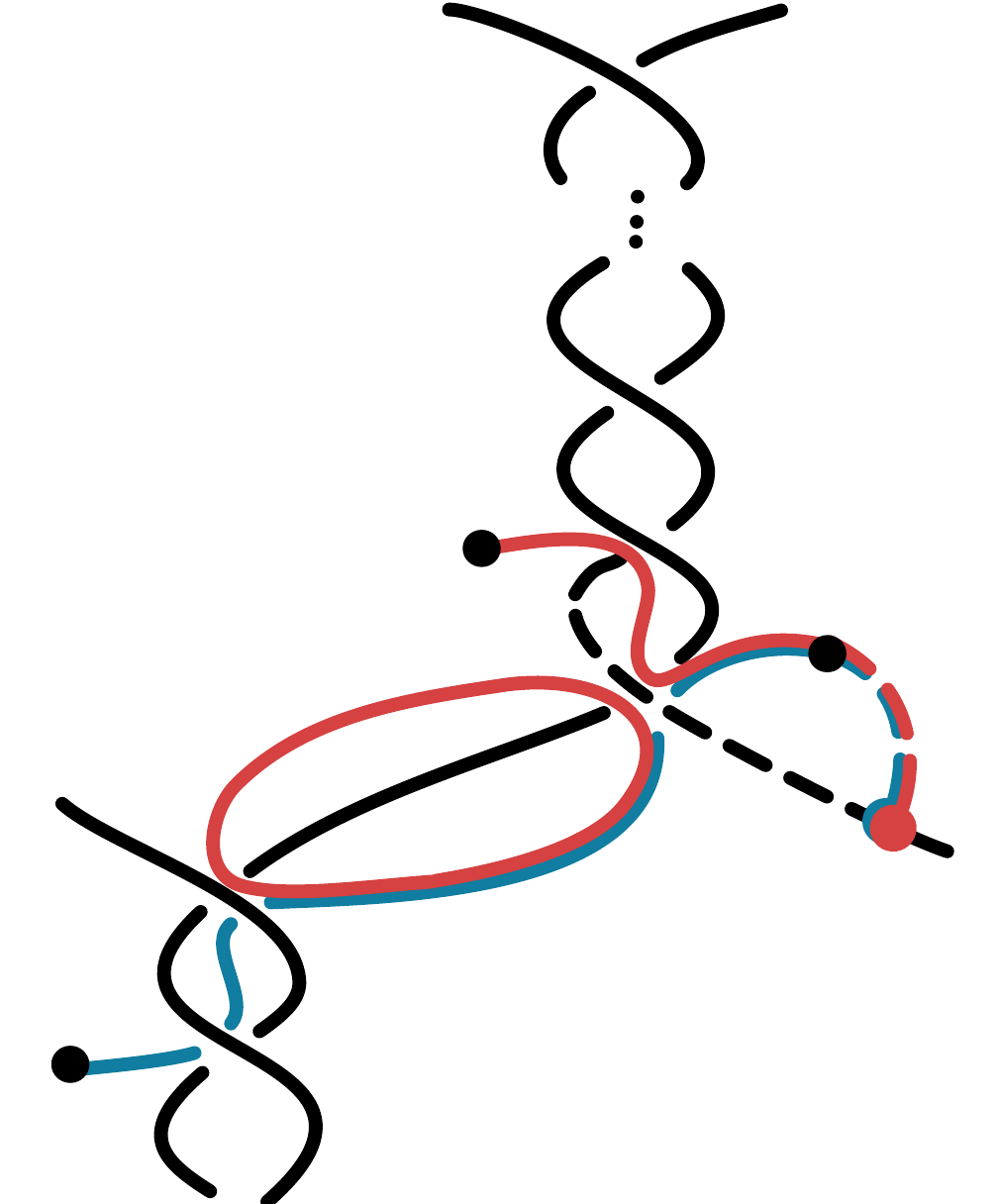}
\put(34,52){$q$}
\put(52,80){$\vdots$}
\put(70,47){$b$}
\put(72,24){$p$}
\put(40,57){$\tilde c$}
\put(40,20){$c'$}
\end{overpic}
}
\caption{A possible configuration of $c'\in \calC_{3,2}$ on $P^+$ that has an arc of $\calK_{3,0}$: (a) The side $P^+$, the twist box $T$ and the forbidden positions for the point $p$. (b) The side $P^-$, the opposite curve $\tilde{c}$, and one of the possible positions for the point $p$. }\label{fig:C32}
\end{figure}
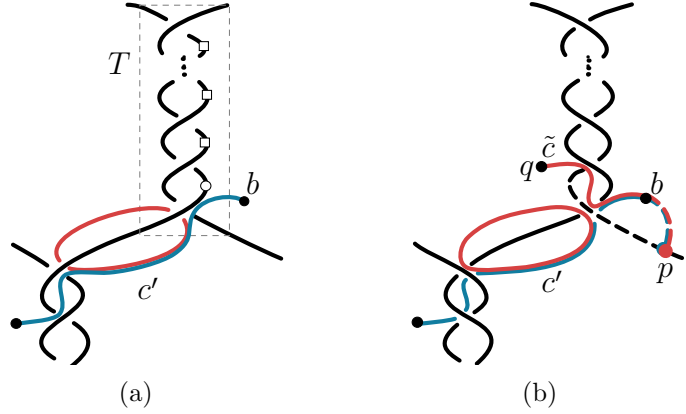

Assume that $c'$ is a curve on $P^+$.
Consider the curve $\tilde c$ opposite to $c'$ at $b$ on $P^-$. See Figure \ref{fig:C32}(b). 
The curve $\tilde c$ must pass through two bubbles in $T$ and through the intersection point $p$. Denote a point on $\tilde c$ just beyond this second bubble by $q$ (see Figure \ref{fig:C32}(b)). 
The points $b$ and $q$ are on opposite sides of the twist box $T$.
The curve $\tilde c$ does not contain any $\eta$ in $\calK$: If it does, then $\eta$ will not pass through the two bubbles or intersection point specified above; thus, $\bbb{\tilde c}\ge \bbb{\eta}+3$ in contradiction to Lemma \ref{lem: classification of chi'=0}.

Therefore, as $\tilde c$ passes through at least two bubbles and two intersection points and contains no arc of $\calK$, by Lemma \ref{lem: classification of chi'=0} it must be $\tilde c\in \calC_{2,2}$. Thus,
$\tilde c\cap L$ is an arc of $L\cap P^-$ connecting the regions containing $b,q$ and ending at the point $p$. 
Since $\tilde c\cap L$ connects the two regions to the left and right of a twist box it must be an arc that passes through the twist box.
Since $\tilde{c}$ is opposite to  $c'$, the point $p$ cannot be any of the points depicted by small empty squares and circles in Figure \ref{fig:C32}(a). This leaves only one possible such arc, depicted by a black dotted line in Figure \ref{fig:C32}(b). 
However, if $\tilde c$ contains such an arc, this would imply that $S$ is not taut in contradiction to the assumption on $S$.
\end{proof}

\begin{lemma}\label{lem: no kappa in K_2,1}
There is no curve $c'\in\calC$ containing a subarc $\kappa$ in $\calK_{2,1}$
\end{lemma}

Note that, in general, if  $c\in\calC_{1,2}$ then the arc $\alpha=c\cap L$ is a bridge. If $\alpha$ contains no crossings then it corresponds to Figure \ref{reducing bubbles a}, which would imply that $S$ is not taut. 
Thus, $\alpha$ passes through either one or two crossings.

\begin{proof}
Assume, for contradiction, that $c'\in\calC$ is a curve containing a subarc $\kappa\in\calK_{2,1}$.
By Lemma \ref{lem: classification of chi'=0} there are two cases: Either $c'$ is a union of two arcs in $\calK_{2,1}$, or $c'\in \calC_{3,2}$ and it contains one arc in $\calK_{2,1}$.

Assume first that $c'$ is the union of exactly two arcs $\kappa_1,\kappa_2\in \calK_{2,1}$.
It follows that $c'\in \calC_{4,2}$, and can be viewed as the union of three arcs 
$\alpha',\beta_1,\beta_2$, where $\alpha'=c'\cap L$ and $\beta_1,\beta_2$ are subarcs of 
$\kappa_1,\kappa_2$ respectively each passing through two bubbles. Let $c_i$, $i=1,2$, be the 
curve in $C_{1,2}$ opposite to $\kappa_i$, and let $\alpha_i=c_i \cap L$.  

Assume first that $\alpha'$ does not pass over any crossing.   Thus, $\beta_1,\beta_2$ emanate from $\alpha'$ 
into adjacent regions with opposite colors. Each of them passes through two bubbles, and hence 
they cannot meet to form a closed curve.

Next assume that $\alpha'$ passes over two crossings. Then, $\alpha_1$ (and $\alpha_2$) passes over a crossing which is not the top or bottom in its twist box, contradicting Remark \ref{rem: c12 passes through top or bottom}. See Figure \ref{fig: two twist boxes and alpha}. 

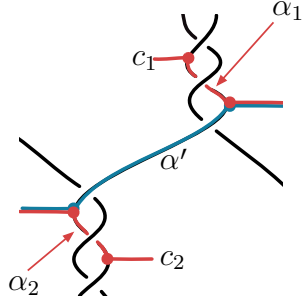
\begin{figure}[ht!]
\centering
\begin{overpic}[width=3.5cm]{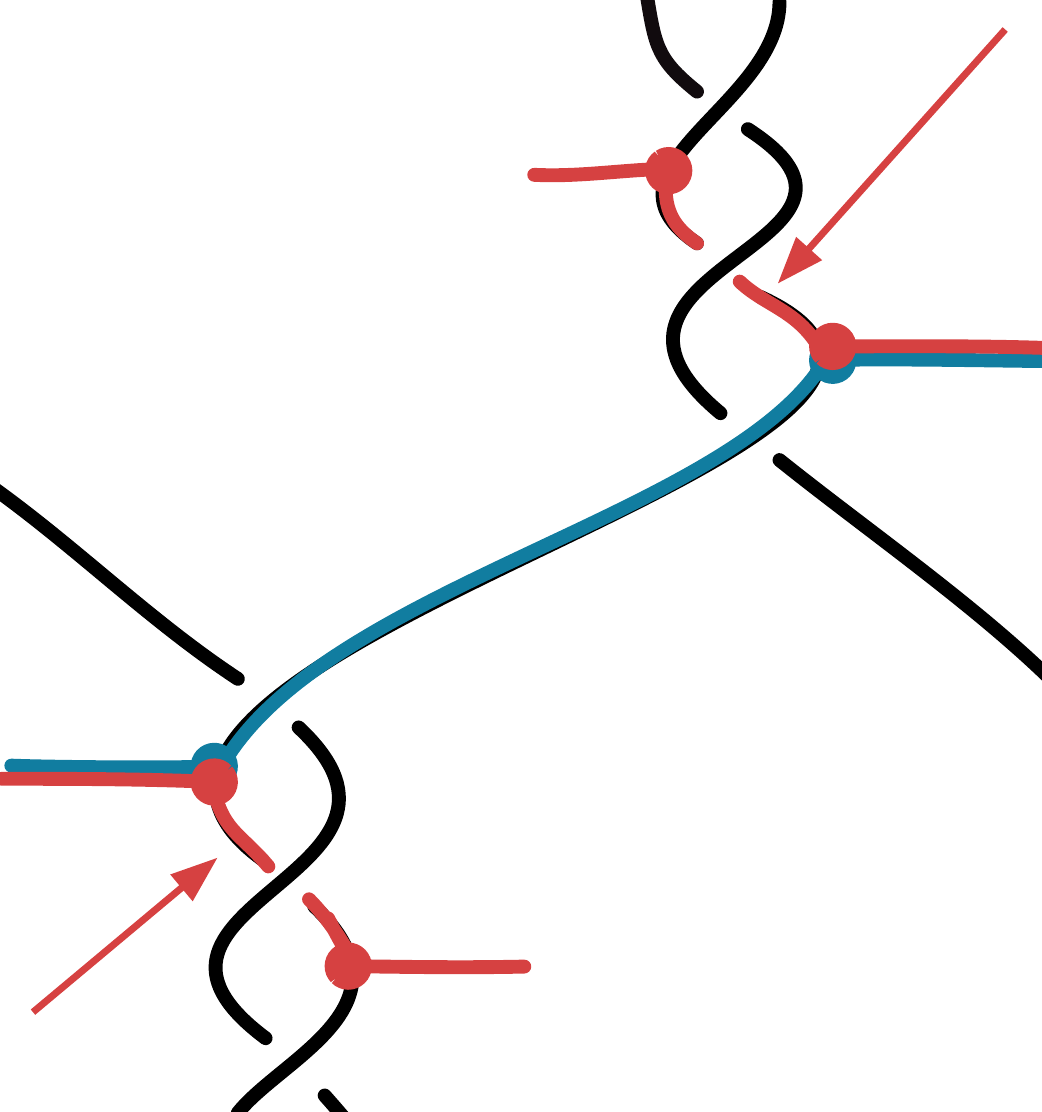}
\put(50,45){$\alpha'$}
\put(90,100){$\alpha_1$}
\put(40,82){$c_1$}
\put(-4,2){$\alpha_2$}
\put(50,11){$c_2$}
\end{overpic}
\caption{The arc $\alpha_1\cup\alpha'\cup \alpha_2$.}
    \label{fig: two twist boxes and alpha}
\end{figure}


Finally, assume $\alpha'$ goes through a single crossing. By Remark \ref{rem: c12 passes through top or bottom}, $\alpha_1$ (and $\alpha_2$) must pass over the top or bottom crossings of a twist box. The only possible configuration is if $\alpha'$ passes over the middle crossing of a twist box with three crossings. The fact that $c_1$ and $c_2$ pass through bubbles at the top or bottom of two other twist boxes, determines the orientations of the two adjacent twist boxes as depicted in Figure \ref{fig:C12a}.
However, as can be readily checked, a plat diagram $L$ cannot contain a subdiagram as in the figure.

\begin{figure}[ht!]
\centering
\begin{overpic}[width=7cm]{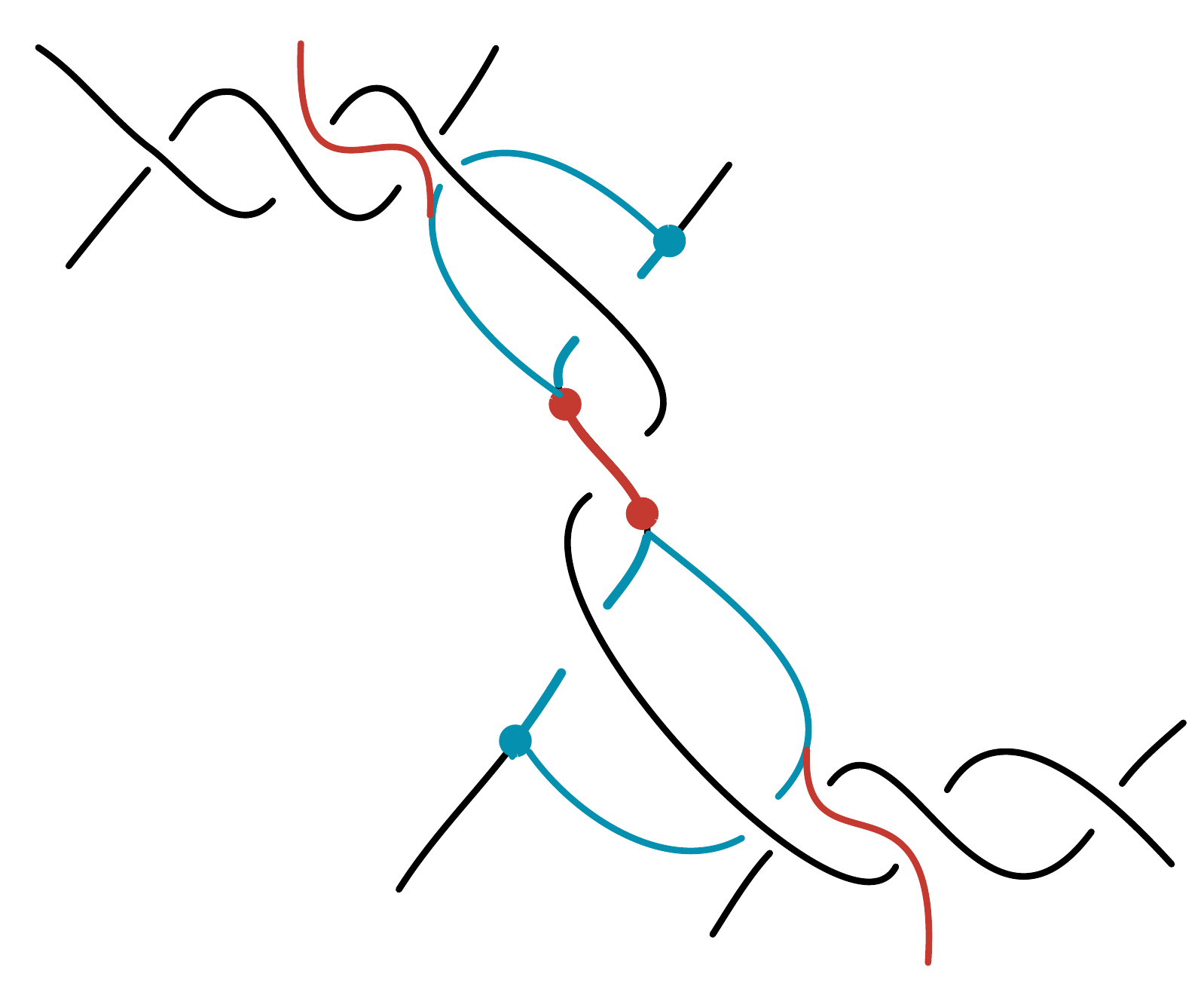}
\put(27,40){$\alpha'$}
\put(53,56){$\alpha_1$}
\put(39,27){$\alpha_2$}
\put(46,72){$c_1$}
\put(46,10){$c_2$}
\end{overpic}
\caption{}\label{fig:C12a}
\end{figure}

Now, assume that $c'\in \calC_{3,2}$ and contains exactly one arc $\kappa\in\calK_{2,1}$. Let $c\in \calC_{1,2}$ be the curve opposite to $\kappa$.

Assume first that $\alpha=c\cap L$ passes through two crossings. Since $c\ssm\alpha$ passes through 
exactly one bubble, the regions containing the endpoints of $\alpha$ are at distance 1.
Such a bridge $\alpha$ exists only in the corners of the diagram, in one of the configurations 
depicted in Figure \ref{fig:C12corners}. The corresponding curve $c'$ in each of the possible 
configurations, contains an arc $\eta$ as in the figure,  which prolongs 
$\kappa$. The endpoints of $\eta$ are at distance greater than one, and hence 
$c' \ssm \eta$ cannot pass through only one bubble, in contradiction to the assumption.

\begin{figure}
    \centering
   \begin{overpic}[height=4cm]{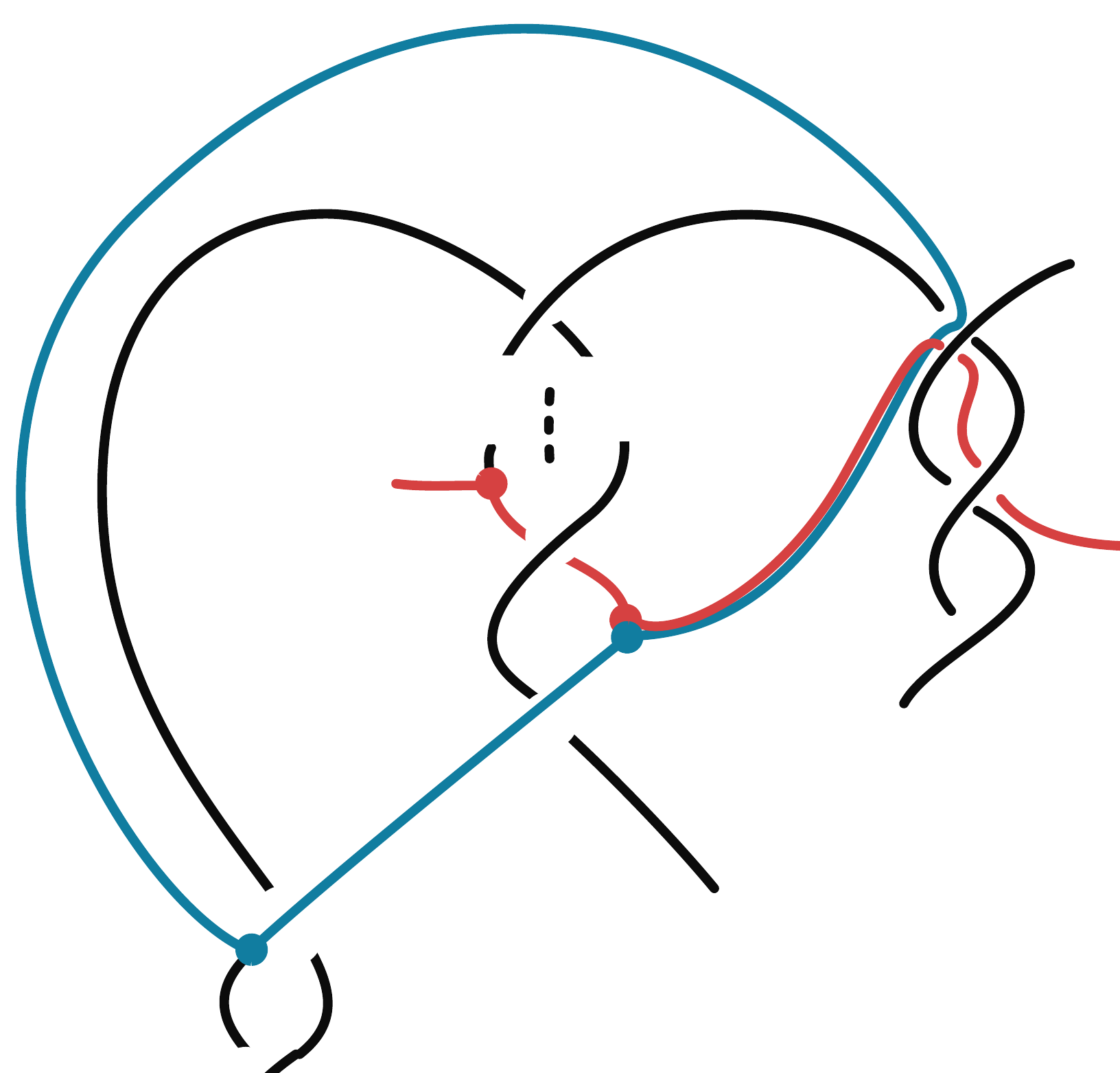}
\put(5,80){$c$}
\put(102,46){$\eta$}
\end{overpic}
\hskip 1cm
   \begin{overpic}[height=4cm]{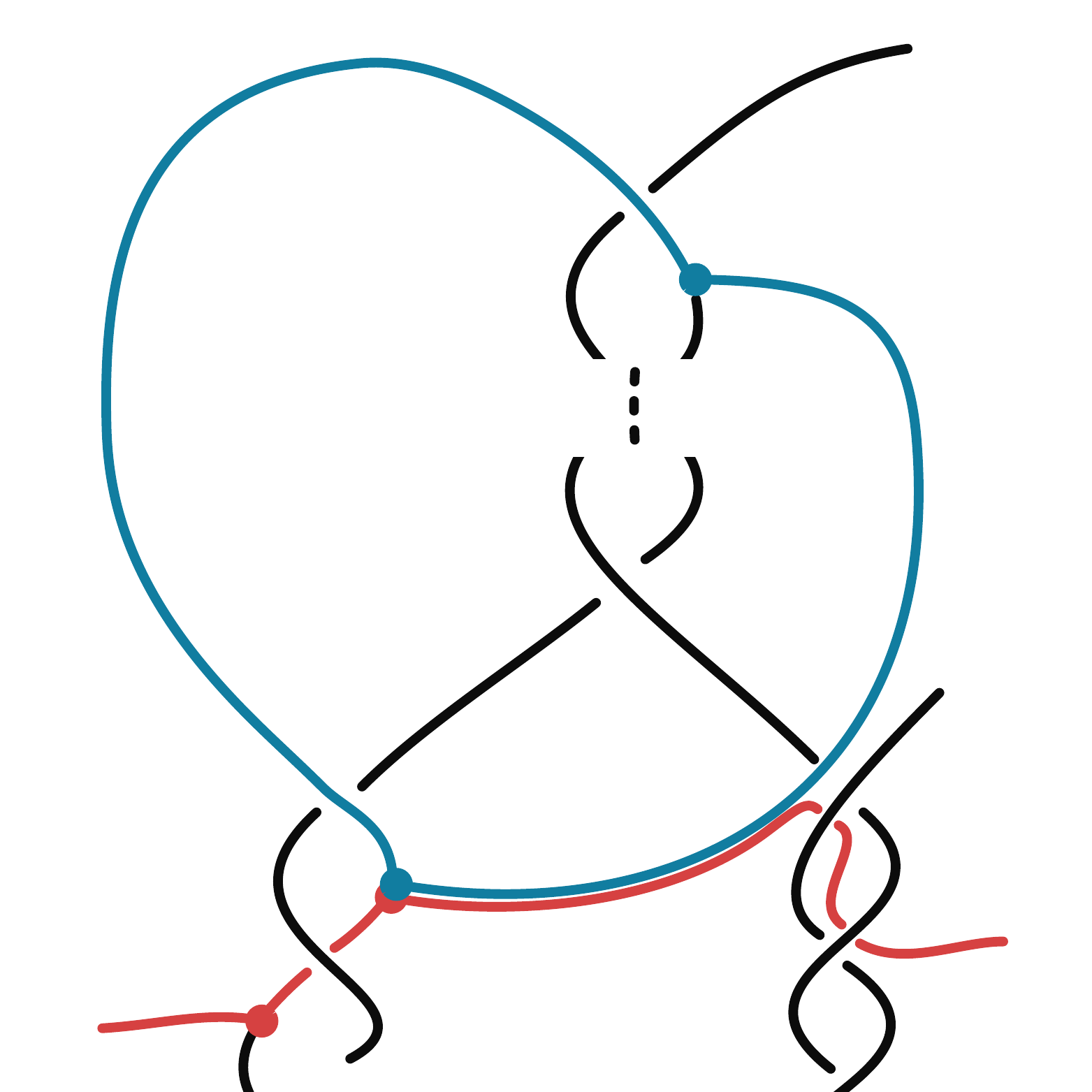}
\put(5,80){$c$}
\put(96,14){$\eta$}
\end{overpic}
\caption{The two possibilities for a curve $c\in\calC_{1,2}$ with a bridge passing through two
crossings.}\label{fig:C12corners}
\end{figure}

Next assume that $\alpha$ passes through one crossing. Then let $\eta$ be small prolongation of 
$\kappa \cup (c' \cap L)$ in $c'$. The arc $c' \cap L$ contains zero, one or two crossings.

If it contains two crossings then the arc $c\cap L$ passes over a crossing which is not the top or bottom of its twist box, in contradiction to Remark \ref{rem: c12 passes through top or bottom} similar to Figure \ref{fig: two twist boxes and alpha}.

If $c'\cap L$ contains one crossing. The curve $c$ bounds a disk $D$ on the plane, containing a segment of $L$ which connects two twist boxes. 
The five possibilities for such a curve $c$ are depicted in Figure \ref{fig: possibilities of K21 and one crossing}. The subarc $\eta$ of $c'$ is the arc depicted in blue in the figure.
In cases (a) and (b), the endpoints of $\eta$ are in regions with the same color, contradicting the assumption that $c\ssm \eta$ passes through a single bubble.
In case (d) and (e), the arc $\kappa$ cannot be continued to an arc $\eta$ such that $\eta\cap L$ has one crossing because of a conflict in orientation, hence (d) and (e) do not occur.
We are left with case (c). In this case $\eta$ can be closed to form $c'$ as shown in figure. However, the arc $\gamma$ opposite to $c'$, at the bottom of the figure, passes through two bubbles and two intersection points, none of which belong to an arc of $\calK_{2,1}$. By Lemma \ref{lem: classification of chi'=0}, the curve containing $\gamma$ must close up with no further bubbles or intersection points, which is impossible.

\begin{figure}[ht]
\subfigure[]{%
\begin{overpic}[width=2.5cm]{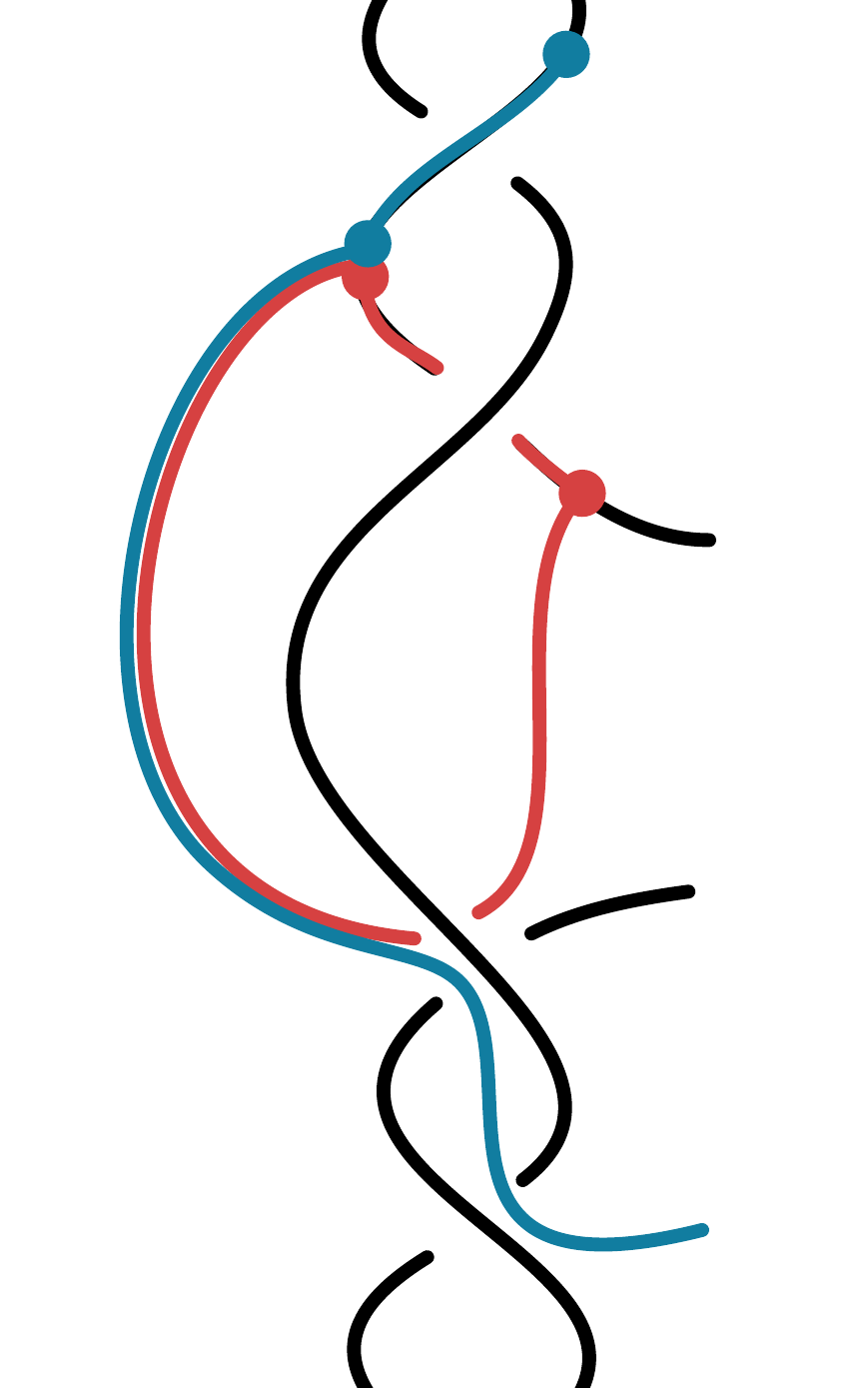}
\put(40,50){$c$}
\put(45,3){$\eta$}
\end{overpic}
}
\subfigure[]{%
\begin{overpic}[width=2.5cm]{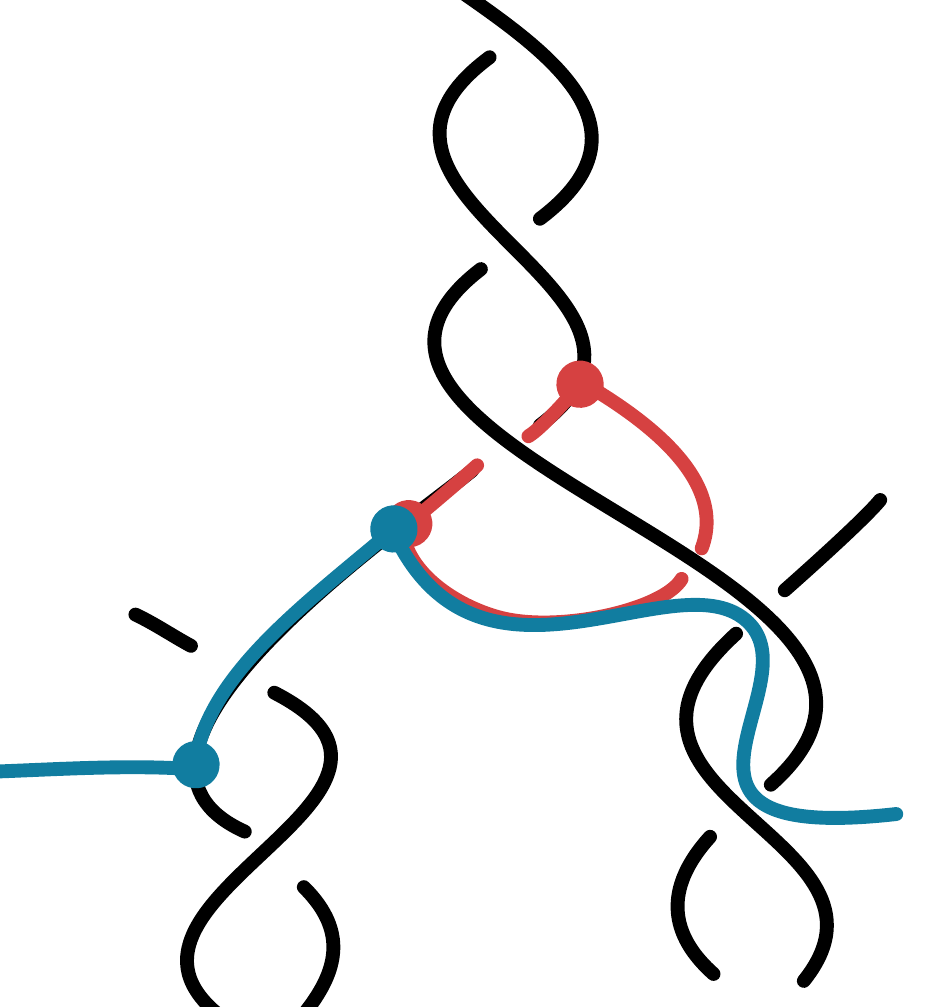}
\put(70,52){$c$}
\put(0,10){$\eta$}
\end{overpic}
}
\subfigure[]{%
\begin{overpic}[width=2.5cm]{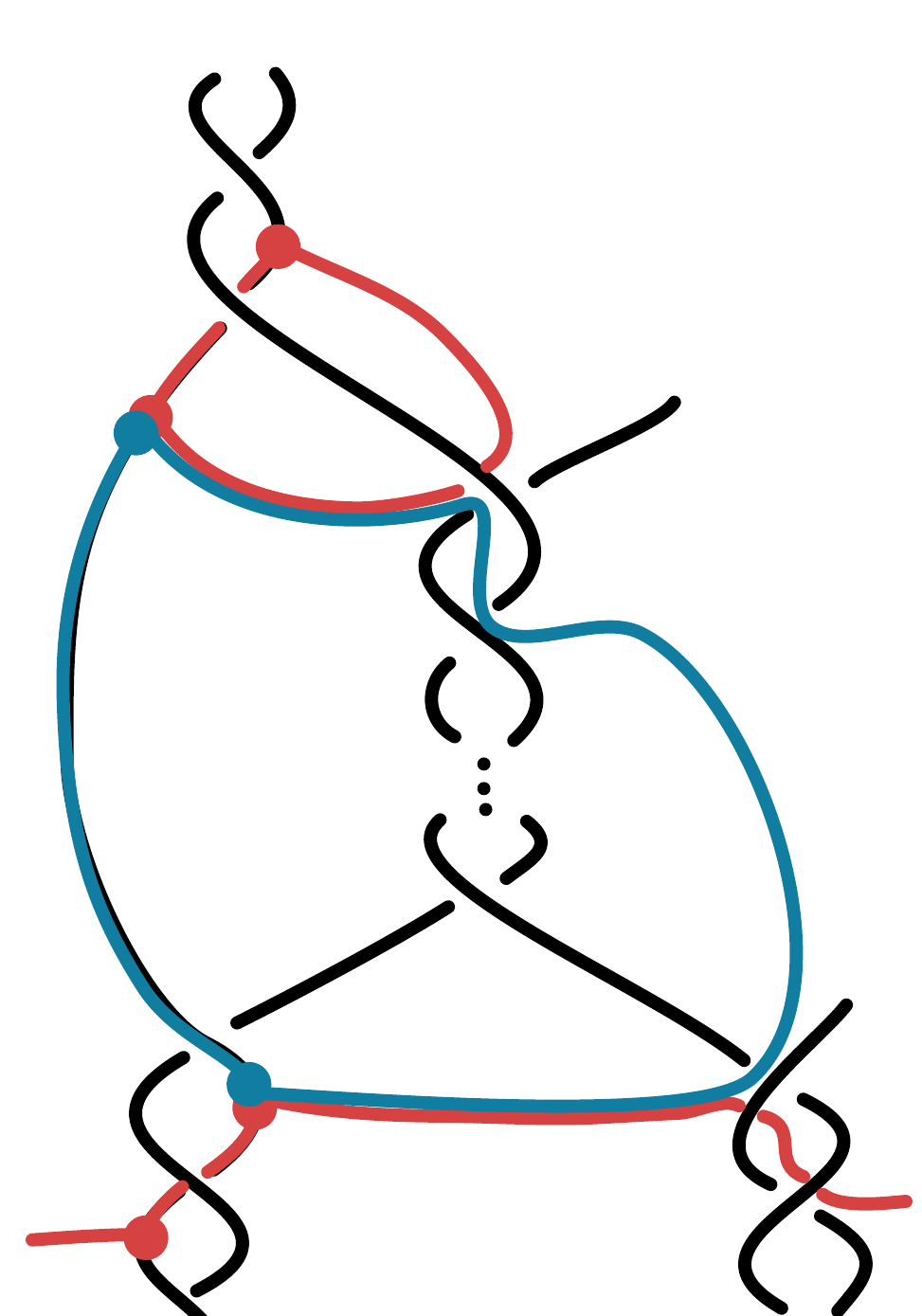}
\put(30,5){$\gamma$}
\put(30,82){$c$}
\put(8,40){$c'$}
\end{overpic}
}
\subfigure[]{%
\begin{overpic}[width=2.5cm]{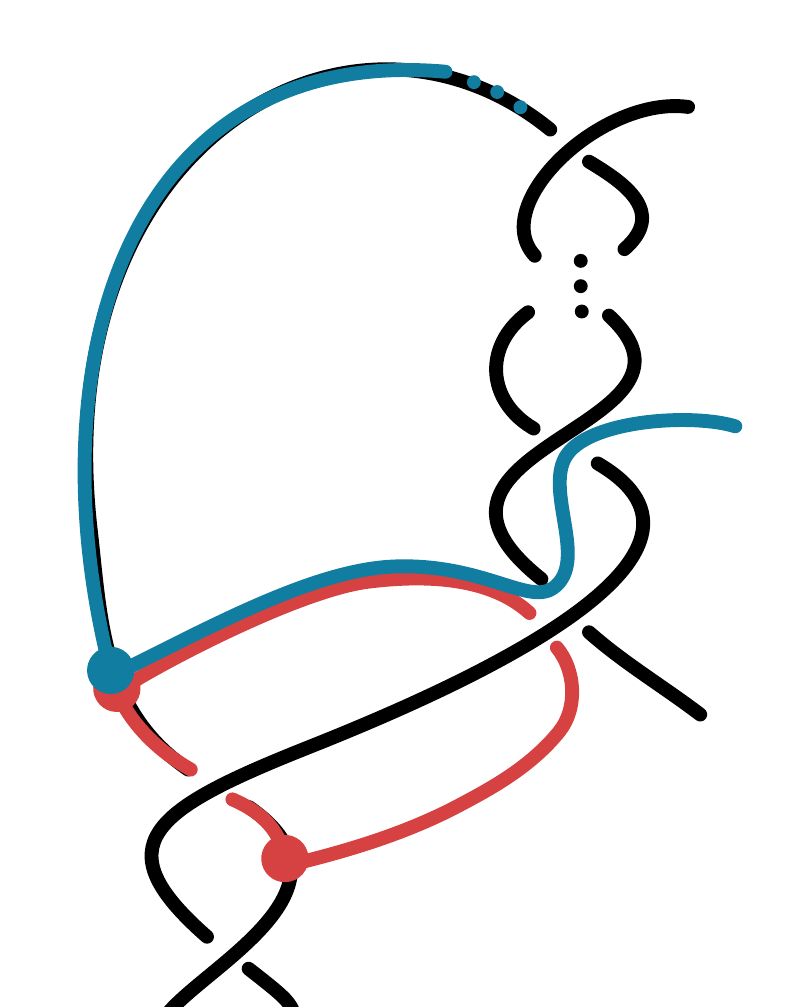}
\put(35,8){$c$}
\put(73,60){$\kappa$}
\end{overpic}
}
\subfigure[]{%
\begin{overpic}[width=2.5cm]{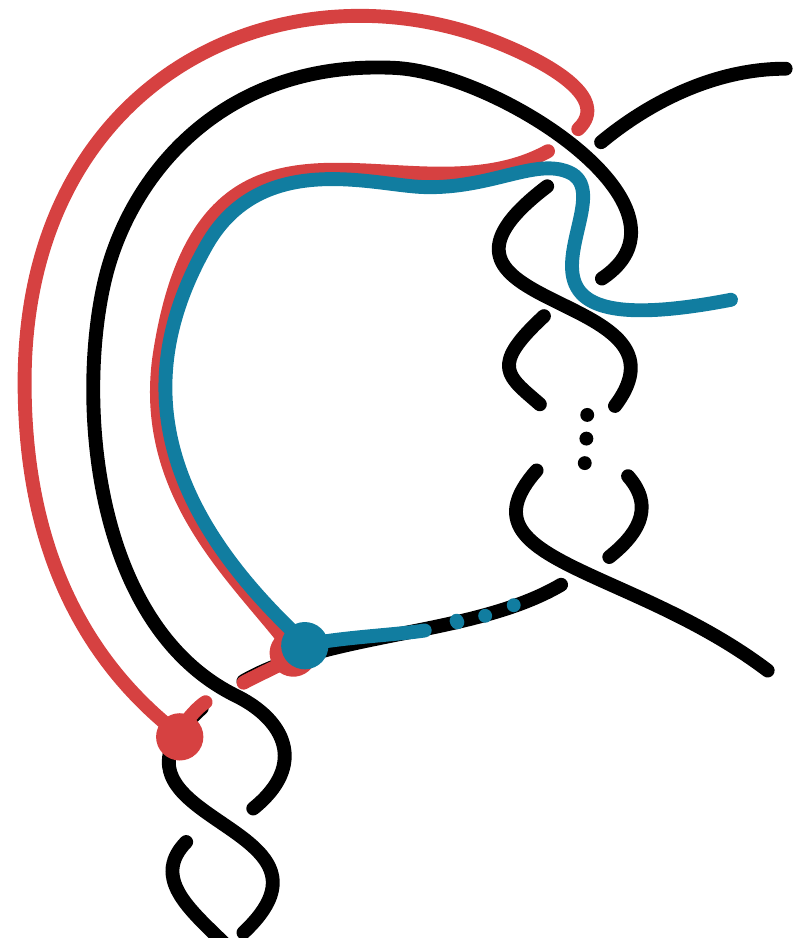}
\put(30,100){$c$}
\put(80,64){$\kappa$}
\end{overpic}
}
\caption{The five possibilities for a curve $c\in\calC_{1,2}$ (in red) and the subarc $\eta$ of the curve $c'$ (in blue) opposite to $c$ such that $c\cap L$ passes under one
crossings. Note that case (a) can occur also on the top boundary of the plat.}
\label{fig: possibilities of K21 and one crossing}
\end{figure}
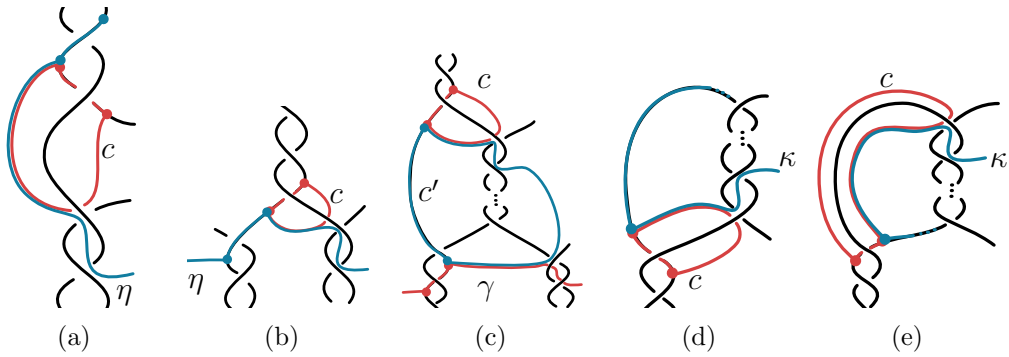

Thus, suppose $c' \cap L$ contains no crossings. In this case, the endpoints of $\eta$, shown in blue in Figure \ref{fig: C12 bridge}, are in the same two regions as the endpoints of the maximal bridge (dashed) passing above the curve $c$. 
As $c'$ contains one additional bubble the endpoints of the bridge are at distance 1 and thus the bridge 
is at a corner of the plat. In each corner, there are two possible such bridges, and on each of these the curve 
$c$ can be oriented in two ways. Thus, altogether there are four possibilities that are depicted in Figure \ref{fig:C12corners2}.

\begin{figure}[ht]
    \centering
\begin{overpic}[width=4cm]{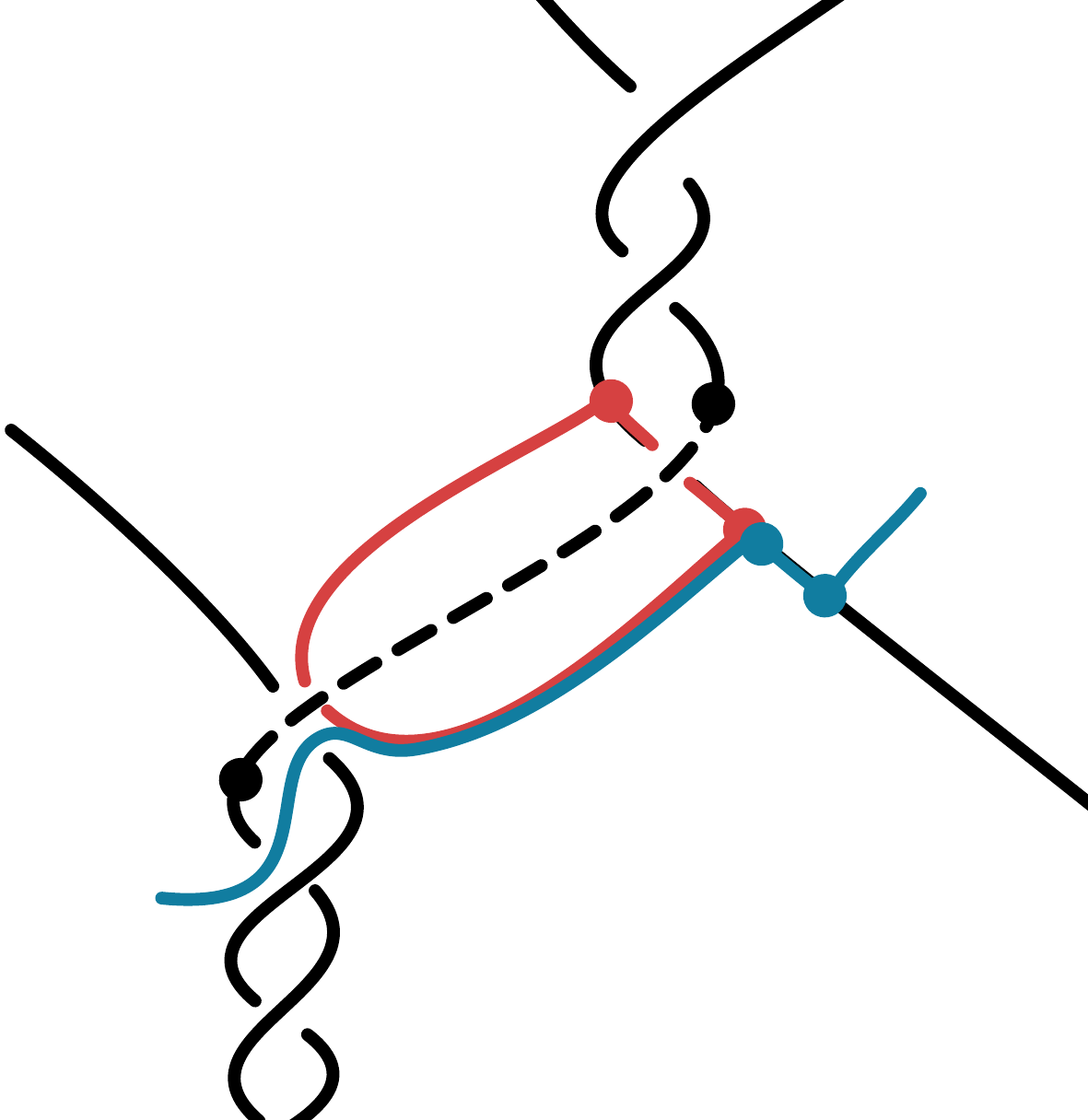}
\put(40,60){$c'$}
\end{overpic}
    \caption{The bridge of a curve $c\in C_{1,2}$ such that $c\cap L$ has one crossing.}
    \label{fig: C12 bridge}
\end{figure}

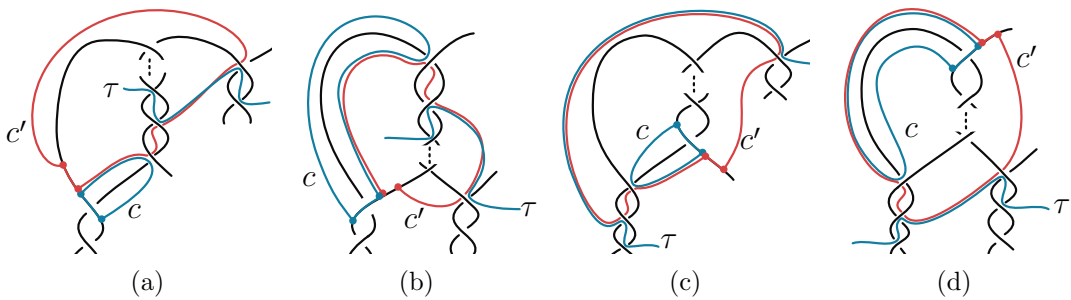
\begin{figure}[ht]
\subfigure[]{%
\begin{overpic}[width=3.3cm]{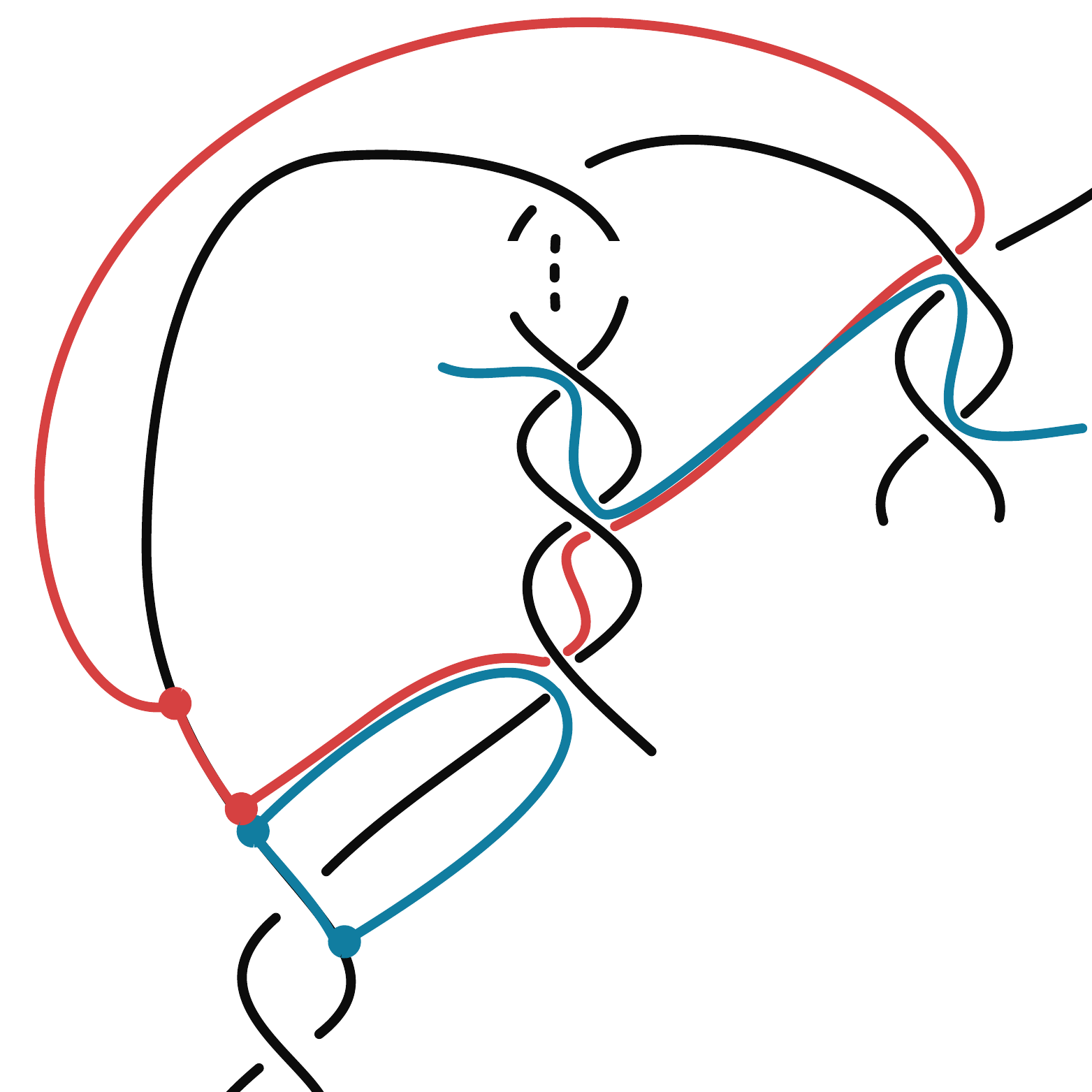}
\put(42,14){$c$}
\put(-7,45){$c'$}
\put(32,63){$\tau$}
\end{overpic}
}
\subfigure[]{%
\begin{overpic}[width=3.3cm]{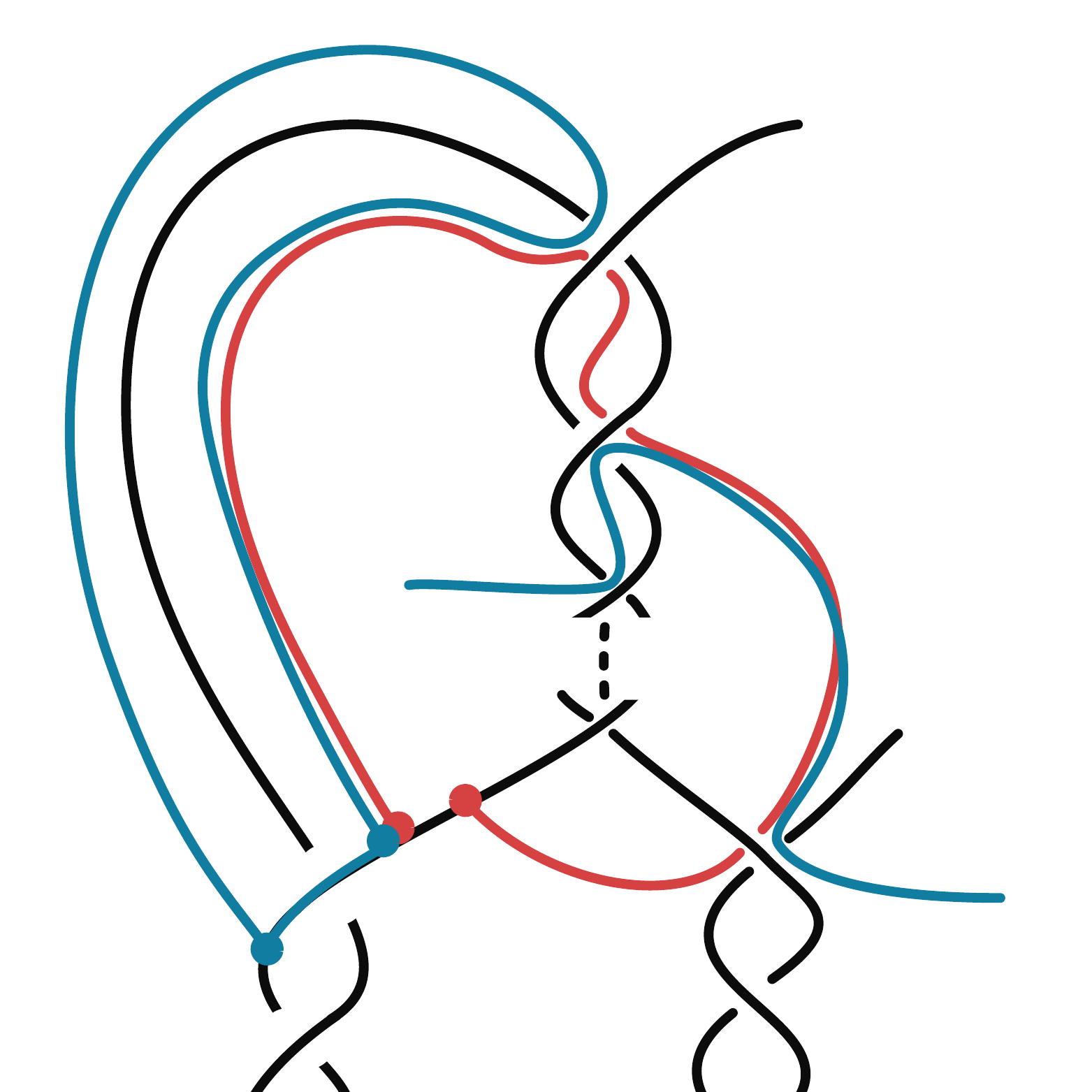}
\put(4,27){$c$}
\put(45,10){$c'$}
\put(92,15){$\tau$}
\end{overpic}
}
\subfigure[]{%
\begin{overpic}[width=3.3cm]{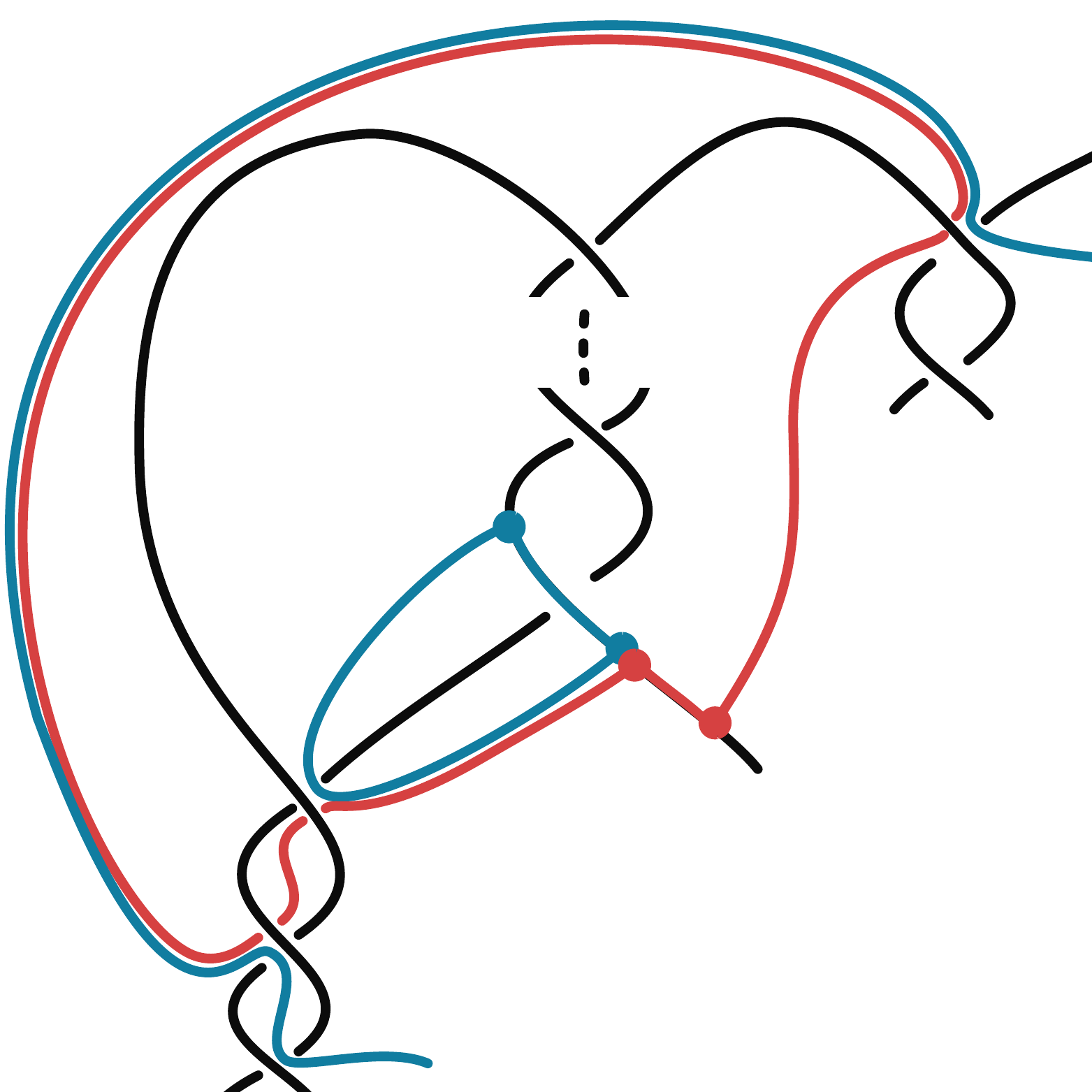}
\put(30,46){$c$}
\put(72,42){$c'$}
\put(40,2){$\tau$}
\end{overpic}
}
\subfigure[]{%
\begin{overpic}[width=3.3cm]{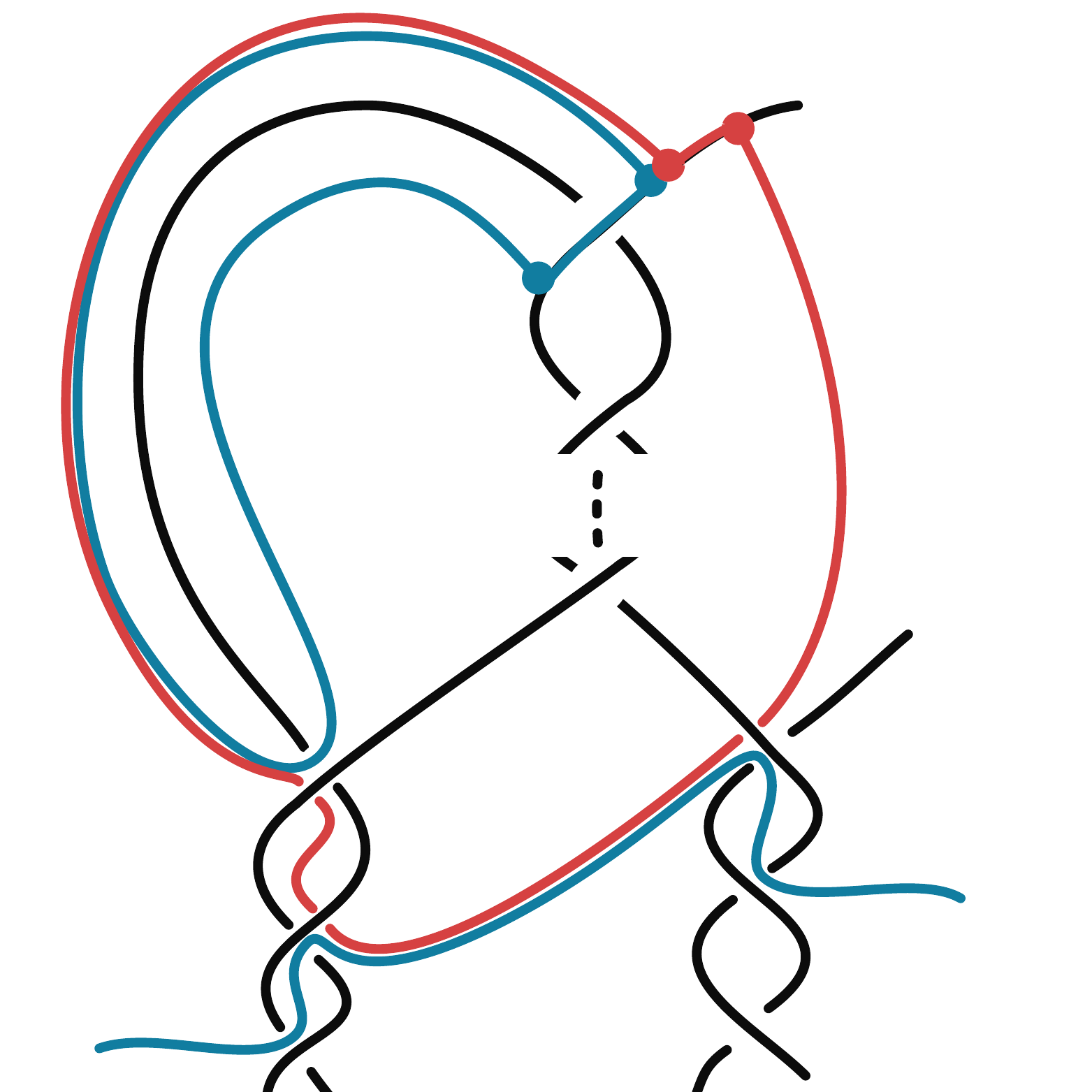}
\put(30,47){$c$}
\put(75,75){$c'$}
\put(90,16){$\tau$}
\end{overpic}
}
\caption{The four possibilities for a curve $c\in\calC_{1,2}$ with a bridge passing through zero
crossings.}
\label{fig:C12corners2}
\end{figure}

In each of these possibilities, there is an arc $\tau$ opposite to $c'$, so that $\tau$ 
has three or four bubbles and its endpoints are at distance at least two. Thus, the closed curve 
containing $\tau$ cannot be in $\calK$ and thus must have  $\chi'<0$.
\end{proof}

\begin{definition}\label{def: cross ans turn}
If a curve enters and exists a twist box from the left or right side edges then we say that it $ \it{crosses}$ the twist box. Otherwise, 
if it enters from a side edge and exists from a top/bottom edge or vice versa we say that it $\it{turns}$ at the twist box.
\end{definition}

\begin{lemma}\label{lem: No C22}
There are no curves in $\mathcal{C}_{2,2}$. 
\end{lemma}

\begin{proof}
For any curve $c$ in $\mathcal{C}_{2,2}$.
Consider the subarcs $\alpha=c\cap L$ and $\beta=c\ssm \alpha$.

The two endpoints of a small extension of $\alpha$ on $c$ must be contained in regions of the same color as the rest of $c$ contains exactly two bubbles. As $S$ is taut, these 
points cannot be in the same region. Thus the arc $\alpha$ passes over at least one 
crossing of the projection diagram.

\begin{claim} 
If $\mathcal{C}_{2,2}$ is nonempty then there exists a curve $c\in\mathcal{C}_{2,2}$ such that $\alpha$ passes over one crossing, and $\beta$ crosses 
one twist box.
\end{claim}

\begin{proof}[Proof of claim]\hfill

\noindent \underline{Case 1.} Assume first, that $\alpha$ passes over two crossings. Following $\beta$ from one of its 
endpoints, $p,q$, we meet a twist box either from its side or its top/bottom. 

If, following $\beta$ from either endpoint, we meet a twist box from its top/bottom, then $\beta$ meets two different twist boxes. 
Consider the curve $c'$ opposite to $c$ sharing the arc of $\beta$ between the bubbles. The curve 
$c'$ crosses the same two twist boxes. The diagram is twist-reduced and thus $c'\notin \calC_{4,0}$ as otherwise it would bound a reducing subdiagram. 
Therefore, we must have $\chi'(c')<0$, a contradiction. 

Thus, following $\beta$ from one of its endpoints, say $p$, $\beta$ meets a twist box $T$ from the side. We claim that one of the two curves opposite to $c$ containing $p$ or $q$ must cross $T$:
Let $c'$ be a curve opposite to $c$ which contains $p$. If $c'$ crosses $T$, we are done. Thus assume that  $c'$ turns at $T$ and exists through, say, its \emph{bottom}. This implies that $\beta$ must cross $T$ and pass through its two bottom bubbles. 
Therefore, the curve $c''$ opposite to $c$ containing $q$ crosses $T$ and passes through its second and third bubbles (counted from the bottom). 

Let $c'$ be a curve opposite to $c$ containing an endpoint of $\beta$ and crossing $T$, as in the previous paragraph.
As $c'$ passes through at least two bubbles and has at least two intersection points, it must be in $\calC_{2,2}$ by Lemma \ref{lem: classification of chi'=0} and Lemma \ref{lem: no kappa in K_2,1}. 
Since $c\cap L$ passes over two crossings, the arcs $c\cap L$ and $c'\cap L$ are as the arcs $\alpha'$ and $\alpha_1$ in Figure \ref{fig: two twist boxes and alpha}. In particular, $c'\cap L$ passes over one crossing and therefore, $c'$ is the required curve.

\medskip

\noindent \underline{Case 2.} Now assume that $\alpha$ passes over one crossing. Repeating the argument above 
we conclude that there is a curve $c'\in\calC_{2,2}$ opposite to $c$ that crosses a twist box. 
If $\alpha'=c'\cap L$ passes over two crossings, we are back to Case 1. Otherwise, it passes 
over one crossing and we are done.
\end{proof}

Let $c$ be a curve as in the claim.  A small pushout of $c$ bounds a subdiagram of $L$ as in the definition of a twist-reduced diagram (Definition \ref{def: twist reduced}). Since 
$L$ is twist-reduced, the subdiagram must be contained in a twist box $T$. Thus, both $\alpha$ and 
$\beta$ are contained in $T$. Since $S$ passes through all bubbles in $T$, and every other 
bridge of $T$, there exists an innermost curve such that both $\alpha$ and $\beta$ pass 
through the same bubble. However, this implies that $S$ is not taut in contradiction.
\end{proof}

\begin{lemma}\label{lem: c in C4}
Every curve $c'\in\calC_{4,0}$ is of one of the forms depicted in 
Figure \ref{fig: C4}.
\end{lemma}


\begin{figure}[ht]
    \centering
   \begin{overpic}[height=3cm]{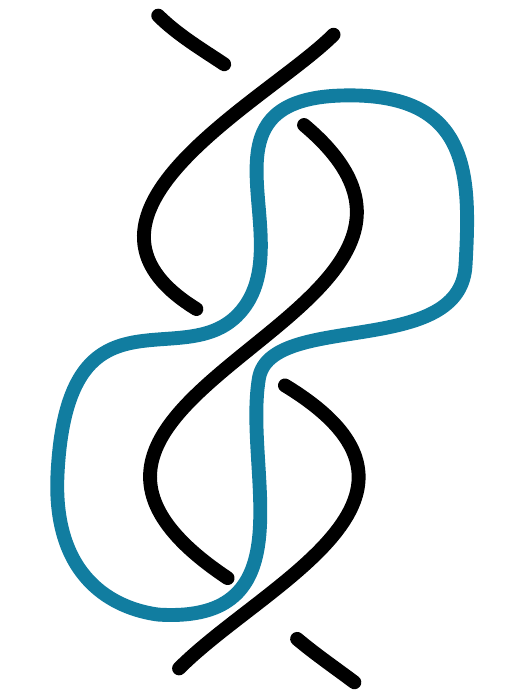}
\put(65,40){$c$}
\end{overpic}
\hskip 1cm
   \begin{overpic}[height=4cm]{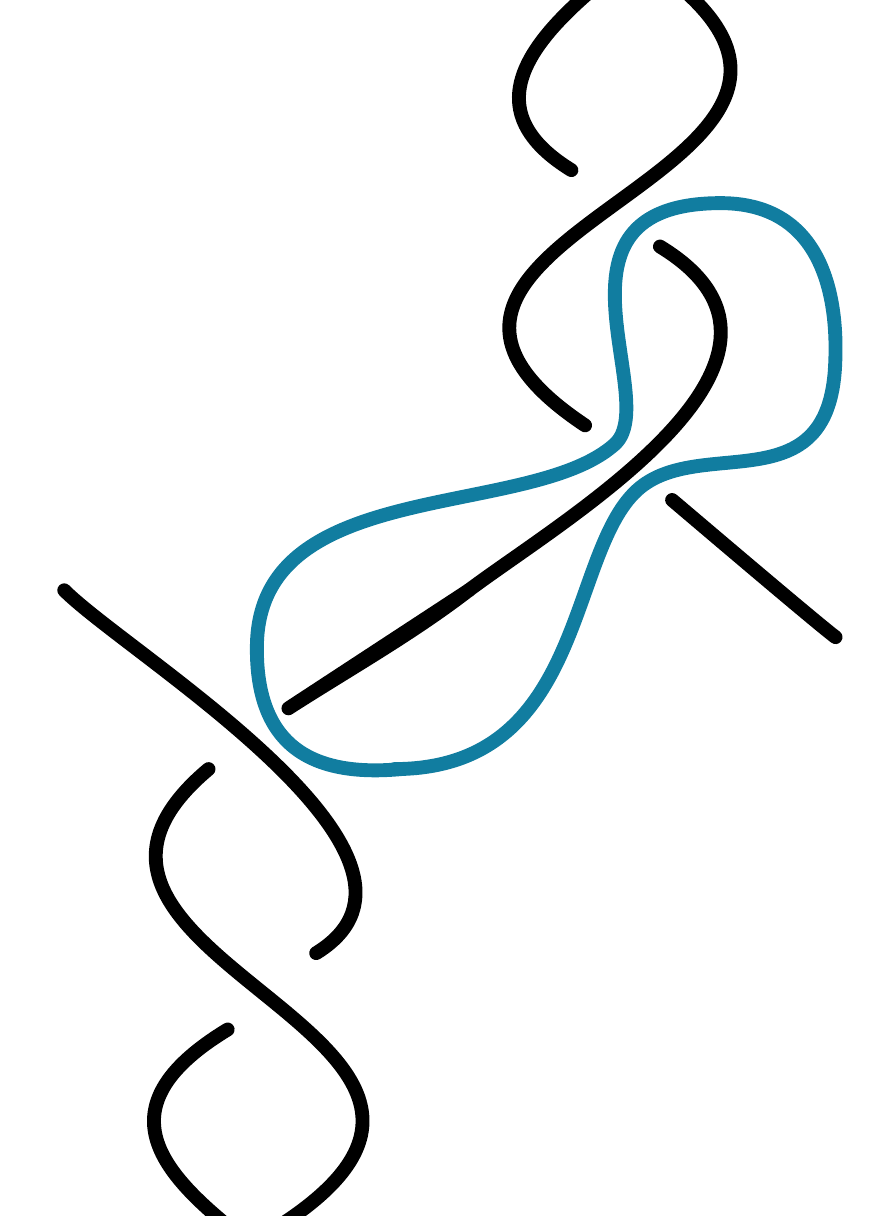}
\put(40,30){$c$}
\end{overpic}
\hskip 1cm
   \begin{overpic}[height=4cm]{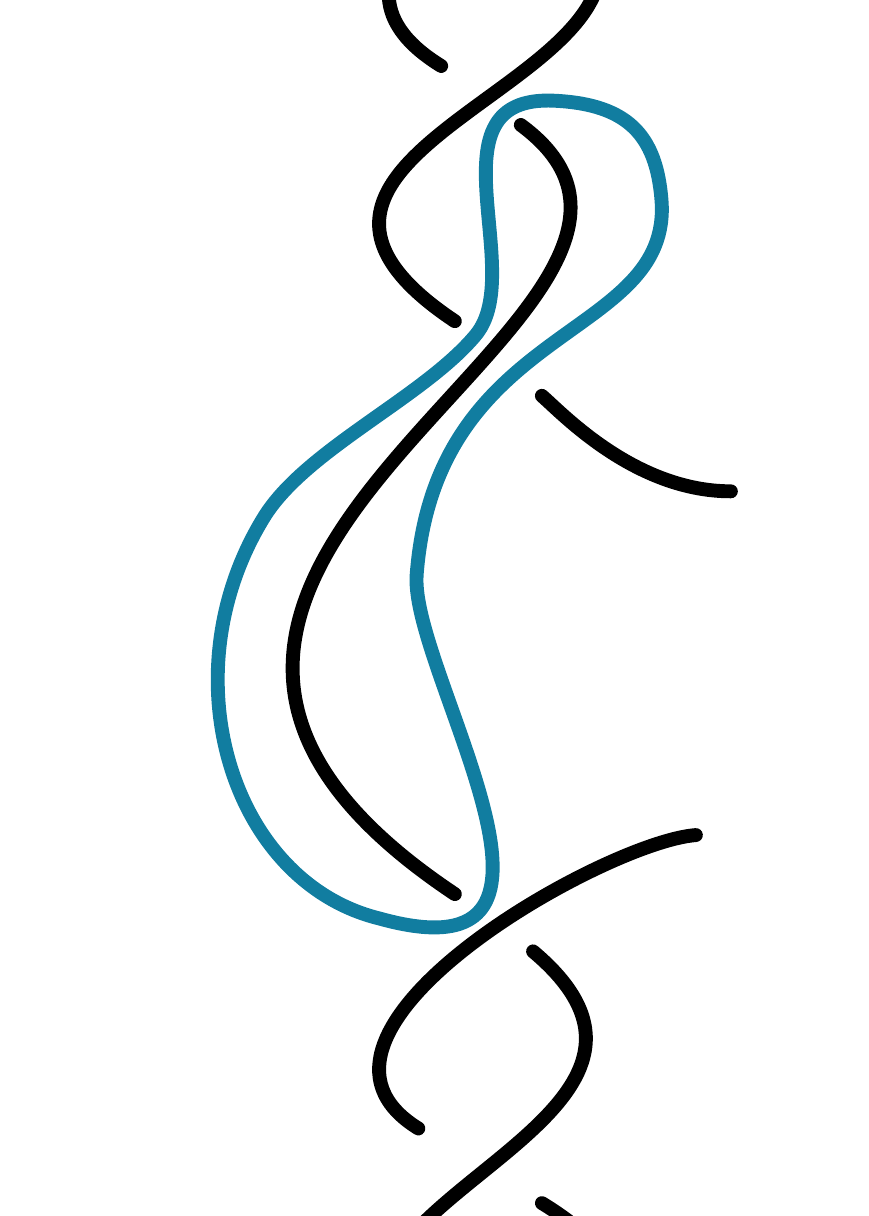}
\put(40,35){$c$}
\end{overpic}
\caption{The possible cases of a curve in $\calC_{4,0}$}  
    \label{fig: C4}
\end{figure}

\begin{proof}
Assume by contradiction that  $c'\in \calC_{4,0}$ and it is not of the form in Figure \ref{fig: C4}.
The curve $c'$ passes through four bubbles which are divided between at most four twist boxes.
Assume that $c'$ only turns at twist boxes. Consider one of the turns of $c'$ at a twist box $T$. 
Such a curve has a curve $c''$ opposite to it which crosses the twist box. The curve $c''$ must also
be in $\calC_{4,0}$ as it cannot be of any of the types formerly investigated: $\calC_{2,2}$ 
is empty and furthermore $c''$ cannot have an arc in $\calK$ since $c''$ is opposite to an arc in $\calC_{4,0}$ 
and not to $\calC_{\pos}$. 

In addition, $c''$ cannot be of one of the forms in Figure \ref{fig: C4}: Otherwise one of the arcs $c'\cap c''$ continues in $c'$ to cross a twist box, in contradiction to the assumption that $c'$ only turns.
By replacing $c'$ 
by $c''$, if need be, we assume from now on that 
\begin{enumerate}
\item[(i)] the curve $c'$ is a curve in $\calC_{4,0}$ which crosses a twist box, and
\item[(ii)] $c'$ is not of one of the forms in Figure \ref{fig: C4}.
\end{enumerate}

Let $c'$ be an innermost such curve. Let $D$ be the innermost disc bounded by $c'$. Consider the 
intersection of $D$ with the twist boxes meeting $c'$. Since the diagram is twist-reduced, 
the curve $c'$ cannot cross two different twist boxes, this leaves us with three cases:
\vskip7pt
\noindent \underline{Case 0.} The curve $c'$ passes twice through the same bubble. Then it is of the 
desired form, in contradiction to the assumption.
\vskip7pt
\noindent \underline{Case 1.} The intersection of one of these twist boxes with $D$ contains a (projection of a) bubble 
that is not met by $c'$. Let $B,B'$ be adjacent bubbles of a twist box so that $B$ is contained in $D$ 
and $c'$ passes through $B'$, see Figure \ref{fig: C4 case 1}. There must be a curve $c''$, on the same side of $P$ as $c'$, passing through the bubbles $B$ and $B'$, as follows from Figure \ref{fig: three types of intersection}. Since $c'$  is innermost, and the curve $c''$ crosses the twist box. By the previous lemmas, namely Lemmas \ref{lem: classification of chi'=0}, \ref{lem: no tilde}, \ref{lem: kappa in K_{3,0}}, \ref{lem: kappa of K_{4,0}}, and \ref{lem: no kappa in K_2,1}, it is in one of the forms shown in Figures \ref{fig:C6a}, \ref{fig:C6}  or 
\ref{fig: C4}. In all of the cases, $c''$ contains an arc outside the twist box, connecting $B$ and $B'$. 
There are two cases to consider, either $c'$ crosses the twist box, or turns at the twist box, see 
Figure \ref{fig: C4 case 1}.  If $c'$ crosses, the curve $c^*$ opposite to $c'$ (and $c''$), 
shown in  Figure \ref{fig: C4 case 1}, passes through four bubbles, it must therefore close up 
without passing through any additional bubbles, which would imply that $c'$ passes through the 
same bubble $B'$ twice, in which case we are done by case 0. 
If $c'$ turns, then the curve $c^*$, passes three times through two bubbles in the
twist box, and in  order to close up, must turn at another twist box at a bubble $B^*$, it follows 
that $c'$ passes through the bubble $B^*$ twice. In both cases, we are done by Case 0.


\begin{figure}[ht]
    \centering
   \begin{overpic}[height=4cm]{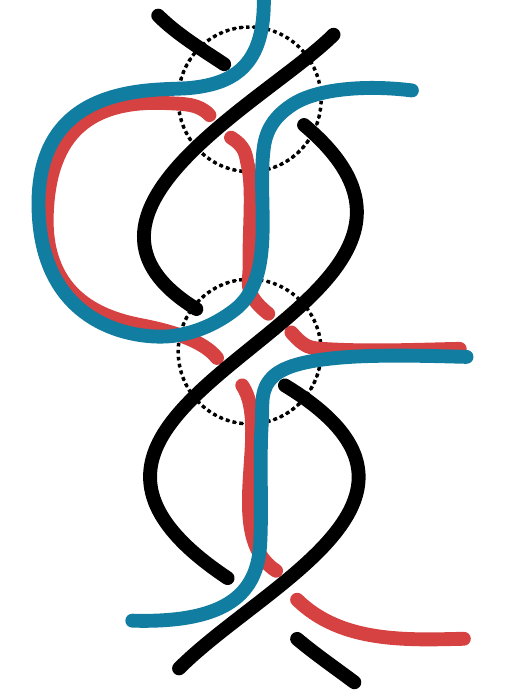}
\put(69,7){$c^*$}
\put(10,7){$c'$}
\put(63,84){$c''$}
\put(5,36){$D$}
\end{overpic}
\hskip 1cm
   \begin{overpic}[height=4cm]{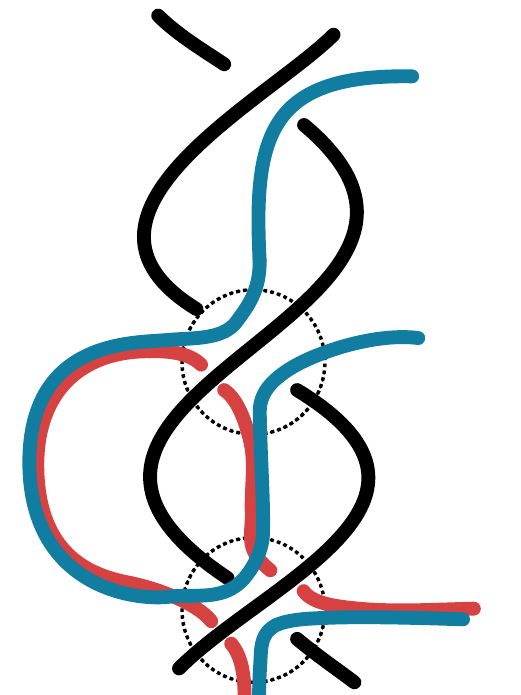}
\put(69,12){$c^*$}
\put(69,3){$c'$}
\put(63,87){$c''$}
\put(5,75){$D$}
\end{overpic}
\caption{The curve $c'$ crosses the twist box on the left, and turns on the right. The minimal disk $D$ whose boundary is $c'$ is shown in gray.}  
    \label{fig: C4 case 1}
\end{figure}
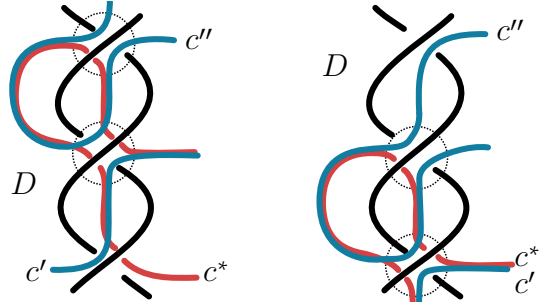

\vskip7pt

We are left with the following case:

\vskip7pt
\noindent \underline{Case 2.} The disc $D$ does not contain a bubble 
in a twist box that $c'$ meets, and $c'$ does not pass twice through the same bubble. Therefore, $c'$ 
must cross once some twist box, which we denote by $U$, and turn at two other twist boxes which we denote by $V$ and $W$ (see Figure~\ref{fig: special case of C4}. Let $B,B'$ be the two adjacent bubbles in $U$ through which $c'$ crosses. 
As we are not in Case 1, the innermost disk bounded by $c'$ contains no bubbles of $U$, $V$ or $W$. Thus, one of the bubbles $B,B'$, say $B$, is extremal in the twist box $U$, meaning its the first bubble in $U$ counted from the top or bottom. Let 
$\tau$ be the subarc of $c'$ exiting $B'$, and connecting $U$ to one of $V,W$, say $V$.  If $\tau$ enters 
$V$ from the side, then opposite to $\tau$ there is a curve $c^*$ crossing the twist boxes $U$ and $V$, which we have shown 
can not exist because $L$ is twist-reduced. Thus, $\tau$ must enter $V$ from the bottom or the top. A similar argument shows 
that the subarc $\eta$ of $c'$ connecting $V$ to $W$, must enter $W$ from the top 
or the bottom. Thus, the subarc $\mu$ of $c'$ connecting $U$ and $W$, emerges from the side of $U$ and ends in the side $W$
as depicted in Figure \ref{fig: special case of C4}.

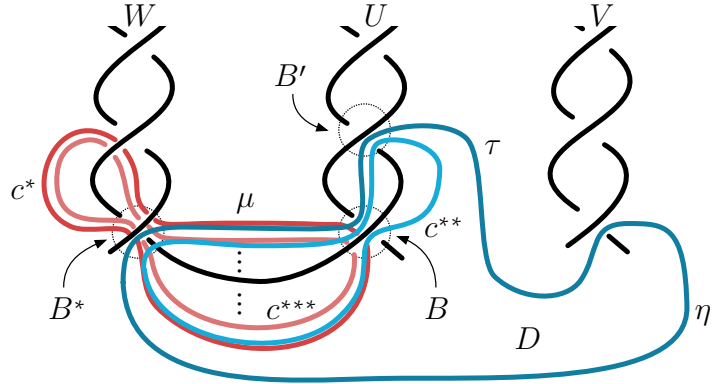
\begin{figure}[ht]
\bigskip
    \centering
   \begin{overpic}[height=5cm]{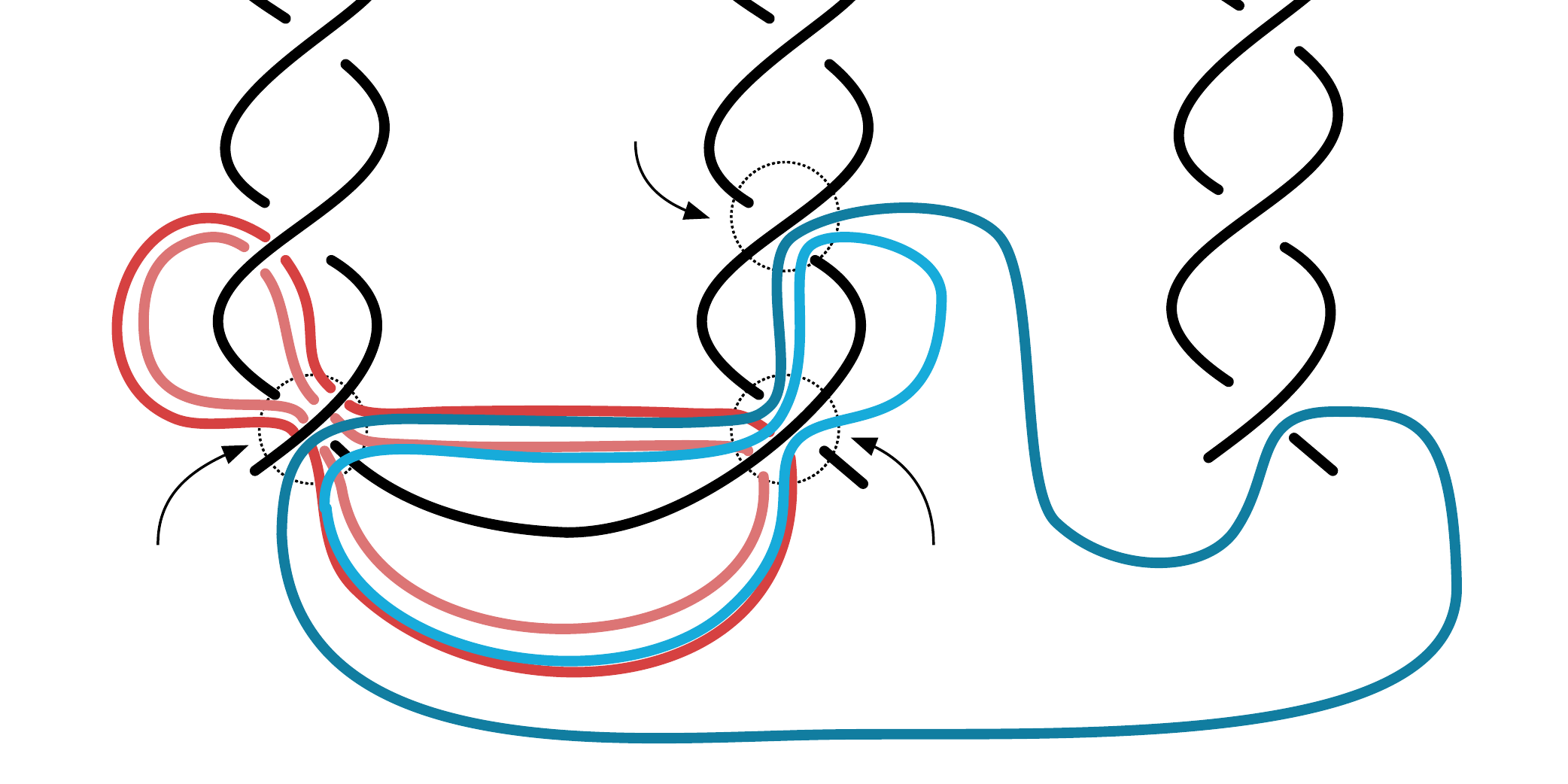}
\put(18,50){$W$}
\put(50,50){$U$}
\put(80,50){$V$}
\put(38,42){$B'$}
\put(58,11){$B$}
\put(8,11){$B^*$}
\put(70,7){$D$}
\put(66,33){$\tau$}
\put(94,11){$\eta$}
\put(33,26){$\mu$}
\put(58,22){$c^{**}$}
\put(3,27){$c^{*}$}
\put(37,11){$c^{***}$}
\put(33,17.5){$\vdots$}
\put(33,12){$\vdots$}
\end{overpic}
\caption{The only possible configuration in Case 2} 
     \label{fig: special case of C4}
\end{figure}


Let $c^*$ be the curve opposite to $c'$ along $\mu$. The curve $c^*$ passes through three bubbles, one in $U$ and 
two in $W$, and must close after an additional turn. By the checkerboard coloring, the additional turn must 
be at the bubble $B^*$ at the bottom of $W$, as shown in Figure \ref{fig: special case of C4}.
Opposite to $c^*$ at $B^*$, there is another curve $c^{**}$ which passes through $B^*$ parallel to $c'$.
By the previous lemmas, as in case 1, $c^{**}$ in $\calC_{4,0}$.  If it is in $\calC_{4,0}$, one can check that 
in order to close up $c^{**}$ must cross $U$ parallel to $c'$. Iterating this argument gives an infinite collection of nested curves intersecting 
$B$ and $B^*$, which is a contradiction.
This finishes the proof of the lemma.
\end{proof}

\vskip15pt
\section{No tori  and no annuli}
\vskip8pt
In this section we prove Theorem \ref{thm: boundary parallel} and Theorem 
\ref{thm: highly twisted plats are hyperbolic}

\vskip5pt

\begin{lemma}\label{lem: desired implies prallel}
 If $S$ is a taut torus or annulus and if all the curves in $\calC_{\le 0}$ are of the form depicted in Figures \ref{fig:C6a}, \ref{fig:C6} and \ref{fig: C4}  
 then $S$ is a boundary parallel torus.
\end{lemma}

\begin{proof} 
Let $S$ intersect $P^{\pm}$ only in curves as in Figures \ref{fig:C6a}, \ref{fig:C6} and \ref{fig: C4}. 
The surface $S$ is obtained by capping each curve in $P^\pm$ by the disk it bounds in $H^\pm$ 
 respectively. Thus, $S$ is a torus following one of the components 
 of $L$. That is, $S$ is boundary parallel.
\end{proof}

\begin{corollary}\label{cor: atoroidal} The link complement $S^3 \ssm \NN(L$) does not
contain essential tori. In particular, the link $L$ is prime.
\end{corollary}

\begin{proof} 
Let $S$ be an essential torus in $S^3 \ssm \NN(L)$. 
By Lemma \ref{lem: properties of curves}, we may assume that $S$ is taut. 
By Lemmas \ref{lem: no tilde}, \ref{lem: kappa of K_{4,0}}, \ref{lem: kappa in K_{3,0}}, and 
\ref{lem: c in C4}, all the curves in $\calC$ are of the form depicted in Figures \ref{fig:C6a}, 
\ref{fig:C6} and \ref{fig: C4}.  Hence, by Lemma \ref{lem: desired implies prallel}, the torus 
$S$ is boundary parallel which contradicts the choice of $S$.
\end{proof}

\begin{proposition}\label{pro: unannular}
The link complement $S^3 \ssm \NN(L$) does not
contain essential annuli.
\end{proposition}

\begin{proof} Let $S$ be an incompressible and boundary incompressible annulus.
Since by Corollary \ref{cor: atoroidal} $L$ is prime, we may assume that $\partial S$ 
has no meridional component.
By Lemma \ref{lem: properties of curves}, we may assume that $S$ is taut.

Assume first that $S$ passes through all twist boxes in disks of Type (0) or (2) only, 
as in Lemma \ref{lem: normal form}. Therefore, as $S$ is an annulus, 
there must be a twist box $T$ such that $S$ meets $T$ in a Type (2) disk. When $S$ emerges from $T$ it 
intersects $P$ in a curve $c\in\calC$. By Lemmas \ref{lem: no kappa in K_2,1} and \ref{lem: No C22} 
such a curve must be in $\calC_{0,4}$. The curve $c$ cannot meet a twist box as otherwise this 
twist box will meet $S$ in a disk of Type (1). Consider the disk $D$ in $P$ bounded by $c$. Either 
$\interior{D} \cap L = \emptyset$ in which case $S$ is not taut, or the diagram $D(L)$ is not 
prime,  a contradiction.

Therefore, there is a twist box $T$ which intersects $S$ in a disk of Type (1).
Since $\partial S$ contains a string of $T$, there is a curve $c$ in $\calC$ which passes through 
the ``middle'' of the twist box $T$ -- i.e., the intersection $c\cap L$ contains a bridge which does 
not meet the top and bottom of $T$. Let $\alpha$ be a small continuation  of this bridge 
in $c$.  Lemmas \ref{lem: no kappa in K_2,1} and \ref{lem: No C22}  rule out curves containing 
$\calK_{2,1}$ and curves in $\calC_{2,2}$. Thus, $c$ must be in $\calC_{0,4}$, and must 
contain another bridge.
Let $\beta$ be a small continuation of that bridge in $c$. 
The arcs $\alpha$ and $\beta$ belong to different components of $\partial S$: Otherwise, the arc on $c$
connecting $\alpha$ and $\beta$ together with $\partial S$ bounds a disk in $S$. An innermost curve 
$c\in\calC$ in this disk must have only two intersection points with $L$, which is a contradiction to 
Lemma \ref{lem: properties of curves}. The endpoints of $\alpha$ in $c$ belong to regions of the same 
color, hence the same holds for $\beta$. Note that a bridge in a plat passes at most two crossings. 
We divide the proof into cases depending on the number of crossings $\beta$ passes over.

\vskip7pt

\noindent\underline{Case 0:} If $\beta$ does not pass over any crossing. Then, as the endpoints of $\beta$ belong to regions of the same color, they must belong to the same region, contradicting the assumption that  $S$ is taut.

\vskip7pt

\noindent\underline{Case 1:} If $\beta$ passes over one crossing. A small pushout of the curve 
$c$ bounds a subdiagram of $L$ containing the two crossings points over which $\alpha$ and $\beta$ 
pass. Since we assumed that $L$ is twist-reduced these two crossing points must belong to the same 
twist box $T$. Let $n$ be the number of bridges of $T$ in-between $\alpha$ and $\beta$. We further 
divide into sub-cases according to $n$.
\vskip7pt
\underline{Sub-case 1.0:} $n=0$. First, if $\alpha$ and $\beta$ meet the same bridge of $L$ 
then $c$  bounds a boundary compression disk for $S$, which is a contradiction. So, assume that $\alpha$ 
and $\beta$ meet adjacent bridges of $S$. The annulus $S$ must thus spiral in-between the strands of 
$L\cap T$.  Thus we obtain a disk of Type 2 (as in Lemma \ref{lem: normal form}) and hence, by the definition of $\calC^+$ and $\calC^-$
as in the beginning of \S\ref{subsec: Curves of intersection} 
this curve does not appear in $\calC$.

\vskip7pt

 \underline{Sub-case 1.1:} $n=1$. The tangle $L\cap T$ has two components $\lambda_1,\lambda_2$. 
Let $l_1,l_2$ be the corresponding components of $L$ (possibly $l_1=l_2$). Because $n = 1$ the arcs 
$\alpha$ and $\beta$ meet the same string of $L\cap T$, say $\lambda_1$. Hence the two boundary components 
of $S$ are contained in the same component $l_1$ of $L$. If $l_1=l_2$, then there exists a curve 
$c'\in\calC_{0,4}$ which meets the bridge in-between $\alpha$ and $\beta$, and this curve must be as 
in Case 1.0. Thus, we may assume that $l_1\ne l_2$. Consider the disk $\Delta$ as depicted in
Figure \ref{fig: C04 proof}(a). Its interior intersects $L$ in a single point in $l_2$, 
and its boundary is the union of an arc on $S$ and an arc on $l_1$. The manifold $\NN(S)\cup \NN(l_1)$ 
has two torus boundary components $U$ and $V$. See Figure \ref{fig: C04 proof}(b). Let $U$ 
be the torus that meets $\Delta$. Let $U_-$ be the component of $S^3\ssm U$ containing $l_2$, and 
let $U_+$ be the other component. The torus $U$ is incompressible in $U_-$, as such a compression 
must be on $\Delta$ and $\Delta$ meets $l_2$ once. It is also incompressible in $U_+$, as if
such a compression disk exists then since it cannot intersect $l_1$, it gives a compression of 
the annulus $S$, in contradiction to the incompressibility of $S$. 

\begin{figure}[ht]
\centering

\subfigure[]{
   \begin{overpic}[height=4cm]{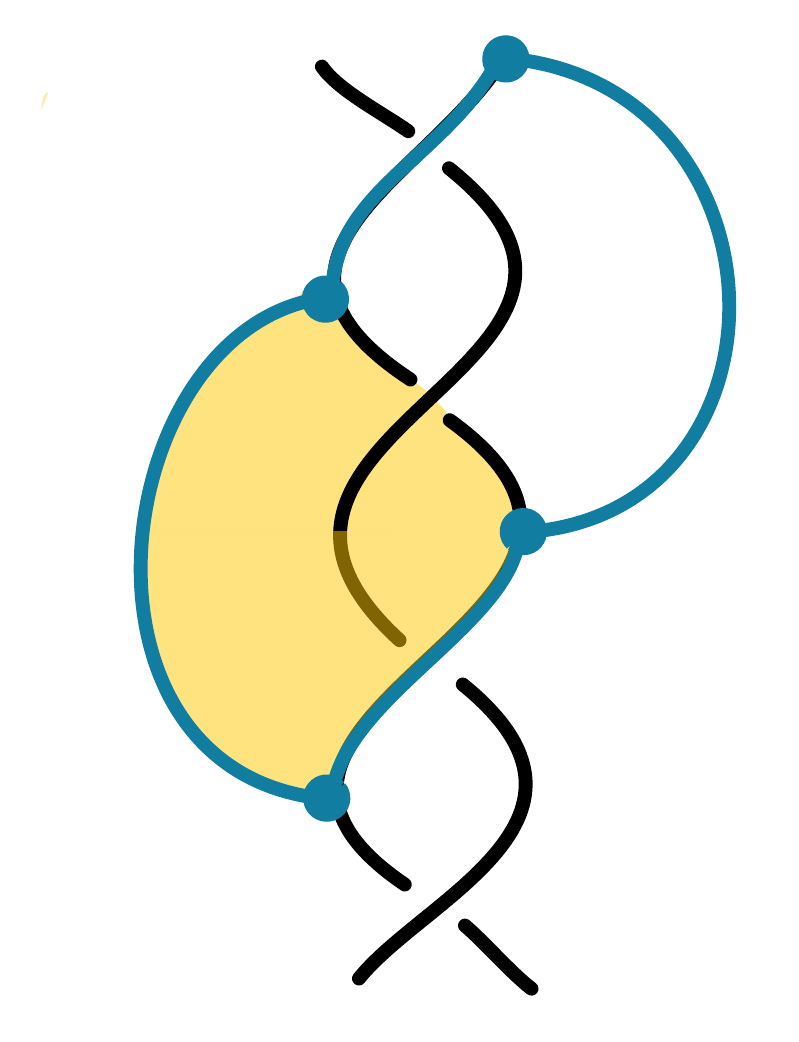}
\put(17,40){$\Delta$}
\put(30,-3){$l_2$}
\put(50,-3){$l_1$}
\put(38,30){\Tiny$\beta$}
\put(38,80){\Tiny$\alpha$}
\end{overpic}
}
\hspace{2cm}
\subfigure[]{
   \begin{overpic}[height=5cm]{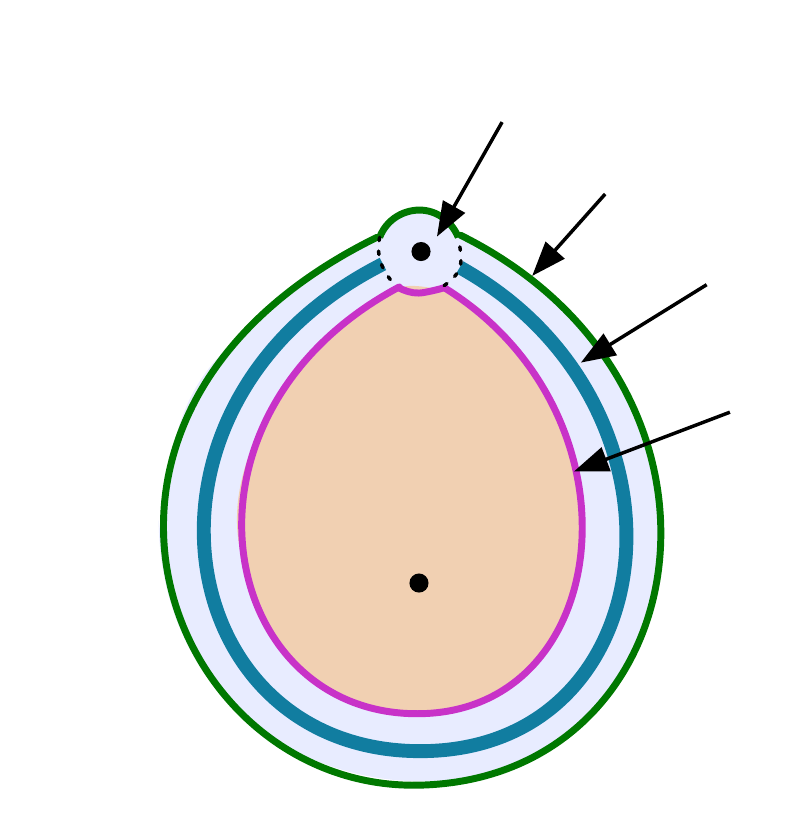}
\put(60,88){$l_1$}
\put(73,75){$V$}
\put(85,65){$S$}
\put(90,50){$U$}
\put(42,30){$l_2$}
\put(45,45){$U^-$}
\end{overpic}
}

\caption{(a) The disk $\Delta$. (b) A cross section of the twist box, the annulus $S$ and the tori $U,V$}  
    \label{fig: C04 proof}
\end{figure}
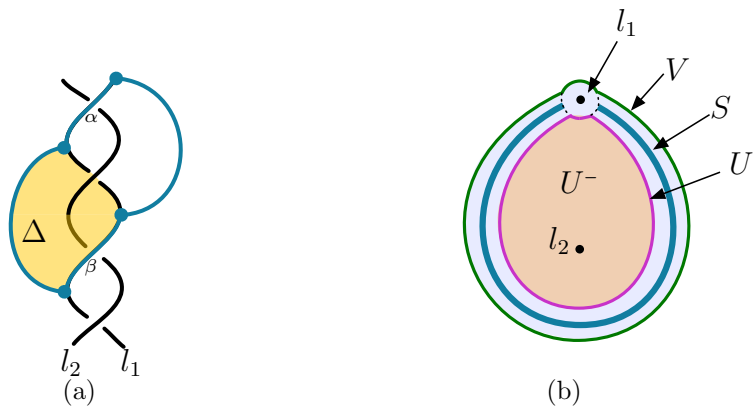




By Corollary \ref{cor: atoroidal}, $U$ must be boundary parallel to either $\partial \NN (l_2)$ or 
$\partial \NN (l_1)$. If $U$ is parallel to $\partial \NN (l_2)$, then since $l_1$ is parallel to 
a curve in $U$ crossing $\Delta$ once, then there exists an annulus $A\subset S^3 \ssm L$ whose 
boundary is $l_1\cup l_2$.
The annulus $A$ is incompressible, since otherwise $l_1 \cup l_2$ would be a 2-component unlink that 
is not linked with $L$, i.e., $L$ is split, contradicting Corollary \ref{cor: nonsplit}. The annulus $A$ is trivially boundary-incompressible because the boundary components of $A$ are on two different components of $L$. If we run the argument for $A$ instead of $S$, Case 1.1 cannot occur because the boundary components of $A$ are on two different components of $L$.

If $U$ is parallel to $\partial \NN (l_1)$, the intersection $\Delta \cap U$ is a curve on $U$ which 
meets the meridian of $\NN(l_1)$ exactly once: as  if it meets it more than once, then the union 
$\NN(\Delta) \cup \NN(l_1)$ determines a once-punctured non-trivial lens space contained in $S^3$, which 
is impossible. Thus, $\partial \Delta$ which is parallel to $\Delta\cap U$ is also parallel to $l_1$. 
Therefore, the arcs $\partial \Delta \ssm l_1 \subset S$ and $l_1 \ssm \partial \Delta \subset L$
bound a disk. Since the arc $\partial \Delta \ssm l_1$ connects different components of $S$ it is 
an essential arc, and the disk is a boundary compression for $S$, a contradiction.

\vskip7pt
\underline{Sub-case 1.2:} $n\ge 2$.
As the boundary of the annulus $S$ must pass through every other bridge in $T$, there must be 
another curve of $\calC_{0,4}$ in between $\alpha$ and $\beta$. By choosing an innermost such 
curve we are back in one of the previous cases.

\vskip10pt
\noindent\underline{Case 2:} If $\beta$ passes over two crossings, then since the endpoints of 
$\beta$ have the same color they are at distance 2. Therefore, we are in Case 
\eqref{obs: single arc. two bubbles and distance 2} 
of Observation \ref{obs: single arc}. 
By the uniqueness of such arcs, $\alpha = \beta$. However, since $\alpha$ is a bridge in the 
``middle'' of a twist box, it has only one crossing, a contradiction.
\end{proof}


Theorem \ref{thm: boundary parallel} immediately follows from Corollary \ref{cor: atoroidal} 
and Proposition \ref{pro: unannular}.

\begin{proof}[Proof of Theorem \ref{thm: highly twisted plats are hyperbolic}]
If $m = 2$ then $L$ is a $2$-bridge knot/link which is not a torus knot/link, since there are at 
least two twist regions,  hence it is atoroidal and unannular by \cite{Hatcher-Thurston}. If 
$m \geq 3$ then the theorem follows directly from Theorem \ref{thm: boundary parallel}.  In both 
case apply  Thurston's result (see \cite{thurston}) that a link complement which contains no 
essential annuli or tori is hyperbolic.
\end{proof}

\bibliographystyle{amsplain}
\bibliography{biblio}

\end{document}